\documentclass[reqno,11pt]{amsart}
\usepackage{setspace,tikz,xcolor,mathrsfs,listings,multicol,amssymb}
\usepackage{rotating}
\usepackage[vcentermath]{youngtab}
\usepackage[margin=1in,includefoot,footskip=30pt]{geometry}
\usepackage{enumerate}
\usepackage{booktabs}
\usepackage[centertableaux]{ytableau}
\usepackage[all,cmtip]{xy}
\usetikzlibrary{arrows,matrix}
\tikzset{tab/.style={matrix of math nodes,column sep=-.35, row sep=-.35,text height=7pt,text width=7pt,align=center,inner sep=2,font=\footnotesize}}

\usepackage[colorlinks=true, pdfstartview=FitV, linkcolor=blue, citecolor=blue, urlcolor=blue]{hyperref}

\newcommand{\inner}[2]{\left\langle #1, #2 \right\rangle}

\newcommand{\Gr}{\operatorname{Gr}}  
\newcommand{\schur}{\kappa} 

\newcommand{\abs}[1]{\left\lvert #1 \right\rvert}

\newcommand{\bra}[1]{\langle #1 \rvert}
\newcommand{\ket}[1]{\lvert #1 \rangle}
\newcommand{\braket}[2]{\langle #1 | #2 \rangle}
\newcommand{\prob}{\mathsf{P}}
\newcommand{\ds}{/\!\!/}  
\newcommand{\bbb}{\mathsf{b}}
\newcommand{\field}{\mathbf{k}}

\newcommand{\G}{G}
\newcommand{\dG}{g}
\newcommand{\wG}{J}
\newcommand{\dwG}{j}

\DeclareMathOperator{\wt}{wt} 

\newcommand{\mcA}{\mathcal{A}}
\newcommand{\mcB}{\mathcal{B}}
\newcommand{\mcC}{\mathcal{C}}

\newcommand{\mcF}{\mathcal{F}}
\newcommand{\mcG}{\mathcal{G}}

\newcommand{\mcP}{\mathcal{P}}
\newcommand{\mcT}{\mathcal{T}}

\newcommand{\fG}{\mathfrak{G}}

\newcommand{\ta}{\mathsf{a}}
\newcommand{\tb}{\mathsf{b}}
\newcommand{\tc}{\mathsf{c}}

\newcommand{\ii}{\mathbf{i}}

\newcommand{\bt}{\mathbf{t}}
\newcommand{\uu}{\mathbf{u}}

\newcommand{\xx}{\mathbf{x}}
\newcommand{\yy}{\mathbf{y}}

\newcommand{\UU}{\mathbf{U}}
\newcommand{\bG}{\mathbf{G}}

\newcommand{\bal}{\boldsymbol{\alpha}}
\newcommand{\bbe}{\boldsymbol{\beta}}
\newcommand{\bpi}{\boldsymbol{\pi}}
\newcommand{\brho}{\boldsymbol{\rho}}
\newcommand{\bnu}{\boldsymbol{\nu}}

\newcommand{\ZZ}{\mathbb{Z}}
\newcommand{\QQ}{\mathbb{Q}}

\newcommand{\CC}{\mathbb{C}}

\definecolor{darkred}{rgb}{0.7,0,0} 
\newcommand{\defn}[1]{{\color{darkred}\emph{#1}}} 

\definecolor{UQgold}{RGB}{196, 158, 54} 
\definecolor{UQpurple}{RGB}{73, 7, 94} 
\definecolor{UMNgold}{RGB}{255,200,46} 
\definecolor{UMNmaroon}{RGB}{106,0,50} 
\definecolor{OCUenji}{RGB}{153,0,51} 
\definecolor{OCUsapphire}{RGB}{0,51,102} 
\definecolor{TUblue}{RGB}{0,77,255} 

\usepackage{listings}
\lstdefinelanguage{Sage}[]{Python}
{morekeywords={False,sage,True},sensitive=true}
\definecolor{dblackcolor}{rgb}{0.0,0.0,0.0}
\definecolor{dbluecolor}{rgb}{0.01,0.02,0.7}
\definecolor{dgreencolor}{rgb}{0.2,0.4,0.0}
\definecolor{dgraycolor}{rgb}{0.30,0.3,0.30}

\theoremstyle{plain}
\newtheorem{thm}{Theorem}[section]
\newtheorem{lemma}[thm]{Lemma}

\newtheorem{prop}[thm]{Proposition}
\newtheorem{cor}[thm]{Corollary}
\theoremstyle{definition}
\newtheorem{dfn}[thm]{Definition}
\newtheorem{ex}[thm]{Example}
\newtheorem{remark}[thm]{Remark}

\numberwithin{equation}{section}

\usepackage[colorinlistoftodos]{todonotes}

\begin{document}
\title[Probability and Schubert Kalculus]{Free fermionic probability theory and K-theoretic Schubert calculus}

\author[S.~Iwao]{Shinsuke Iwao}
\address[S.~Iwao]{Faculty of Business and Commerce, Keio University, Hiyosi 4--1--1, Kohoku-ku, Yokohama-si, Kanagawa 223-8521, Japan}
\email{iwao-s@keio.jp}

\author[K.~Motegi]{Kohei Motegi}
\address[K.~Motegi]{Faculty of Marine Technology, Tokyo University of Marine Science and Technology, Etchujima 2--1--6, Koto-Ku, Tokyo 135-8533, Japan}
\email{kmoteg0@kaiyodai.ac.jp}
\urladdr{https://sites.google.com/site/motegikohei/home}

\author[T.~Scrimshaw]{Travis Scrimshaw}
\address[T.~Scrimshaw]{Department of Mathematics, Hokkaido University, 5 Ch\=ome Kita 8 J\=onishi, Kita Ward, Sapporo, Hokkaid\=o 060-0808, Japan}
\email{tcscrims@gmail.com}
\urladdr{https://tscrim.github.io/}


\thanks{
S.I.~was partially supported by Grant-in-Aid for Scientific Research (C) 19K03605, 22K03239, 23K03056.
K.M.~was partially supported by Grant-in-Aid for Scientific Research (C) 21K03176, 20K03793.
T.S.~was partially supported by Grant-in-Aid for JSPS Fellows 21F51028 and for Scientific Research for Early-Career Scientists 23K12983.
}

\begin{abstract}
For each of the four particle processes given by Dieker and Warren, we show the $n$-step transition kernels are given by the (dual) (weak) refined symmetric Grothendieck functions up to a simple overall factor.
We do so by encoding the particle dynamics as the basis of free fermions first introduced by the first author, which we translate into deformed Schur operators acting on partitions.
We provide a direct combinatorial proof of this relationship in each case, where the defining tableaux naturally describe the particle motions.
\end{abstract}

\maketitle
\setcounter{tocdepth}{1}  
\tableofcontents

\section{Introduction}
\label{sec:introduction}

An asymmetric simple exclusion process (ASEP) is a probabilistic model for particles on a lattice, typically one dimensional, domain such that each position can be occupied by at most one particle.
As such, it has been used as a simple model for a diverse range of natural processes, such as in transportation through microscopic channels~\cite{CL99}, vehicle traffic moving in a single lane~\cite{CSS00}, or the dynamics of ribosomes along RNA~\cite{MGP68} (the earliest known publication as far as the authors are aware).
The study of such particle systems is an active area of research, with some recent mathematical articles being~\cite{AGLS23,AMM23,AMM22,AN22,BB21,BLSZ22,CMP19,CMW22,CZ22,DW23,PS22,QS23}. 

We will focus on the case when the particles only move in one direction (here, to the right) on a $\ZZ$ lattice.
This is known as a totally asymmetric simple exclusion process (TASEP) on a line.
Given that TASEP can be interpreted as a model for electrons moving down a wire, in this introduction we will be representing states using the description of free fermions.
The question becomes how to encode the dynamics of the TASEP considered in terms of operators acting on the free fermions.

The versions of TASEP we will focus on are the four variations that were studied by Dieker and Warren~\cite{DW08}, where the particles will all lie on $\ZZ$ starting from the step initial condition, where the $j$-th particle starts at site $-j$ for all $j \geq 1$, and move in discrete time.
(In~\cite{DW08}, they used a ``bosonic'' formulation that can easily be translated into the fermionic description we use in the introduction; see Section~\ref{sec:particle_processes} for a precise relationship.)
These TASEP variations have been studied before by various authors and sometimes using different models; we refer the reader to~\cite{DW08} for further connections and references.
The particles will stay in order, so we can identify a partition $\lambda$ with the positions of particles by having the $j$-th particle be at position $\lambda_j - j$.
All four variations are based on random matrices $[w_{ji}]_{i,j}$ with $w_{ji}$ either being Bernoulli or geometric random variables and taking either a (zero temperature) first or last passage percolation model.
Translating this to the motion of the particles, the $w_{ji}$ specifies how many steps the $j$-th particle wants to move at time $i$, and the first (resp.\ last) passage percolation corresponds to the particles either being blocked by smaller particles (resp.\ pushing the smaller particles).
(See Section~\ref{sec:particle_processes} for a precise description.)
In order to encode the dynamics using free fermions, we will show the transition probabilities can be described using symmetric functions coming from the K-theory of a classical algebraic variety, the Grassmannian.

In more detail, the Grassmannian $\Gr(k, n)$ is the set of $k$-dimensional subspaces of $\CC^n$.
Next, this has a natural action of the group of invertable upper triangular $n \times n$ matrices $B$, which acts with finitely many orbits on $\Gr(k, n)$ indexed by partitions $\lambda$ inside a $k \times (n-k)$ rectangle.
The closures of these orbits (under the Zariski topology) are known as Schubert varieties, and they give a CW decomposition of $\Gr(k, n)$.
Hence, they give rise to a basis for the cohomology ring $H^{\bullet}(\Gr(k, n), \ZZ)$, where under Borel's isomorphism~\cite{Borel53} the cohomology class indexed by $\lambda$ corresponds to the Schur function $s_{\lambda}(\xx)$.
This construction be extended to the (connective) K-theory ring of $\Gr(k, n)$ by using the Bott--Samelson resolution of Schubert varieties, where now the K-theory class indexed by $\lambda$ corresponds~\cite{LS82,LS83} to the (symmetric\footnote{We henceforth drop the word ``symmetric'' for simplicity as we will not consider the ``nonsymmetric'' Grothendieck polynomials coming from the K-theory of the complete flag variety, which are analogous to Schubert polynomials.}) Grothendieck function $\G_{\lambda}(\xx; \beta)$.
When working with symmetric functions, it is natural from representation theory to consider the Schur functions as an orthonormal basis as they are characters of the irreducible representations of the Lie group of all invertible $n \times n$ matrices.
Therefore, we can take the dual basis $\{\dG_{\lambda}(\xx; \beta)\}_{\lambda}$ to $\{\G_{\lambda}(\xx; \beta)\}_{\lambda}$ in (a completion of) the ring of symmetric functions over all partitions $\lambda$.
Additionally, there is an algebra involution $\omega$ defined by $\omega s_{\lambda} = s_{\lambda'}$, where $\lambda'$ is the conjugate partition of $\lambda$. that defines the ``weak'' versions $\wG_{\lambda}(\xx; \beta) := \omega \G_{\lambda}(\xx; \beta)$ and $\dwG_{\lambda}(\xx; \beta) = \omega \dG_{\lambda}(\xx; \beta)$.

These symmetric functions have combinatorial descriptions using variations of the classical semistandard tableaux description of $s_{\lambda}(\xx)$ (see, \textit{e.g.},~\cite{ECII}).
The Schur decomposition of $G_{\lambda}(\xx; \beta)$ was shown by Lenart~\cite{Lenart00}, and expressing $\G_{\lambda}(\xx; \beta)$ as a generating function of set-valued tableaux was given by Buch~\cite{Buch02}.
Lam and Pylyavskyy subsequently gave~\cite[Thm.~9.15, Prop.~9.22]{LamPyl07} a combinatorial interpretation of $\wG_{\lambda}(\xx; \beta)$, $\dG_{\lambda}(\xx; \beta)$, and $\dwG_{\lambda}(\xx; \beta)$ as multiset-valued tableaux, reverse plane partitions, and valued-set tableaux, respectively.
Later work of Galashin, Grinberg, and Liu~\cite{GGL16} then refined the parameter $\beta$ into a family $\bbe$ for $\dG_{\lambda}(\xx; \bbe)$, and the dual version $\G_{\lambda}(\xx; \bbe)$ was introduced by Chan and Pflueger~\cite{CP21}.
In a different direction, Yeliussizov~\cite{Yel17} introduced a combination of $\G_{\lambda}(\xx; \beta)$ and $\wG_{\lambda}(\xx; \beta)$ as the canonical Grothendieck polynomials and their duals (along with a combinatorial interpretation), and the refined version $\G_{\lambda}(\xx; \bal, \bbe)$ and the duals $\dG_{\lambda}(\xx; \bal, \bbe)$ were introduced by Hwang \textit{et al.}~\cite{HJKSS24,HJKSS25} (also with a combinatorial formulas).
The Schur function decomposition allows these combinatorial objects to be deconstructed into pairs of tableaux by RSK-type algorithms~\cite{LamPyl07,HS20,PPPS20}.

The first connection between the TASEP models and the K-theoretic Schubert calculus was noted by Yeliussizov~\cite{Yel20} when $\xx = \bbe = \sqrt{q}$, where he showed that the geometric last passage percolation --- Case~A in~\cite{DW08} --- probabilities equaled a dual Grothendieck polynomial up to an overall simple factor.
This was later generalized to $w_{ji}$ having its parameter $\pi_j x_i$ in~\cite{MS20}.
We can also connect this to a more classical version of TASEP in discrete time on $\ZZ$ starting from the step initial condition, where the $i$-th particle moving to the right with probability $x_i$ if the site is free, and $w_{ji}$ records how long waits from it could move (see, \textit{e.g.},~\cite{Johansson00} or~\cite[App.~A]{MS20}).
Taking the obvious time-dependent refinement of~\cite{DW08}, the determinant formula is readily seen to be the Jacobi--Trudi formula for the natural skew version $\dG_{\lambda/\mu}(\xx; \bbe)$~\cite{AY20,Kim20,Kim20II,HJKSS24,HJKSS25}.
By applying the $\omega$ involution (see also specialized versions of the Jacobi--Trudi formulas of~\cite{HJKSS24,HJKSS25}), we also obtain the Bernoulli last passage percolation~\cite[Case~D]{DW08}.

However, we are interested in trying to understand  the relationship at the level of local dynamics; in particular, to address the question of describing the TASEP dynamics using free fermions.
To make this precise, we want to describe the $n$-step transition kernel $\prob_{X,n}(\lambda|\mu)$ in Case~X (for $X = A,B,C,D$ as given in~\cite{DW08}) from the particles starting at positions $\mu$ and ending at positions $\lambda$.
Building upon the work of the first author~\cite{Iwao19,Iwao21,Iwao20}, our previous work~\cite{IMS22} introduced a free fermion description of $\G_{\lambda\ds\mu}(\xx; \bal, \bbe)$ and $\dG_{\lambda/\mu}(\xx; \bal, \bbe)$.
This lead to a Jacobi--Trudi formula~\cite[Thm.~4.1]{IMS22} for $\G_{\lambda\ds\mu}(\xx; \bal, \bbe)$, which then appropriate specializations give formulas (up to an overall simple factor) for the transition probabilities of the first passage percolation cases~\cite[Case~B,C]{DW08}.
Likewise, the generalized Schur operators in~\cite{Iwao19}, which are acting on partitions and come from the current operators, behave exactly like the TASEP pushing and blocking dynamics.
Therefore, our main goal in this paper is to introduce refined versions of these operators from~\cite{Iwao19} and show 
the correspondence between transition probabilities of four types of discrete time TASEP and refined (dual) Grothendieck polynomials.
To state the correspondence,  here we introduce the four types of TASEP.
Let the Weyl chamber $\Omega_\ell$ be
\begin{align}
\Omega_\ell=\{ (z_1,z_2,\dots,z_\ell) \in \mathbb{Z}^\ell: z_1>z_2>\cdots>z_\ell \},
\end{align}
for $\ell \ge 2$.

We introduce four types of evolution of particles $X_t^{\mathrm{A}}$, $X_t^{\mathrm{B}}$, $X_t^{\mathrm{C}}$ and $X_t^{\mathrm{D}}$ on $\Omega_\ell$.
Each process corresponds to case $A,B,C$ and $D$ of the TASEP version of the model by Dieker--Warren \cite{DW08}.
We consider the discrete time evolution from time zero to time $n$ with time step one.
\begin{subequations}
\label{eq:particle_transitions}

\medskip
\noindent \textbf{(A)} Geometric distribution with pushing behavior:

From time $t$ to time $t+1$,
the evolution of the particle system $X_t^\mathrm{A} \in \Omega_\ell$ is defined as
\begin{align}
X_{t+1}^\mathrm{A}(k)=\max(X_t^\mathrm{A}(k),X_{t+1}^\mathrm{A}(k+1)+1)+\xi^\mathrm{A}(k,t+1),
\end{align}
for $k=1,\dots,\ell-1$ and $X_{t+1}^\mathrm{A}(\ell)=X_t^\mathrm{A}(\ell)+\xi^\mathrm{A}(\ell,t+1)$.
$\xi^\mathrm{A}(k,t)$ are random variables satisfying
$\mathsf{P}(\xi^\mathrm{A}(k,t)=r)=(1-\pi_k x_t) (\pi_k x_t)^{r}$, $r \in \mathbb{Z}_{\ge 0}$, where $\pi_k, x_t$ are real numbers satisfying $0 < \pi_k x_t < 1$ for all $k,t$.

\medskip 
\noindent \textbf{(B)} Bernoulli distribution with blocking behavior:

From time $t$ to time $t+1$,
the evolution of the particle system $X_t^\mathrm{B} \in \Omega_\ell$ is defined as
\begin{align}
X_{t+1}^\mathrm{B}(k)=\min(X_t^\mathrm{B}(k)+\xi^\mathrm{B}(k,t+1) ,X_{t+1}^\mathrm{B}(k-1)-1),
\end{align}
for $k=2,\dots,\ell$ and $X_{t+1}^\mathrm{B}(1)=X_t^\mathrm{B}(1)+\xi^\mathrm{B}(1,t+1)$.
$\xi^\mathrm{B}(k,t)$ are random variables satisfying
$\displaystyle \mathsf{P}(\xi^\mathrm{B}(k,t)=1)=\frac{\rho_k x_t}{1+\rho_k x_t}$, 
$\displaystyle \mathsf{P}(\xi^\mathrm{B}(k,t)=0)=\frac{1}{1+\rho_k x_t}$, where $\rho_k, x_t$ are real numbers satisfying $0 < \rho_k x_t$ for all $k,t$.

\medskip 
\noindent \textbf{(C)} Geometric distribution with blocking behavior:

From time $t$ to time $t+1$,
the evolution of the particle system $X_t^\mathrm{C} \in \Omega_\ell$ is defined as
\begin{align}
X_{t+1}^\mathrm{C}(k)=\min(X_t^\mathrm{C}(k)+\xi^\mathrm{C}(k,t+1),X_{t}^\mathrm{C}(k-1)-1),
\end{align}
for $k=2,\dots,\ell$ and $X_{t+1}^\mathrm{C}(1)=X_t^\mathrm{C}(1)+\xi^\mathrm{C}(1,t+1)$.
$\xi^\mathrm{C}(k,t)$ are random variables satisfying
$\mathsf{P}(\xi^\mathrm{C}(k,t)=r)=(1-\pi_k x_t) (\pi_k x_t)^{r}$, $r \in \mathbb{Z}_{\ge 0}$, where $\pi_k, x_t$ are real numbers satisfying $0 < \pi_k x_t < 1$ for all $k,t$.

\medskip 
\noindent \textbf{(D)} Bernoulli distribution with pushing behavior:

From time $t$ to time $t+1$,
the evolution of the particle system $X_t^\mathrm{D} \in \Omega_\ell$ is defined as
\begin{align}
X_{t+1}^\mathrm{D}(k)=\max(X_t^\mathrm{D}(k)+\xi^\mathrm{D}(k,t+1) ,X_{t+1}^\mathrm{D}(k+1)+1),
\end{align}
for $k=1,\dots,\ell-1$ and $X_{t+1}^\mathrm{D}(\ell)=X_t^\mathrm{D}(\ell)+\xi^\mathrm{D}(\ell,t+1)$.
$\xi^\mathrm{D}(k,t)$ are random variables satisfying
$\displaystyle \mathsf{P}(\xi^\mathrm{D}(k,t)=1)=\frac{\rho_k x_t}{1+\rho_k x_t}$, 
$\displaystyle \mathsf{P}(\xi^\mathrm{D}(k,t)=0)=\frac{1}{1+\rho_k x_t}$, where $\rho_k, x_t$ are real numbers satisfying $0 < \rho_k x_t$ for all $k,t$.
\end{subequations}

For a sequence of integers $\lambda=(\lambda_1,\lambda_2,\dots,\lambda_\ell)$ $(\lambda_1 \ge \lambda_2 \ge \cdots \ge \lambda_\ell)$,
we write $X=\tilde{\lambda}$ if $X \in \Omega_\ell$ satsifies $X(j)=\lambda_j-j$, $j=1,2,\dots,\ell$.

\begin{thm}
\label{thm:transition_prob}
Suppose $\ell(\lambda), \lambda_1 \leq \ell$.
Suppose $\pi_j x_i \in (0, 1)$ and $\rho_j x_i > 0$ for all $i$ and $j$.
Set $\alpha_j = \rho_{j+1}$ and $\beta_j = \pi_{j+1}$.
The transition probabilities of the  four particle systems $X=X^\mathrm{A},X^{\mathrm{B}}, X^\mathrm{C}, X^\mathrm{D}$
from the initial configuration $X_0=\tilde{\mu}$ at time zero to the final configuration $X_n=\tilde{\lambda}$ at time $n$ are given by
\begin{align*}
\mathsf{P}(X^\mathrm{A}_n=\tilde{\lambda}|X^\mathrm{A}_0=\tilde{\mu})
 & = \prod_{j=1}^{\ell} \prod_{i=1}^n (1 - \pi_j x_i) \bpi^{\lambda/\mu} \dG_{\lambda/\mu}(\xx_n; \bpi^{-1}) =:\prob_{A,n}(\lambda | \mu) ,
\\
\mathsf{P}(X^\mathrm{B}_n=\tilde{\lambda}|X^\mathrm{B}_0=\tilde{\mu})
& = \frac{\brho^{\lambda/\mu}}{\displaystyle \prod_{i=1}^n (1 + \rho_1 x_i)} \wG_{\lambda' \ds \mu'}(\xx_n; \bal)=:\prob_{B,n}(\lambda | \mu) ,
\\
\mathsf{P}(X^\mathrm{C}_n=\tilde{\lambda}|X^\mathrm{C}_0=\tilde{\mu})
 & = \prod_{i=1}^n (1 - \pi_1 x_i) \bpi^{\lambda/\mu} \G_{\lambda \ds \mu}(\xx_n; \bbe)
=:\prob_{C,n}(\lambda | \mu)
,
\\
\mathsf{P}(X^\mathrm{D}_n=\tilde{\lambda}|X^\mathrm{D}_0=\tilde{\mu})
 & = \frac{\brho^{\lambda/\mu}}{\displaystyle\prod_{j=1}^{\ell} \prod_{i=1}^n (1 + \rho_j x_i)} \dwG_{\lambda'/\mu'}(\xx_n; \brho^{-1})
=:\prob_{D,n}(\lambda | \mu)
.
\end{align*}

\end{thm}

Here, $\bpi^{\lambda/\mu}=\prod_{j=1}^\ell \pi_j^{\lambda_j-\mu_j} $
and defined similarly for $\brho^{\lambda/\mu}$.
See Section~\ref{sec:Ksymfunc} for combinatorial definitions of the skew refined Grothendieck polynomials, and/or its dual/weak version. 
There are several expressions and methods to derive determinant forms
of the skew refined (dual) Grothendieck polynomials.
One type of determinant forms involve a summation for each of the matrix elements
which is often referred to as the Jacobi--Trudi determinant formulas.
See \cite{AY20,Kim20,Kim20II} and~\cite[Thm.~6.1]{HJKSS24} for the canonical refined version derived by combinatorial arguments.
An algebraic approach using the free fermion technique is introduced in \cite{Iwao19,Iwao21,Iwao20} and the canonical refined version is given in \cite[Thm.~4.1]{IMS22}.
This Jacobi--Trudi type determinants together with the overall factors and specialized to $\xx_n=1$ correspond to determinant forms of transition probabilities for time homogeneous case in~\cite[Thm.~1]{DW08} and~\cite[Thm.~2.1]{Johansson10}.
For example, the Jacobi--Trudi determinant for the dual Grothendieck polynomials is given by~\cite[Thm.~6.1]{HJKSS24}, \cite[Thm.~4.1]{IMS22}, and~\cite[Thm.~4.15]{MS20}:
\begin{align}
g_{\lambda/\mu}(\xx_n;\bt)
=
\mathrm{det}
\Bigg[\sum_{m \ge 0} h_{\lambda_i-\mu_j-i+j-m}(\mathbf{x}_n ) \alpha_{m}^{ij}(\bt) 
\Bigg]_{i,j=1}^{\ell},
\end{align}
where $\alpha_m^{ij}(\bt)=h_m(t_j,\dots,t_{i-1})$ (a homogeneous symmetric function) for $i \geq j$
and $\alpha_m^{ij}(\bt)=e_m(-t_i,\dots,-t_{j-1})$ (an elementary symmetric function) for $i<j$.
Multiplying by the normalization factor $\prod_{j=1}^{\ell} \prod_{i=1}^n (1 - \pi_j x_i) \bpi^{\lambda/\mu}$ gives
the determinant form for transition probabilities for Case~A \cite[Cor. 4.18]{MS20},
which becomes the expression for the time homogeneous version
\cite[Thm.~1.A]{DW08} and \cite[Thm.~2.1]{Johansson10} by setting
 $\xx_n=1$ (specializing all $\xx$-variables to 1).
The Jacobi--Trudi determinant expression for the Grothendieck polynomials is \cite[Thm.~4.1]{IMS22}
\begin{align}
G_{\lambda \ds \mu}(\xx_n;\bt)
=
\mathrm{det}
\Bigg[\sum_{m \ge 0} h_{\lambda_i-\mu_j-i+j+m}(\mathbf{x}_n ) \beta_{m}^{ij}(\bt) 
\Bigg]_{i,j=1}^{\ell},
\end{align}
where $\beta_m^{ij}(\bt)=h_m(t_i,\dots,t_{j-1})$ for $i \leq j$
and $\beta_m^{ij}(\bt)=e_m(-t_j,\dots,-t_{i-1})$ for $i>j$,
and the determinants for the weak version $j_{\lambda \ds \mu}(\xx_n;\bt)$, $J_{\lambda \ds \mu}(\xx_n;\bt)$
are obtained from $g_{\lambda \ds \mu}(\xx_n;\bt)$, $G_{\lambda \ds \mu}(\xx_n;\bt)$ respectively,
by applying the $\omega$-involution on the $\xx$-variables, which interchanges $h_j(\xx_n)$ and $e_j(\xx_n)$.

Rewriting each of the sum in the matrix elements of the Jacobi--Trudi determinants as an integral, the determinant forms become the discrete time version of Sch\"utz type formulas.
For example, we have \cite[Thm.~4.19]{IMS22}
\begin{align}
g_{\lambda/\mu}(\xx_n;\bt)
=
\mathrm{det}
\Bigg[
\frac{1}{2 \pi \mathbf{i}} \oint_{\gamma}
\frac{\prod_{k=1}^{j-1}(1-t_k z) }{\prod_{k=1}^{i-1}(1-t_k z) }
\frac{1}{\prod_{k=1}^n (1-x_k z) z^{\lambda_i-\mu_j-i+j+1}}
dz
\Bigg]_{i,j=1}^{\ell},
\end{align}
where $\gamma$ is a small circle centered at the origin, and changing the integration variable to its inverse and multiplying overall factors gives~\cite[Thm.~1]{JR20}.
This type of determinant forms is the starting point to finally obtain the Fredholm determinant expressions for multipoint distributions in several papers.
See \cite{JR20} for Case~A and more generally by \cite{MatetskiRemenik23,MatetskiRemenik23b} for all cases which they further generalized to more generic models of sequential and parallel updates.
See also \cite{BLSZ22} corresponding to Case~B, where a nonintersecting lattice path construction of transition probabilities was given which further lead them to the Fredholm determinant expressions.
There are several formulas for the refined Grothendieck polynomials and its variants derived in~\cite{HJKSS24,HJKSS25,IMS22} for example, which may be useful for further studies on transition probabilities.
We discuss some applications in Section~\ref{sec:multipoint}.

We also make a comment that a different version of inhomogeneous TASEP (space inhomogeneous version) was studied in~\cite{Assiotis20,Petrov20}, and correlation kernels were obtained as determinants with matrix elements involving double integrals.
In \cite{Assiotis20}, this was done by introducing and generalizing the analysis of the process on Gelfand--Tsetlin patterns originally due to \cite{BF14}, and in \cite{Petrov20} this was derived in a different way by identifying with some certain Schur process.
It is an open question to derive these types of determinants directly from the approach used in this paper.

Later, rather than the TASEP version, we use the equivalent bosonic version by Dieker--Warren in~\cite{DW08}, which is more suitable for the proofs given in this paper, that where the $i$-th particle is shifted by $i$ steps to the right.
See Section~\ref{sec:particle_processes} for the precise statement.


Our proof shows these generalized Schur operators satisfy the Knuth relations, which has proven useful in the study of symmetric functions, such as in~\cite{FG98,Fomin95}, and we then use the Markov property to reduce the equivalence to a computation for $n = 1$.
We give two extensions that could be used to build the K-theoretic symmetric functions, but only the ones for $\G_{\lambda\ds\mu}(\xx; \bal, \bbe)$ satisfy the Knuth relations (Theorem~\ref{thm:noncommutative_blocking} and Example~\ref{ex:dual_non_Knuth}).
This allows us to describe the transition functions for a new particle process in Section~\ref{sec:canonical_process} that involves $\bal$ acting as ``local current'' parameters, where the rate depends on the position of the particle.
We prove the analogous result to Theorem~\ref{thm:transition_prob} for this new particle process in Theorem~\ref{thm:canonical_process} with $\G_{\lambda\ds\mu}(\xx; \bal, \bbe)$.
We also describe extensions of this process to the other cases.
However, this process can only be described ``bosonically,'' hence it cannot be applied to the slow bond problem (which requires a ``fermionic'' presentation; see Remark~\ref{rem:boson_pos_param}).
When $\pi_i = 0$, our new model becomes a special case of the model from~\cite{KPS19} (see Remark~\ref{rem:canonical_KPS_comparison} for the precise relationship).

We also give another proof of Theorem~\ref{thm:transition_prob} in Section~\ref{sec:bijection} through a direct combinatorial bijection between the tableaux description using more refined conditional probabilities.
Essentially, it is given by showing the branching rules~\cite[Prop.~4.5]{IMS22} corresponds to the Markov property and showing the result directly for $n = 1$.
Our proof can be seen as analogous to using the (dual) RSK bijection and the Schur decomposition in Case~A and Case~D that was described in~\cite{MS20} (compare also the intertwining kernels of~\cite{DW08} with~\cite[Thm.~4.4]{IMS22}).
As a consequence, we can explicitly describe how the tableaux encode the movement of the particles.

Let us mention one technical point about the Case~C result in Theorem~\ref{thm:transition_prob}.
We need to take the Taylor series expansion of $f(\zeta) = (1 + \zeta)^{-1}$ (around $\zeta = 0$) in order to obtain the equality with the combinatorial formula with $\wG_{\lambda \ds \mu}(\xx_n; \bal)$.
Thus, strictly speaking, to do the Taylor series expansion we should require that $\rho_i x_j \in (0, 1)$.
We can make a change to the combinatorial description in terms of rational functions to address this, and, in principle, we would then need to prove new free fermionic and Jacobi--Trudi formulas.
We leave this for the interested reader.

We describe some additional results we obtain from Theorem~\ref{thm:transition_prob}.
Using the skew Cauchy identity~\cite[Thm.~4.6]{IMS22}, we give determinantal formulas for the multi-point distributions in Section~\ref{sec:multipoint} for all cases.
We also give another proof for Case~A using a refinement of~\cite[Cor.~3.14]{MS20}, but this expansion has a natural geometric interpretation as coming from the K-homology classes of structure sheaves of Schubert varieties studied in the work of Takigiku~\cite{Takigiku18,Takigiku18II} (see also~\cite{Takigiku19}) at $\bbe = 1$ (as opposed to ideal sheaves of boundaries of Schubert varieties in~\cite{LamPyl07}).
For the Case~C multi-point distributions, a similar geometric construction for the underlying symmetric functions can likely be given from~\cite{WZJ16}.
It would be interesting to see how these geometry interpretations relate to the integrability and determinant (and integral) formulas.
We show that taking the continuous time limit for the blocking behavior recovers the classical continuous time TASEP (Theorem~\ref{thm:continuous_blocking}).
We prove that the pushing behavior satisfies the same master equation, but with different boundary conditions (Theorem~\ref{thm:continuous_pushing}).

Let us discuss how our results relate to two other independent works that appeared while this paper was being prepared.
The first is by Bisi, Liao, Saenz, and Zygouras~\cite{BLSZ22} that also studied the time-dependent version of~\cite[Case~B]{DW08}.
However, the techniques and results in their work~\cite{BLSZ22} are (generally) different than what we obtained here.
It is likely that our results could lead to new proofs of some of the formulas in~\cite{BLSZ22}.
The second is a Grothendieck measure was studied in~\cite{GP23}, which with $\bbe = \beta$ is coming from the Cauchy-type identity in~\cite[Cor.~5.4]{MS13} with the dual Grothendieck functions $\overline{G}_{\lambda}(\yy; \bbe)$ being rescaled.
This Cauchy-type identity is different than the one considered in~\cite{HJKSS24,HJKSS25,IMS22} (instead it appears to be related to~\cite{MS13}; see also~\cite{GK17}), and so it does not relate to our results.
Lastly, shortly before this paper appeared on the ar$\chi$iv, independent work by Assiotis~\cite{Assiotis23} also was posted, where Assiotis also studied a position-inhomogeneous version of the TASEP cases we study here.
More specifically, determinantal correlation functions are constructed with some limit results are proven based on generalizing Toeplitz matrices and intertwinings of Markov semigroups are given in~\cite{Assiotis23}.
Thus,~\cite{Assiotis23} has related but distinct results from ours using different techniques.

We conclude the introduction by mentioning some potential applications of our work.
Since we are allowing any initial configuration $\mu$ with finite distance from the step initial condition, we can approximate the flat initial condition by taking $\mu$ to be a sufficiently large staircase partition.
As such, we expect to be able to compute generalizations of results on flat initial conditions such as~\cite{BFPS07,BFS08} with having probabilities depend on the particles.
Furthermore, we believe that limit shapes/densities can be computed using the fermionic Fock space description following~\cite{Okounkov01,OR03}, although currently an explicit description of the projection operator as an element in the Clifford algebra is not known.
However, because of~\cite[Prop.~1.2]{GP23}, it is likely that the projection operator cannot be used with Wick's theorem.

This paper is organized as follows.
In Section~\ref{sec:background}, we give some background on refined (dual) Grothendieck polynomials, the related combinatorics, and the stochastic processes we consider.
In Section~\ref{sec:noncomm_operators}, we describe our Schur operators for canonical and dual Grothendieck polynomials.
In Section~\ref{sec:operator_dynamics}, we prove Theorem~\ref{thm:transition_prob} using our Schur operators.
In Section~\ref{sec:bijection}, we prove Theorem~\ref{thm:transition_prob} using a direct combinatorial argument.
In Section~\ref{sec:multipoint}, we give our formulas for the multi-point distributions in each case.
In Section~\ref{sec:continuous}, we show the continuous time limits of our TASEP processes.
In Section~\ref{sec:canonical_process}, we describe our new blocking-behavior particle process with ``local current'' parameters $\bal$ that is the common generalization of Case~B and Case~C.
In Section~\ref{sec:conclusion}, we offer some concluding remarks on our work.

\subsection*{Acknowledgements}

The authors thank Guillaume Barraquand, Dan Betea, A.~B.\ Dieker, Darij Grinberg, Yuchen Liao, Jang Soo Kim, Leonid Petrov, Mustazee Rahman, Tomohiro Sasamoto, Jon Warren, Damir Yeliussizov, Paul Zinn-Justin, and Nikos Zygouras for valuable conversions.
The authors thank Theo Assiotis for letting us know of his recent preprint and explanations of his results.
The authors thank the referee for their comments.

This work benefited from computations using {\sc SageMath}~\cite{sage,combinat}.
This work was partly supported by Osaka City University Advanced Mathematical Institute (MEXT Joint Usage/Research Center on Mathematics and Theoretical Physics JPMXP0619217849).
This work was supported by the Research Institute for Mathematical Sciences, an International Joint Usage/Research Center located in Kyoto University.

\section{Background}
\label{sec:background}

Let $\lambda = (\lambda_1, \lambda_2, \dotsc, \lambda_{\ell})$ be a \defn{partition}, a weakly decreasing finite sequence of positive integers.
We denote the set of all partitions by $\mcP$.
We draw the Young diagrams of our partitions using English convention.
We will often extend partitions with additional entries at the end being $0$, and let $\ell(\lambda)$ denote the largest index $\ell$ such that $\lambda_{\ell} > 0$.
Let $\lambda'$ denote the conjugate partition.
We often write our partitions as words.
A \defn{hook} is a partition $\lambda$ of the form $a1^{m} = (a, 1, \dotsc, 1)$ with $1$ appearing $m$ times, where the \defn{arm} is $a-1$ and the \defn{leg} is $m$.
For $\mu \subseteq \lambda$, a skew shape $\lambda / \mu$ is the Young diagram formed by removing $\mu$ from $\lambda$, and we identity $\lambda / \emptyset = \lambda$.

Let $\xx = (x_1, x_2, \ldots)$ denote a countably infinite sequence of indeterminates.
We will often set all but finitely many of the indeterminates $\xx$ to $0$, which we denote as $\xx_n := (x_1, \dotsc, x_n, 0, 0, \ldots)$.
We make similar definitions for any other sequence of indeterminates, such as $\yy = (y_1, y_2, \ldots)$.
We also require infinite sequences of parameters $\bal = (\alpha_1, \alpha_2, \ldots)$ and $\bbe = (\beta_1, \beta_2, \ldots)$, which we often treat as indeterminates.

\subsection{K-theoretic symmetric functions}
\label{sec:Ksymfunc}

A \defn{semistandard tableau} is a filling of a Young diagram $\lambda / \mu$ with positive integers such that the rows are weakly increasing left-to-right and strictly increasing top-to-bottom.
The \defn{weight} of a semistandard tableau $T$ is
\[
\wt(T) = \prod_{i=1}^{\infty} x_i^{m_i},
\]
where $m_i$ is the number of $i$'s that appear in $T$.
This is a finite product since there are finitely many boxes in $\lambda / \mu$.

The \defn{Schur function} is the generating function
\[
s_{\lambda / \mu}(\xx) = \sum_T \wt(T),
\]
where we sum over all semistandard tableaux $T$ of shape $\lambda / \mu$.
The Schur functions $\{ s_{\lambda} \}_{\lambda \in \mcP}$ form a basis for the ring of symmetric functions, and so we can define the \defn{Hall inner product} by declaring the Schur functions are an orthonormal basis
\[
\inner{s_{\lambda}}{s_{\mu}} = \delta_{\lambda\mu}.
\]
Furthermore, we have a natural grading on symmetric functions by having the degree of $s_{\lambda}$ be $\abs{\lambda}$.
There is also an algebra involution defined by $\omega s_{\lambda / \mu} = s_{\lambda' / \mu'}$.
We refer the reader to~\cite[Ch.~7]{ECII} and~\cite[Ch.~I]{MacdonaldBook} for more details on symmetric functions.

A \defn{hook-valued tableau} of skew shape $\lambda / \mu$ is a filling of the Young diagram by hook shaped tableau satisfying the local conditions
\[
\ytableaushort{{\ta}{\tb},{\tc}}
\qquad\qquad
\begin{array}{c@{\;}c@{\;}c}
\max(\ta) & \leq & \min(\tb) \\[-8pt] \rotatebox{-90}{$<$} \\ \min(\tc)
\end{array}
\]
(provided the requisite box exists).
Note that this is a generalization of the semistandard conditions as the conditions are equivalent to the usual standard ones when $\ta,\tb,\tc$ all consist of a single entry.

For $\mu \subseteq \lambda$, the \defn{canonical Grothendieck function}\footnote{This should be called the refined canonical Grothendieck polynomials following~\cite{HJKSS24,HJKSS25} as they are a refinement of those introduced by Yeliussizov~\cite{Yel17}, we dropped the word ``refined'' to simplify our nomenclature.}
is the generating function
\[
\G_{\lambda / \mu}(\xx; \bal, \bbe) = \sum_T \prod_{\bbb \in T} (-\alpha_i)^{a(\bbb)} (-\beta_j)^{b(\bbb)} \wt(\bbb),
\]
where we sum over all hook-valued tableaux $T$ of shape $\lambda / \mu$, product over all entries $\bbb$ in $T$ with $a(\bbb)$ (resp.~$b(\bbb)$) the arm (resp.\ leg) of the shape of $\bbb$ and $i$ (resp.~$j$) the row (resp.\ column) of the entry.
We indicate various specializations and relation with the literature in Table~\ref{table:G_spec}.
We note that technically the canonical Grothendieck functions lie in the completion of the ring of symmetric functions given by the grading; so we are allowed to take infinite sums with finite sums in each graded component.
However this does not affect our computations or results, and so we suppress this distinction in this paper.
A basis for (the completion of) symmetric functions is given by $\{\G_{\lambda}(\xx; \bal, \bbe)\}_{\lambda \in \mcP}$ since
\[
G_{\lambda}(\xx; \bal, \bbe) = s_{\lambda}(\xx) + \sum_{\lambda \subsetneq \mu} (-1)^{\abs{\mu} - \abs{\lambda}} E_{\lambda}^{\mu}(\bal, \bbe) s_{\mu}(\xx),
\]
where $E_{\lambda}^{\mu}(\bal, \bbe) \in \ZZ_{\geq 0}[\bal, \bbe]$~\cite{HS20,HJKSS24,HJKSS25}.

\begin{table}
\begin{center}
\begin{tabular}{cccccc}
\toprule
Work & \cite{HJKSS24,HJKSS25} & \cite{HS20,Yel17} & \cite{CP21,Iwao20,Iwao19} & \cite{MS20}
\\ \midrule
Specialization & $\G_{\lambda}(\xx; -\bal, \bbe)$ & $\G_{\lambda}(\xx; -\alpha, -\beta)$ & $\G_{\lambda}(\xx; 0, \beta)$ & $\G_{\lambda}(\xx; 0, -\bbe)$
\\ \bottomrule
\end{tabular}
\end{center}
\caption{The relationship between our sign choices and some other papers in the literature.}
\label{table:G_spec}
\end{table}

We note that if $\bal = 0$, then all entries of the hook-valued tableau must be column shapes, which we can equate with sets and recovers the set-valued tableau description of~\cite{CP21}, which refines~\cite{Buch02}.
Likewise, if $\bbe = 0$, then we have row shapes that we equate with multisets, refining~\cite{LamPyl07}.
Similarly, we call $\G_{\lambda}(\xx; \bbe) := \G_{\lambda}(\xx; 0, \bbe)$ a \defn{Grothendieck function}\footnote{Typically these are called the symmetric Grothendieck functions to distinguish these from those $\fG_w$ that arise from the (connective) K-theory of the flag variety, which depend on a permutation $w$. However, since we do not use $\fG_w$ here, we omit the word ``symmetric'' from our terminology.} and $\wG_{\lambda}(\xx; \bal) := \G_{\lambda}(\xx; \bal, 0)$ a \defn{weak Grothendieck functions}.

The \defn{dual canonical Grothendieck functions} $\{\dG_{\lambda}(\xx; \bal, \bbe)\}_{\lambda \in \mcP}$ are defined as the dual basis to the canonical Grothendieck functions under the Hall inner product.
A combinatorial definition was given in~\cite{HJKSS25}, which can be seen as the obvious refinement of the rim border tableaux description of~\cite{Yel17}.
As such, we can extend the definition to skew shapes.
For our purposes, we will only use the dual basis to the (weak) Grothendieck functions (that is, $\bal = 0$ or $\bbe = 0$), and thus we restrict to describing the combinatorics of the special cases for the dual basis to the (weak) Grothendieck functions following~\cite[Thm.~9.15]{LamPyl07}.
A \defn{reverse plane partition} (resp.\ \defn{valued-set tableau}) is a semistandard Young tableau where we are allowed to merge boxes in the same column (resp.\ row) with the entry considered to be aligned at the bottom (resp.\ right).
The weight is the same as for semistandard tableaux.
We then have the generating functions
\[
\dG_{\lambda/\mu}(\xx; \bbe) = \dG_{\lambda/\mu}(\xx; 0, \bbe) = \sum_T \prod_{j=1}^{\ell-1} \beta_j^{r_j} \wt(T),
\quad
\dwG_{\lambda/\mu}(\xx; \bbe) = \dG_{\lambda/\mu}(\xx; \bal, 0) = \sum_T \prod_{j=1}^{\lambda_1-1} \alpha_j^{c_j} \wt(T),
\]
where $r_j$ (resp.~$c_j$) is the number of boxes in row (resp.\ column) $j$ that have been merged with the box below (resp.\ to the right) and we sum over all reverse plane partitions (resp.\ valued-set tableaux) of shape $\lambda / \mu$.
The definition of valued-set tableau follows~\cite{HS20}, which is conjugate to that of~\cite{LamPyl07}.

We note that our description of reverse plane partitions matches the classical definition by simply filling in the merged boxes with the entry of the merged box.
The inverse map is merging duplicated entries in the same column.
Thus, the description of $\dG_{\lambda / \mu}(\xx; \bbe)$ matches that introduced in~\cite{GGL16}.
We have described reverse plane partitions as above to better demonstrate the following symmetry that motivated the definition of canonical Grothendieck polynomials, even though we will write our reverse plane partitions below using the classical description.

\begin{thm}[{\cite[Thm.~1.7]{HJKSS24}}]
\label{thm:canonical_conjugate}
We have
\[
\omega \G_{\lambda/\mu}(\xx; \bal, \bbe) = \G_{\lambda'/\mu'}(\xx; \bbe, \bal),
\qquad\qquad
\omega \dG_{\lambda/\mu}(\xx; \bal, \bbe) = \dG_{\lambda'/\mu'}(\xx; \bbe, \bal).
\]
\end{thm}

As a consequence of Theorem~\ref{thm:canonical_conjugate}, we have the relationship
\[
\omega \G_{\lambda/\mu}(\xx; \bal) = \wG_{\lambda/\mu}(\xx; \bal),
\qquad\qquad
\omega \dG_{\lambda/\mu}(\xx; \bal) = \dwG_{\lambda/\mu}(\xx; \bal),
\]
between the (dual) weak Grothendieck functions and the (dual) Grothendieck functions.
These are a refinement of~\cite[Prop.~9.22]{LamPyl07}.

For the canonical Grothendieck polynomials, we note that the skew shape description is not natural from the branching rules (a precise description is given below), the skewing operator~\cite[Sec.~4]{Iwao20}, nor coproduct formula~\cite[Sec.~5]{Buch02} perspective.
Refining~\cite[Eq.~(6.4)]{Buch02} and~\cite[Prop.~8.8]{Yel17}, we define~\cite[Sec.~4.1]{IMS22}
\begin{equation}
\label{eq:double_to_single_slash}
\G_{\lambda\ds\mu}(\xx; \bal, \bbe) := \sum_{\nu \subseteq \mu} \prod_{(i,j) \in \mu/\nu} -(\alpha_i + \beta_j) \G_{\lambda / \nu}(\xx; \bal, \bbe),
\end{equation}
where $\nu$ is formed by removing some of the corners of $\mu$ (that is, boxes $(i, \mu_i)$ such that $\mu_i > \mu_{i+1}$).
We remark that we have the following identity by the same proof as~\cite[Thm.~4.4]{IMS22}.

\begin{prop}
We have
\[
\omega \G_{\lambda\ds\mu}(\xx; \bal, \bbe) = \G_{\lambda'\ds\mu'}(\xx; \bbe, \bal).
\]
\end{prop}

Equation~\eqref{eq:double_to_single_slash} allows us to give the following (\textit{cf.}~\cite[Prop.~8.7, 8.8]{Yel17}).

\begin{prop}[{Branching rules~\cite[Prop.~4.5]{IMS22}}]
\label{prop:branching_rules}
We have
\begin{align*}
\G_{\lambda/\mu}(\xx, \yy; \bal, \bbe) & = \sum_{\nu \subseteq \lambda} \G_{\lambda\ds\nu}(\yy; \bal, \bbe) \G_{\nu/\mu}(\xx; \bal, \bbe),
\\
\G_{\lambda\ds\mu}(\xx, \yy; \bal, \bbe) & = \sum_{\mu \subseteq \nu \subseteq \lambda} \G_{\lambda\ds\nu}(\yy; \bal, \bbe) \G_{\nu\ds\mu}(\xx; \bal, \bbe),
\\
\dG_{\lambda/\mu}(\xx, \yy; \bal, \bbe) & = \sum_{\mu \subseteq \nu \subseteq \lambda} \dG_{\lambda / \nu}(\yy; \bal, \bbe) \dG_{\nu/\mu}(\xx; \bal, \bbe).
\end{align*}
\end{prop}

As was shown in~\cite[Eq.~(4.3)]{IMS22}, we have
\[
G_{\lambda/\mu}(\xx; \bal, \bbe)=\prod_{(i,j)\in \mu/\lambda}(\alpha_i+\beta_j)\cdot G_{\lambda/(\lambda\cap \mu)}(\xx; \bal, \bbe),
\]
and in particular, we do \emph{not} have the vanishing property for $\G_{\lambda/\mu}(\xx; \bal, \bbe)$ that $\G_{\lambda\ds\mu}(\xx; \bal, \bbe)$ exhibits: that $\G_{\lambda\ds\mu}(\xx; \bal, \bbe) = 0$ whenever $\mu \subseteq \lambda$.

We will also need the skew Cauchy formula from~\cite[Thm.~4.6]{IMS22} (non-skew versions can be found in~\cite{HJKSS24,HJKSS25} or as a consequence of~\cite[Rem.~3.9]{CP21}).
This is a refined version of~\cite[Thm.~1.1]{Yel19}.

\begin{thm}[Skew Cauchy formula]
\label{thm:skew_cauchy}
We have
\[
\sum_{\lambda} \G_{\lambda \ds \mu}(\xx; \bal, \bbe) \dG_{\lambda / \nu}(\yy; \bal, \bbe) 
= \prod_{i,j} \frac{1}{1 - x_i y_j} \sum_{\eta}  \G_{\nu \ds \eta}(\xx; \bal, \bbe)\dG_{\mu / \eta}(\yy; \bal, \bbe).
\]
\end{thm}

\subsection{Supersymmetric functions}

We set some additional standard notation from symmetric function theory.
We (again) refer the reader to the standard textbooks~\cite{MacdonaldBook,ECII} for more information.
Let
\begin{align*}
e_m(\xx) & = \sum_{i_1 > \cdots > i_m} x_{i_1} \dotsm x_{i_m},
& e_{\lambda}(\xx) & = e_{\lambda_1} \cdots e_{\lambda_{\ell}},
\\ h_m(\xx) & = \sum_{i_1 \leq \cdots \leq i_m} x_{i_1} \dotsm x_{i_m},
& h_{\lambda}(\xx) & = h_{\lambda_1} \cdots h_{\lambda_{\ell}},
\\ p_m(\xx) & = \sum_{i=1}^{\infty} x_i^m,
& p_{\lambda}(\xx) & = p_{\lambda_1} \cdots p_{\lambda_{\ell}},
\end{align*}
denote the elementary, homogeneous, and power sum symmetric functions, respectively.
We consider $e_0(\xx) = h_0(\xx) = p_0(\xx) = 1$.
We have $\omega p_m(\xx) = (-1)^{m-1} p_m(\xx) = -p_m(-\xx)$ for $m > 0$.
Since, the powersum symmetric functions generate the ring of symmetric functions as a polynomial ring (over $\QQ$) and the monomials (in the variables $p_m := p_m(\xx)$) form a basis.
Thus, we can define polynomials by the equations
\[ 
E_{\mu}(p_1, p_2, \ldots) = e_{\mu}(\xx),
\qquad\qquad
H_{\mu}(p_1, p_2, \ldots) = h_{\mu}(\xx),
\qquad\qquad
S_{\mu}(p_1, p_2, \ldots) = s_{\mu}(\xx).
\] 
For example, $E_{21}(p_1, p_2, \ldots) = \frac{1}{2} p_1^3 - \frac{1}{2} p_2 p_1$.

Now we recall some particular \defn{supersymmetric functions}; we refer the reader to~\cite[Ch.~I]{MacdonaldBook} for more details.
We define the supersymmetric elementary, homogeneous, powersum, and Schur functions as
\begin{align*}
e_m(\xx/\yy) & = \sum_{k=0}^m (-1)^{m-k} e_k(\xx) h_{m-k}(\yy),
\\ h_m(\xx/\yy) & = \sum_{k=0}^m (-1)^{m-k} h_k(\xx) e_{m-k}(\yy),
\\ p_m(\xx/\yy) & = p_m(\xx) - p_m(\yy),
\\ s_{\lambda}(\xx/\yy) & = \sum_{\mu} (-1)^{\abs{\lambda} - \abs{\mu}} s_{\mu}(\xx) s_{\lambda' / \mu'}(\yy).
\end{align*}
When $\yy = \emptyset$, that is we have set all of the $\yy$ indeterminates to $0$, we have $f(\xx / \yy) = f(\xx)$ for any supersymmetric function $f$.
The involution $\omega$ also extends to supersymmetric functions by
$
\omega s_{\lambda / \mu}(\xx / \yy) = s_{\lambda'/\mu'}(\xx / \yy).
$


The supersymmetric functions can also be described in terms of plethystic substitution.
While we will not give a detailed account, we will briefly review the relevant descriptions for understanding the results in~\cite{HJKSS24,HJKSS25} and refer the reader to~\cite{LR11} and~\cite[Ch.~I]{MacdonaldBook} for a more detailed description.
Let $X = x_1 + x_2 + \cdots$ and $Y = y_1 + y_2 + \cdots$.
For a symmetric function $f$, we define $f[X] = f(x_1, x_2, \ldots)$, and if $Z = z_1 + z_2 + \cdots + z_n$, then we have $f[Z] = f(z_1, z_2, \dotsc, z_n, 0, 0, \ldots)$.
We also can define
\begin{align*}
h_m[X - Y] = h_m(\xx/\yy),
\qquad\qquad
e_m[X - Y] = e_m(\xx/\yy),
\qquad\qquad
p_m[X - Y] = p_m(\xx/\yy).
\end{align*}
As a consequence, we have that $h_m[-Y] = (-1)^m e_m(\yy)$ and $e_m[-Y] = (-1)^m h_m(\yy)$.
Furthermore, we have the well-known plethystic identities
\begin{align*}
h_m\bigl( (\xx \sqcup \xx')/(\yy \sqcup \yy') \bigr) & = h_m[X + X' - Y - Y'] = \sum_{a+b=m} h_a(\xx/\yy) h_b(\xx'/\yy'),
\\
e_m\bigl( (\xx \sqcup \xx')/(\yy \sqcup \yy') \bigr) & = e_m[X + X' - Y - Y'] = \sum_{a+b=m} e_a(\xx/\yy) e_b(\xx'/\yy'),
\end{align*}
(see, \textit{e.g.},~\cite[Prop.~2.1]{HJKSS24}).
Next, we recall the notation given in~\cite[Def.~2.3]{HJKSS24}:
\[
h_m[X \ominus Y] := \sum_{a-b=m} h_a[X] h_b[Y],
\qquad\qquad
e_m[X \ominus Y] := \sum_{a-b=m} e_a[X] e_b[Y].
\]
We note that these can have infinite nonzero terms and be nonzero even when $m$ is negative.

In order to avoid confusion with the plethystic negative and negating the variables, we will not use plethystic notation, and instead follow~\cite{IMS22}, where we write $h_m(\xx \ds \yy) := h_m[X \ominus Y]$ and $e_m(\xx \ds \yy) := e_m[X \ominus Y]$.

Additionally, note the ordering of the variables for the elementary symmetric functions.
We have chosen this so that the definitions extend to the \defn{noncommutative symmetric functions} introduced by Fomin and Greene~\cite{FG98}.
We summarize the results that we need as follows.
For a sequence of linear operators $\{\tau_i \colon \field[\mcP] \to \field[\mcP] \}_{i=1}^{\infty}$, the \defn{weak Knuth relations} are
\begin{subequations}
\label{eq:knuth_relations}
\begin{align}
\tau_j \tau_i \tau_k & = \tau_j \tau_k \tau_i & &\text{for all } i \geq j > k, \quad i - k \geq 2, \label{eq:knuth_right}
\\ \tau_i \tau_k \tau_j & = \tau_k \tau_i \tau_j & &\text{for all } i > j \geq k, \quad i - k \geq 2, \label{eq:knuth_left}
\\ (\tau_i + \tau_{i+1}) \tau_{i+1} \tau_i & = \tau_{i+1} \tau_i (\tau_i + \tau_{i+1}) & & \text{for all } i. \label{eq:weak_relation}
\end{align}
\end{subequations}
The (strong) Knuth relations are formed by removing the requirement $i - k \geq 2$ from~\eqref{eq:knuth_relations}.

\begin{thm}[{\cite{FG98}}]
Let $\{\tau_i \colon \field[\mcP] \to \field[\mcP] \}_{i=1}^{\infty}$ denote a sequence of linear operators that satisfy the weak Knuth relations.
Then the Schur functions $\{ s_{\lambda}(\tau_1, \tau_2, \ldots) \}_{\lambda \in \mcP}$ commute and $s_{\nu/\lambda}(\tau_1, \tau_2, \ldots) = \sum_{\mu} c_{\lambda,\mu}^{\nu} s_{\mu}(\tau_1, \tau_2, \ldots)$, where $c_{\lambda,\mu}^{\nu}$ are the usual Littlewood--Richardson coefficients.
\end{thm}

\subsection{Free fermions and Schur operators}

We describe the free-fermion presentation of the (dual) canonical Grothendieck polynomials from~\cite{IMS22}.
For more details, we refer the reader to~\cite{AZ13,kac90,MJD00}.
Let $\field$ be a field of characteristic $0$.
We consider the unital associative $\field$-algebra of \defn{free-fermions} $\mcA$ is generated by $\{\psi_n, \psi_n^* \mid n \in \ZZ\}$ with relations
\[
\psi_m \psi_n + \psi_n \psi_m = \psi_m^* \psi_n^* + \psi_n^* \psi_m^* = 0,
\qquad\qquad
\psi_m \psi_n^* + \psi_n^* \psi_m = \delta_{m,n},
\]
known as the canonical anti-commuting relations.
This is a Clifford algebra arising from the canonical bilinear form of an infinite dimensional vector space $V$ with a basis $\{v_i\}_{i \in \ZZ}$ with its (restricted) dual space $V^*$ spanned by $\{v_i^*\}_{i \in \ZZ}$.
As such, there is an anti-algebra involution on $\mcA$ defined by $\psi_n \leftrightarrow \psi_n^\ast$; that is $(xy)^\ast=y^\ast x^\ast$ for any $x,y\in \mathcal{A}$.
We will also define the fields
\[
\psi(z) = \sum_{n \in \ZZ} \psi_n z^n,
\qquad\qquad
\psi^*(w) = \sum_{n \in \ZZ} \psi_n w^{-n}.
\]

The \defn{current operators}\footnote{We have omitted the normal ordering as we will not consider the current operator $a_0$ in this text. See, \textit{e.g.},~\cite[Sec.~2]{AZ13} and~\cite[Sec.~5.2]{MJD00} for more details.} are defined as
\[
a_k := \sum_{i \in \ZZ} \psi_i \psi_{i+k}^*,
\]
and satisfy the Heisenberg algebra relations and duality
\[
[a_m, a_k] = m \delta_{m,-k},
\qquad\qquad
a_k^* = a_{-k}.
\]
We will use the \defn{Hamiltonian operator}
\[
H(\xx / \yy) := \sum_{k > 0} \frac{p_k(\xx/\yy)}{k} a_k,
\]
and the corresponding \defn{half vertex operator} $e^{H(\xx/\yy)}$.
These satisfy the relations (see, \textit{e.g.},~\cite[Eq.~(17), Eq.~(18)]{Iwao21})
\begin{subequations}
\label{eq:eH_relations}
\begin{align}
\label{eq:eH_commute}
e^{H(\xx/\yy)} \psi_k e^{-H(\xx/\yy)}& = \sum_{i=0}^{\infty} h_i(\xx/\yy) \psi_{k-i},
\\
\label{eq:seH_commute}
e^{-H(\xx/\yy)} \psi^*_k e^{H(\xx/\yy)}& =  \sum_{i=0}^{\infty} h_i(\xx/\yy) \psi^*_{k+i}.
\end{align}
\end{subequations}
Note that $-H(\xx/\yy) = H(\yy/\xx)$.
Let $H^*(\xx/\yy) = (H(\xx/\yy))^*$ denote the \defn{dual Hamiltonian operator}.
For $k > 0$, we have the relations (see, \textit{e.g.},~\cite[Eq.~(2.4)]{AZ13})
\begin{equation}
\label{eq:ea_commute}
[a_{-k}, e^{H(t)}] = -t^k e^{H(t)},
\qquad\qquad
[e^{H^*(t)}, a_k] = -t^k e^{H^*(t)},
\end{equation}
and $[a_k, e^{H(t)}] = [e^{H^*(t)}, a_{-k}] = 0$.
We will also use the \defn{transposed Hamiltonian operator}
\[
J(\xx/\yy) := \omega H(\xx / \yy) = -H(-\xx / -\yy).
\]
Therefore, we can write
\begin{equation}
\label{eq:exp_current_sum}
e^{H(\xx / \yy)} = \sum_{k=0}^{\infty} H_k(a_1, a_2, \ldots) p_k(\xx / \yy),
\qquad\qquad
e^{J(\xx / \yy)} = \sum_{k=0}^{\infty} E_k(a_1, a_2, \ldots) p_k(\xx / \yy).
\end{equation}

We will consider the spinor representation $\mcF$ of semi-infinite wedge products subject to a finiteness condition, but we will not present this here in detail (see, \textit{e.g.},~\cite{kac90,KRR13,MJD00} for more information).
This is sometimes referred to as \defn{fermionic Fock space}.
Instead, we will realize $\mcF$ as the cyclic $\mcA$-representation generated by the \defn{vacuum vector} $\ket{0}$ that satisfies the relations
\[
\psi_n \ket{0} = \psi_m^* \ket{0} = 0, \qquad\qquad n < 0, \quad m \geq 0.
\]
Therefore, we can describe the basis as the vectors
\[
\psi_{n_1} \psi_{n_2} \cdots \psi_{n_r} \psi_{m_1}^* \psi_{m_2}^* \cdots \psi_{m_s}^* \ket{0},
\qquad
(r,s \geq 0, n_1 > \cdots n_r \geq 0 > m_s > \cdots > m_1).
\]
We define the \defn{shifted vacuum vectors} as
\[
\ket{m} = \begin{cases} \psi_{m-1} \dotsm \psi_0 \ket{0} & \text{if } m \geq 0, \\ \psi_m^* \dotsm \psi_{-1}^* \ket{0} & \text{if } m < 0, \end{cases}
\qquad\qquad
\bra{m} = \begin{cases} \bra{0} \psi_0^* \dotsm \psi_{m-1}^* & \text{if } m \geq 0, \\ \bra{0} \psi_{-1} \dotsm \psi_m & \text{if } m < 0. \end{cases}
\]
Note that
\[
e^{H(\xx/\yy)} \ket{m} = \ket{m},
\qquad\qquad
\bra{m} e^{H^*(\xx/\yy)} = \bra{m},
\]
for all $m$.
For finitely many (noncommutative) expressions $\Psi_1,\dots,\Psi_\ell$, we will use the notation
\[
\prod_{1\leq i\leq \ell}^{\rightarrow} \Psi_i=\Psi_1\Psi_2\cdots \Psi_\ell,
\qquad\qquad
\prod_{1\leq i\leq \ell}^{\leftarrow} \Psi_i=\Psi_\ell\cdots \Psi_2\Psi_1,
\]
to indicate the order of multiplication.
We will use the vectors
\begin{align*}
\ket{\lambda}^{\sigma}_{[\bal,\bbe]} 
& := \prod^{\rightarrow}_{1 \leq i \leq \ell} \left( e^{-H(A_{\sigma_i-1})} \psi_{\lambda_i-i} e^{H(\beta_i)} e^{H(A_{\sigma_i-1})} \right) \ket{-\ell},
\\
\ket{\lambda}_{\sigma}^{[\bal,\bbe]} 
& :=
\prod^{\rightarrow}_{1 \leq i \leq \ell} \left( e^{H^*(A_{\sigma_i})} \psi_{\lambda_i-i} e^{-H^*(\beta_i)} e^{-H^*(A_{\sigma_i})} \right) 
{
e^{H^\ast(A_{\sigma_{\ell}})}}
\ket{-\ell},
\end{align*}
where $\sigma = (\sigma_1, \sigma_2, \dotsc, \sigma_{\ell}) \in \ZZ_{\geq 0}^{\ell}$ and $A_k = -\bal_k = (-\alpha_1, \dotsc, - \alpha_k)$.
When $\sigma = \lambda$, we simply write $\ket{\lambda}_{[\bal,\bbe]} := \ket{\lambda}^{\lambda}_{[\bal,\bbe]}$ and $\ket{\lambda}^{[\bal,\bbe]} := \ket{\lambda}_{\lambda}^{[\bal,\bbe]}$.
We use the notation
\[
\ket{\lambda}_{[\bbe]} :=  \ket{\lambda}^\sigma_{[0,\bbe]},
\qquad\qquad
\ket{\lambda}^{[\bbe]} := \ket{\lambda}_\sigma^{[0,\bbe]},
\qquad\qquad
\ket{\lambda} := \ket{\lambda}_{\sigma}^{[0,0]} = \ket{\lambda}^{\sigma}_{[0,0]}.
\]
(note these are independent of $\sigma$) for brevity.
We will typically restrict ourselves to the subspace $\mcF^0$, which we describe as the span of either of the bases~\cite[Thm.~3.10]{IMS22}
\[
\{\ket{\lambda}_{[\bal,\bbe]}\}_{\lambda \in \mcP},
\qquad\qquad
\{\ket{\lambda}^{[\bal,\bbe]}\}_{\lambda \in \mcP}.
\]

There is also the dual $\mcA$-representation $\mcF^*$, which has a canonical bilinear pairing called the \defn{vacuum expectation value} that satisfies
\[
\braket{k}{m} = \delta_{km},
\qquad\qquad
(\bra{w} X) \ket{v} = \bra{w} (X \ket{v}),
\]
for all $k,m \in \ZZ$, $X \in \mcA$, $\bra{w} \in \mcF^*$, and $\ket{v} \in \mcF$.
We abbreviate $\langle X \rangle = \bra{0} X \ket{0}$, and note that $\ket{k}^* = \bra{k}$.
Define
\[
{}_{[\bal,\bbe]}\bra{\lambda}_{\sigma} = (\ket{\lambda}_{\sigma}^{[\bal,\bbe]})^*,
\qquad\qquad
{}^{[\bal,\bbe]}\bra{\lambda}^{\sigma} = (\ket{\lambda}^{\sigma}_{[\bal,\bbe]})^*,
\]
and similar abbreviations as above.
With respect to this inner product, we have the orthonormal bases~\cite[Thm.~3.10]{IMS22}
\begin{equation}
\label{eq:orthonormal_basis}
{}_{[\bal,\bbe]}\braket{\lambda}{\mu}_{[\bal,\bbe]} = {}^{[\bal,\bbe]}\braket{\lambda}{\mu}^{[\bal,\bbe]} = \delta_{\lambda\mu}.
\end{equation}
Moreover, there is an isomorphism from $\mcF^0$ to symmetric functions defined by $\ket{v} \mapsto \bra{0} e^{H(\xx/\yy)} \ket{v}$, which satisfies~\cite[Cor.~4.2, Eq.~(4.1)]{IMS22}
\begin{equation}
\label{eq:free_fermion_grothendiecks}
\G_{\lambda \ds \mu}(\xx; \bal, \bbe) = {}^{[\bal,\bbe]} \bra{\mu} e^{H(\xx)} \ket{\lambda}^{[\bal,\bbe]},
\qquad\qquad
\dG_{\lambda / \mu}(\xx; \bal, \bbe) = {}_{[\bal,\bbe]} \bra{\mu} e^{H(\xx)} \ket{\lambda}_{[\bal,\bbe]}.
\end{equation}
From~\cite[Thm.~4.3]{IMS22}, we also have
\begin{equation}
\label{eq:dual_free_fermion_grothendiecks}
\G_{\lambda' \ds \mu'}(\xx; \bal, \bbe) = {}^{[\bal,\bbe]} \bra{\mu} e^{J(\xx)} \ket{\lambda}^{[\bal,\bbe]},
\qquad\qquad
\dG_{\lambda' / \mu'}(\xx; \bal, \bbe) = {}_{[\bal,\bbe]} \bra{\mu} e^{J(\xx)} \ket{\lambda}_{[\bal,\bbe]}.
\end{equation}

\subsection{Particle processes}
\label{sec:particle_processes}

We describe the four different versions of the discrete \defn{totally asymmetric simple exclusion process} (TASEP) given in~\cite{DW08}.
All of these processes will be considered using a bosonic presentation following~\cite{DW08}, where they are given by particles labeled $1, \dotsc, \ell$ that can occupy the same sites on the lattice $\ZZ_{\geq 0}$.
The position of these particles will remain in order and only move to the right (\textit{i.e.}, increase the value of their position), and so we can index states by partitions $\lambda$, where $\lambda_i$ corresponds to the position of the $i$-th particle.
To obtain a fermionic presentation and justifying calling this TASEP, simply move the $i$-th particle from position $\lambda_i$ to $\lambda_i - i$, and hence, we can identity the shape $\lambda = \emptyset$ with the step initial condition:
\[
\begin{tikzpicture}[scale=0.7,baseline=0cm]
\draw (2.7, 0) node {$\cdots$};
\draw[-] (0,0) -- (2.2,0);
\foreach \x in {0,...,2} {
  \draw[-] (\x,-.1) -- (\x,+.1);
  \draw (\x, 0) node[below] {\tiny $\x$};
}
\fill[darkred] (0,0.1) circle (0.1);
\fill[blue] (0,0.3) circle (0.1);
\fill[UQpurple] (0,0.5) circle (0.1);
\fill[dgreencolor] (0,0.7) circle (0.1);
\end{tikzpicture}
\; \text{(bosonic)}
\quad \longleftrightarrow \quad
\ket{\emptyset} = \ket{0} =
\begin{tikzpicture}[scale=0.7,baseline=0cm]
\draw (-4.7, 0) node {$\cdots$};
\draw[-] (-4.2,0) -- (2.2,0);
\foreach \x in {-4,...,2} {
  \draw[-] (\x,-.1) -- (\x,+.1);
  \draw (\x, 0) node[below] {\tiny $\x$};
}
\draw (2.7, 0) node {$\cdots$};
\fill[darkred] (-1,0.1) circle (0.1);
\fill[blue] (-2,0.1) circle (0.1);
\fill[UQpurple] (-3,0.1) circle (0.1);
\fill[dgreencolor] (-4,0.1) circle (0.1);
\end{tikzpicture}
\; \text{(fermionic)}.
\]
Unless otherwise noted, we will consider our particle configurations using the bosonic description.

\begin{ex}
We identify $\lambda=(3,3,1,0)$ with the (bosonic) particle distribution
\[
\ytableausetup{boxsize=1em}
\ydiagram{3,3,1} \qquad \longleftrightarrow \qquad
\begin{tikzpicture}[scale=0.7,baseline=0cm]
\draw[-] (-0.2,0) -- (4.2,0);
\foreach \x in {0,...,4} {
  \draw[-] (\x,-.1) -- (\x,+.1);
  \draw (\x, 0) node[below] {\tiny $\x$};
}
\draw (4.7, 0) node {$\cdots$};
\fill[darkred] (3,0.1) circle (0.1);
\fill[blue] (3,0.3) circle (0.1);
\fill[UQpurple] (1,0.1) circle (0.1);
\fill[dgreencolor] (0,0.1) circle (0.1);
\end{tikzpicture}
\]
with the first and second particles are in position $3$, the third particle is in position 1, and the fourth particle is in position $0$.
In terms of the fermionic presentation, we have
\[
\ket{\lambda} = \psi_2 \psi_1 \psi_{-2} \psi_{-4} \ket{-5} = \psi_2 \psi_1 \psi_{-2} \ket{-4} =
\begin{tikzpicture}[scale=0.7,baseline=0cm]
\draw[-] (-4.2,0) -- (3.2,0);
\foreach \x in {-4,...,3} {
  \draw[-] (\x,-.1) -- (\x,+.1);
  \draw (\x, 0) node[below] {\tiny $\x$};
}
\draw (-4.7, 0) node {$\cdots$};
\draw (3.7, 0) node {$\cdots$};
\foreach \x/\c in {2/darkred, 1/blue, -2/UQpurple, -4/dgreencolor} {
  \draw[color=\c,thick] (\x+0.15,0.15) -- (\x-0.15,-0.15);
  \draw[color=\c,thick] (\x-0.15,0.15) -- (\x+0.15,-0.15);
}
\end{tikzpicture}
\]
\end{ex}

The $j$-th particle at time $i$ will attempt to move $w_{ji}$ (which we take as a random variable) steps to the right according to either the geometric or Bernoulli distribution
\begin{gather*}
\prob_{Ge}(w_{ji} = k) = (1 - \pi_j x_i) (\pi_j x_i)^k \quad (k \in \ZZ_{\geq 0}),
\\
\prob_{Be}(w_{ji} = 1) = \frac{\rho_i x_j}{1 + \rho_j x_i}
\quad \text{and} \quad 
\prob_{Be}(w_{ji} = 0) = (1 + \rho_j x_i)^{-1},
\end{gather*}
respectively, where $\pi_j x_i \in (0, 1)$ and $\rho_j x_i > 0$.
We can assemble all of these random variables, ranging over all particles and time, into a (random) matrix $W = [w_{ji}]_{i,j}$.
Since we are working in discrete time, we need a rule to determine the behavior when two particles decide to move simultaneously that results in a conflict.
There are two natural ways to resolve this.
\begin{description}
\item[\defn{Pushing}] The larger particle pushes the smaller particle.
\item[\defn{Blocking}] The smaller particle blocks the larger particle.
\end{description}

For the geometric (resp.\ Bernoulli) distribution, we will update the particles from largest-to-smallest (resp.\ smallest-to-largest).
As such, we obtain four different variations of discrete TASEP, which we organize using the same indexing as~\cite{DW08}:
\begin{enumerate}[(A)]
\item Geometric distribution with pushing behavior.
\item Bernoulli distribution with blocking behavior.
\item Geometric distribution with blocking behavior.
\item Bernoulli distribution with pushing behavior.
\end{enumerate}
Table~\ref{table:cases_summary} provides a summary of these cases and Figure~\ref{fig:jump_behavior_example} gives an example of the blocking versus pushing behavior.
We also give a recursive description of these processes following~\eqref{eq:particle_transitions}.
Here we give the equivalent bosonic version.
We introduce $
\Omega_\ell^b:=\{ (z_1,z_2,\dots,z_\ell) \in \mathbb{Z}_{ \geq 0}^\ell: z_1 \geq z_2 \geq \cdots \geq z_\ell \}
$
for $\ell \ge 2$.
We introduce four types of evolution of particles $Y_t^{\mathrm{A}}$, $Y_t^{\mathrm{B}}$, $Y_t^{\mathrm{C}}$ and $Y_t^{\mathrm{D}}$ on $\Omega_\ell^b$.
\begin{dfn}
\label{defn:recursive_description}
The four particle processes whose positions at time $t+1$ are given recursively as
\begin{align}
Y_{t+1}^\mathrm{A}(\ell)&=Y_t^\mathrm{A}(\ell)+\xi^\mathrm{A}(\ell,t+1), \ \ \
Y_{t+1}^\mathrm{B}(1)=Y_t^\mathrm{B}(1)+\xi^\mathrm{B}(1,t+1), \nonumber \\
Y_{t+1}^\mathrm{C}(1)&=Y_t^\mathrm{C}(1)+\xi^\mathrm{C}(1,t+1), \ \ \
Y_{t+1}^\mathrm{D}(\ell)=Y_t^\mathrm{D}(\ell)+\xi^\mathrm{D}(\ell,t+1), \nonumber
\end{align}
and
\begin{subequations}
\begin{align}
Y_{t+1}^\mathrm{A}(k) &= \max(Y_t^\mathrm{A}(k), Y_{t+1}^\mathrm{A}(k+1))+\xi^\mathrm{A}(k,t+1), \label{eq:recursive_A_defn}
\\
Y_{t+1}^\mathrm{B}(k) &= \min(Y_t^\mathrm{B}(k)+\xi^\mathrm{B}(k,t+1) ,Y_{t+1}^\mathrm{B}(k-1)),
\\
Y_{t+1}^\mathrm{C}(k) &= \min(Y_t^\mathrm{C}(k)+\xi^\mathrm{C}(k,t+1),Y_{t}^\mathrm{C}(k-1)),
\\
Y_{t+1}^\mathrm{D}(k) &= \max(Y_t^\mathrm{D}(k)+\xi^\mathrm{D}(k,t+1) ,Y_{t+1}^\mathrm{D}(k+1)),
\end{align}
\end{subequations}
for $k=1,\dots,\ell-1$ for Case (A), (D)
and $k=2,\dots,\ell$ for Case (B), (C),
where $\xi^\mathrm{D}(k,t+1)$ and $\xi^\mathrm{C}(k,t+1)$ (resp/\ $\xi^\mathrm{B}(k,t+1)$ and $\xi^\mathrm{D}(k,t+1)$) are random variables has distribution $\prob_{Ge}$ (resp.\ $\prob_{Be}$).
\end{dfn}

For a sequence of integers $\lambda=(\lambda_1,\lambda_2,\dots,\lambda_\ell)$ $(\lambda_1 \ge \lambda_2 \ge \cdots \ge \lambda_\ell)$,
we write $Y=\lambda$ if $Y \in \Omega_\ell^b$ satisifies $Y(j)=\lambda_j$, $j=1,2,\dots,\ell$.

The bosonic version of Theorem~\ref{thm:transition_prob} is that
the transition probabilities of the four particle systems $Y=Y^\mathrm{A},Y^{\mathrm{B}}, Y^\mathrm{C}, Y^\mathrm{D}$
from the initial configuration $Y_0=\mu$ at time zero to the final configuration $Y_n=\lambda$ at time $n$ are given by
\begin{align*}
\mathsf{P}(Y^\mathrm{A}_n=\lambda|Y^\mathrm{A}_0=\mu)
 & =\prob_{A,n}(\lambda | \mu), \ \mathsf{P}(Y^\mathrm{B}_n=\lambda|Y^\mathrm{B}_0=\mu)
 =\prob_{B,n}(\lambda | \mu), \\
\mathsf{P}(Y^\mathrm{C}_n=\lambda|Y^\mathrm{C}_0=\mu)
 & =\prob_{C,n}(\lambda | \mu), \ \mathsf{P}(Y^\mathrm{D}_n=\lambda|Y^\mathrm{D}_0=\mu)
 =\prob_{D,n}(\lambda | \mu).
\end{align*}
Recall that $\prob_{A,n}(\lambda | \mu)$, $\prob_{B,n}(\lambda | \mu)$, $\prob_{C,n}(\lambda | \mu)$, $\prob_{D,n}(\lambda | \mu)$
are Grothendieck polynomials with overall factors multiplied, explicitly given in Theorem~\ref{thm:transition_prob}.
By abuse of notation, we also denote transition probabilities using these notations.
In Section~\ref{sec:operator_dynamics}, where we give a proof, we regard these notations as transition probabilities.

\begin{table}
\begin{center}
\begin{tabular}{c|c|c|}
\multicolumn{1}{c}{} & \multicolumn{1}{c}{Pushing} & \multicolumn{1}{c}{Blocking} \\ \cline{2-3}
Geometric & A & C \\ \cline{2-3}
Bernoulli & D & B \\ \cline{2-3}
\end{tabular}
\end{center}
\medskip
\caption{Summary of the four cases of discrete TASEP that we consider in this paper.}
\label{table:cases_summary}
\end{table}

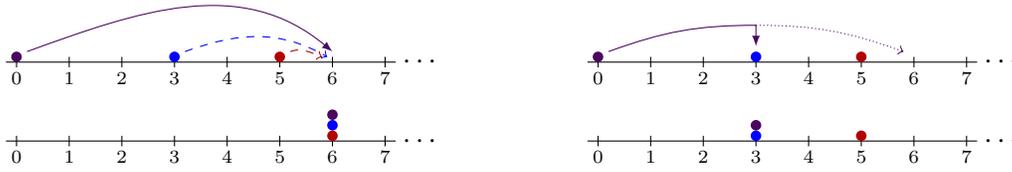
\begin{figure}
\[
\begin{tikzpicture}[scale=0.7,baseline=0]
\foreach \y [count=\t from 0] in {0,-1.5} {
  \draw[-] (-0.2,\y) -- (7.2,\y);
  \foreach \x in {0,...,7} {
    \draw[-] (\x,\y-.1) -- (\x,\y+.1);
    \draw (\x, \y) node[below] {\tiny $\x$};
  }
  \draw (7.7, \y) node {$\cdots$};
}
\fill[darkred] (5,0.1) circle (0.1);
\fill[blue] (3,0.1) circle (0.1);
\fill[UQpurple] (0,0.1) circle (0.1);
\draw[->,UQpurple,>=latex] (0.2,.2) to[out=20,in=140] (6,.2);
\draw[->,blue,dashed] (3.2,.2) to[out=20,in=150] (5.9,.1);
\draw[->,darkred,dashed] (5.2,.2) to[out=20,in=160] (5.8,.1);
\fill[darkred] (6,0.1-1.5) circle (0.1);
\fill[blue] (6,0.3-1.5) circle (0.1);
\fill[UQpurple] (6,0.5-1.5) circle (0.1);
\end{tikzpicture}
\hspace{50pt}
\begin{tikzpicture}[scale=0.7,baseline=0]
\foreach \y [count=\t from 0] in {0,-1.5} {
  \draw[-] (-0.2,\y) -- (7.2,\y);
  \foreach \x in {0,...,7} {
    \draw[-] (\x,\y-.1) -- (\x,\y+.1);
    \draw (\x, \y) node[below] {\tiny $\x$};
  }
  \draw (7.7, \y) node {$\cdots$};
}
\fill[darkred] (5,0.1) circle (0.1);
\fill[blue] (3,0.1) circle (0.1);
\fill[UQpurple] (0,0.1) circle (0.1);
\draw[->,UQpurple,>=latex] (0.2,.2) to[out=20,in=180] (3,.7) -- (3,.3);
\draw[->,UQpurple,densely dotted] (3,.7) to[out=0,in=160] (5.8,.2);
\fill[darkred] (5,0.1-1.5) circle (0.1);
\fill[blue] (3,0.1-1.5) circle (0.1);
\fill[UQpurple] (3,0.3-1.5) circle (0.1);
\end{tikzpicture}
\]
\caption{Examples of the third particle making a jump of $6$ steps with the pushing (left) and blocking (right) behaviors.}
\label{fig:jump_behavior_example}
\end{figure}

Let us give an equivalent two-dimensional description.
Let $G(j, i)$ denote the position of the $j$-th particle at time $i$ in the bosonic formulation.
For the pushing behavior, we can further realize it as a directed last-passage percolation model on $W$.
To see the last-passage percolation, define
\[
G(k, n) = \max_{\Pi} \sum_{(j,i) \in \Pi} w_{ji},
\qquad\qquad
\bG(n) = \bigl(G(1, n), G(2, n), \dotsc, G(\ell, n) \bigr),
\]
where the maximum is taken over a certain set of paths from $(\ell, 1)$ to $(k, n)$.
The paths are given in the natural matrix coordinates $(r, c)$ being the $r$-th row and $c$-th column.
For the geometric distribution, the paths use unit steps to the right or up; that is, starting at position $(j_a, i_a)$, then either $(j_{a+1}, i_{a+1}) = (j_a, i_a+1), (j_a-1, i_a)$.
Translating this into the update rule described in~\cite{DW08} (\textit{cf}.~\eqref{eq:recursive_A_defn}), we have
\[
G(j,i) = \max(G(j,i-1), G(j+1,i)) + w_{ji}.
\]
Note $G(j,i)$ corresponds to $Y_i(j)$ and $w_{ji}$ corresponds to $\xi(j,i)$ in the previous bosonic particle process description.

On the other hand, the Bernoulli distribution uses paths such that for every time $w_{j_a i_a} = 1$, we must have the next step move right $(j_{a+1}, i_{a+1}) = (j_a, i_a+1)$, and we also allow paths to end at $(k', n)$ for some $k \leq k' \leq \ell$.
Then from a simple argument using conditional probability, we see that the position of the particles at time $n$ is given by $\bG(n)$ (see, \textit{e.g.},~\cite{DW08,Johansson00,MS20}).
To obtain the positions of the particles for the blocking behavior, this becomes a first-passage percolation model as we replace $\max$ by $\min$, but we have to make some other changes.
The first is that we shift the indices so that the $w_{ji}$ correspond to the weights on the horizontal edges $(i,j-1) \to (i, j)$ with the other edges being SE diagonal (resp.\ vertical) edges with weight $0$ for the geometric (resp.\ Bernoulli) case.
The other is that we instead start at $(0, 0)$.

\begin{ex}
Consider $\ell = 3$ and $n = 4$.
Then the correspondence between a random matrix and the motion of particles in Case~A, resp.\ Case~C, is given by
\[
[w_{ji}]_{ji} = \begin{bmatrix}
2 & 1 & 0 & 0 \\
0 & 1 & 3 & 0 \\
1 & 1 & 0 & 2 \\
\end{bmatrix},
\quad
\begin{tikzpicture}[scale=0.7,baseline=-1.75cm]
\foreach \y [count=\t from 0] in {0,-1.2,...,-4.8} {
  \draw[-] (-0.2,\y) -- (6.2,\y);
  \foreach \x in {0,...,6} {
    \draw[-] (\x,\y-.1) -- (\x,\y+.1);
    \draw (\x, \y) node[below] {\tiny $\x$};
  }
  \draw (6.7, \y) node {$\cdots$};
  \draw (-1, \y) node {$t = \t$};
}
\fill[darkred] (0,0.1) circle (0.1);
\fill[blue] (0,0.3) circle (0.1);
\fill[UQpurple] (0,0.5) circle (0.1);
\fill[darkred] (3,0.1-1.2) circle (0.1);
\fill[blue] (1,0.1-1.2) circle (0.1);
\fill[UQpurple] (1,0.3-1.2) circle (0.1);
\draw[->,darkred] (0.2,.1-1.2) -- (0.8,.1-1.2);
\draw[->,blue] (0.2,.3-1.2) -- (0.8,.3-1.2);
\fill[darkred] (4,0.1-2.4) circle (0.1);
\fill[blue] (3,0.1-2.4) circle (0.1);
\fill[UQpurple] (2,0.1-2.4) circle (0.1);
\draw[->,blue] (1.2,.1-2.4) -- (1.8,.1-2.4);
\fill[darkred] (6,0.1-3.6) circle (0.1);
\fill[blue] (6,0.3-3.6) circle (0.1);
\fill[UQpurple] (2,0.1-3.6) circle (0.1);
\draw[->,darkred] (4.2,.1-3.6) -- (5.8,.1-3.6);
\fill[darkred] (6,0.1-4.8) circle (0.1);
\fill[blue] (6,0.3-4.8) circle (0.1);
\fill[UQpurple] (4,0.1-4.8) circle (0.1);
\end{tikzpicture}
\quad
\begin{tikzpicture}[scale=0.7,baseline=-1.75cm]
\foreach \y [count=\t from 0] in {0,-1.2,...,-4.8} {
  \draw[-] (-0.2,\y) -- (4.2,\y);
  \foreach \x in {0,...,4} {
    \draw[-] (\x,\y-.1) -- (\x,\y+.1);
    \draw (\x, \y) node[below] {\tiny $\x$};
  }
  \draw (4.7, \y) node {$\cdots$};
  \draw (-1, \y) node {$t = \t$};
}
\fill[darkred] (0,0.1) circle (0.1);
\fill[blue] (0,0.3) circle (0.1);
\fill[UQpurple] (0,0.5) circle (0.1);
\fill[darkred] (2,0.1-1.2) circle (0.1);
\fill[blue] (0,0.1-1.2) circle (0.1);
\fill[UQpurple] (0,0.3-1.2) circle (0.1);
\draw[->,UQpurple] (0.5,.3-1.2) -- (0.2,.3-1.2);
\fill[darkred] (3,0.1-2.4) circle (0.1);
\fill[blue] (1,0.1-2.4) circle (0.1);
\fill[UQpurple] (0,0.1-2.4) circle (0.1);
\draw[->,UQpurple] (0.5,.1-2.4) -- (0.2,.1-2.4);
\fill[darkred] (3,0.1-3.6) circle (0.1);
\fill[blue] (3,0.3-3.6) circle (0.1);
\fill[UQpurple] (0,0.1-3.6) circle (0.1);
\draw[->,blue] (3.5,.3-3.6) -- (3.2,.3-3.6);
\fill[darkred] (3,0.1-4.8) circle (0.1);
\fill[blue] (3,0.3-4.8) circle (0.1);
\fill[UQpurple] (2,0.1-4.8) circle (0.1);
\end{tikzpicture}
\]
\end{ex}

First note that all of these transition probabilities $\prob_X(\lambda | \mu) := \prob_X(\bG(1) = \lambda | \bG(0) = \mu)$ in Case~X are independent of the value of $\ell$ provided $\ell(\lambda) \leq \ell$.
Indeed, the $j$-th particle for $j > \ell(\lambda)$ must be fixed in place, which occurs with probability $1$.
Note that by taking $\ell \to \infty$, we can effectively ignore the value of $\ell$ if desired and identify states with elements of the fermionic Fock space $\mcF$ (considered with the shifted positions $\lambda_j - j$).
Furthermore, the step initial condition becomes $\ket{0}$, and this is sometimes known as the Dirac sea.

\section{Schur operators for canonical Grothendieck polynomials}
\label{sec:noncomm_operators}

We begin by briefly reviewing the Fomin--Greene theory of symmetric functions in noncommutative variables.
Let $\field [\mcP]$ denote the $\field$-module with basis indexed by $\mcP$, the set of all partitions.

\subsection{Noncommutative blocking operators}

We denote $\schur_i \colon \field[\mcP] \to \field[\mcP]$ the $i$-th (row) \defn{Schur operator} that adds a box to the $i$-th row of a partition $\lambda$ if $\lambda_i < \lambda_{i-1}$ (that is, we can add the box and obtain a partition) and is $0$ otherwise.
The Schur operators satisfy the Knuth relations.

We define the linear operator $U_i^{(\bal, \bbe)}$ by
\[
U_i^{(\bal, \bbe)} := \schur_i + \Theta_i,\quad \mbox{where}\quad
\Theta_i \cdot \lambda := \begin{cases}
-\alpha_{\lambda_i} \lambda & \text{if } \lambda_i < \lambda_{i-1}, \\
\beta_{i-1} \lambda & \text{if } \lambda_i = \lambda_{i-1},
\end{cases}
\]
for any $\lambda \in \mcP$.
We consider $\lambda_0 = \infty$ and $\alpha_0 = 0$ (although our proofs could have $\alpha_0$ be an arbitrary parameter).
When there is no ambiguity in the parameters, we will simply write $U_i := U_i^{(\bal,\bbe)}$.
When $\bal = 0$ and $\bbe = \beta$, the operators $\{U_i^{(0, \beta)}\}_{i=1}^{\infty}$ are the operators introduced in~\cite[Sec.~6]{Iwao19}.

\begin{ex}
\label{ex:not_strong_knuth}
We compute
\begin{align*}
\ytableausetup{boxsize=0.7em}
U_2 U_1 U_1 \cdot \ydiagram{1,1} & = U_2 U_1 \left( \ydiagram{2,1} - \alpha_1 \ydiagram{1,1} \right)
\\ & = U_2 \left( \ydiagram{3,1} - (\alpha_1 + \alpha_2) \ydiagram{2,1} + \alpha_1^2 \ydiagram{1,1} \right)
\\ & = \ydiagram{3,2} - \alpha_1 \ydiagram{3,1} - (\alpha_1 + \alpha_2) \ydiagram{2,2} + \alpha_1 (\alpha_1 + \alpha_2) \ydiagram{2,1} + \alpha_1^2 \beta_1 \ydiagram{1,1}\,,
\\
U_1 U_2 U_1 \cdot \ydiagram{1,1} & = U_1 U_2 \left( \ydiagram{2,1} - \alpha_1 \ydiagram{1,1} \right)
\\ & = U_1 \left( \ydiagram{2,2} - \alpha_1 \ydiagram{2,1} - \beta_1 \alpha_1 \ydiagram{1,1} \right)
\\ & = \ydiagram{3,2} - \alpha_1 \ydiagram{3,1} - \alpha_2 \ydiagram{2,2} + (\alpha_1 \alpha_2 - \alpha_1 \beta_1) \ydiagram{2,1} + \alpha_1^2 \beta_1 \ydiagram{1,1}\,.
\end{align*}
As a consequence, the operators $U_i$ do not satisfy the (strong) Knuth relations.
However, we verify they satisfy the weak Knuth relation~\eqref{eq:weak_relation} $U_1 U_2 U_1 + U_2 U_2 U_1 = U_2 U_1 U_1 + U_2 U_1 U_2$ as
\begin{align*}
U_2 U_1 U_2 \cdot \ydiagram{1,1} 
& = \beta_1 \ydiagram{2,2} - \alpha_1 \beta_1 \ydiagram{2,1}- \alpha_1 \beta_1^2 \ydiagram{1,1}\,,
\\
U_2 U_2 U_1 \cdot \ydiagram{1,1} 
& = (\beta_1 - \alpha_1) \ydiagram{2,2} + \alpha_1^2 \ydiagram{2,1} - \alpha_1 \beta_1^2 \ydiagram{1,1}\,.
\end{align*}
\end{ex}

\begin{lemma}\label{lemma:weak_Knuth}
The operators $\UU = \{U_i\}_{i=1}^{\infty}$ satisfy the weak Knuth relations.
\end{lemma}

\begin{proof}
Since $[\Theta_i,\Theta_j]=0$ for any $i,j$ and
\begin{equation}\label{eq:kappa_Theta}
[\kappa_i,\Theta_j]=0\quad \text{for } i\neq j-1,j,
\end{equation}
the ``non-local commutativity" 
$U_iU_j=U_jU_i$ for $|i-j|\geq 2$ immediately follows from $\kappa_i\kappa_j=\kappa_j\kappa_i$ for $|i-j|\geq 2$.
This fact implies (\ref{eq:knuth_right}) and (\ref{eq:knuth_left}).

Define a linear operator $T_i \colon \field[\mcP] \to \field[\mcP]$ by
\[
T_i\cdot \lambda:=
\begin{cases}
0 & \text{if } \lambda_{i+1} < \lambda_{i}, \\
\kappa_i\cdot \lambda & \text{if } \lambda_{i+1} = \lambda_{i}
\end{cases}\qquad
\text{for any } \lambda\in\mcP.
\]
Note that, for any partition $\lambda$, $T_i\cdot\lambda\neq 0$ if and only if $\lambda_{i+1}=\lambda_i<\lambda_{i-1}$.
We need the following commutation relations to prove (\ref{eq:weak_relation}):
\begin{subequations}
\begin{align}
T_i \kappa_i & = \kappa_{i+1}^2T_i=0,\label{eq:comm_1}\\
\Theta_{i+1} \kappa_{i+1}T_i &= \beta_i \kappa_{i+1}T_i.\label{eq:comm_2}\\
[\kappa_i, \kappa_{i+1}] & = -\kappa_{i+1}T_i,\label{eq:comm_3}\\
\Theta_{i+1} T_i & = T_i\Theta_i, \label{eq:comm_4}\\
[\kappa_i, \Theta_{i+1}] & = \beta_i T_i - \Theta_{i+1}T_i.\label{eq:comm_5}
\end{align}
\end{subequations}
Equation (\ref{eq:comm_1}) follows from the facts that (i) $\mu:=\kappa_i\cdot \lambda$ satisfies $\mu_{i+1}<\mu_{i}$ whenever $\mu\neq 0$ and that (ii) $\nu:=T_i\cdot \lambda$ satisfies $\nu_{i}=\nu_{i+1}+1$ whenever $\nu\neq 0$ for any $\lambda\in \mcP$.
Equation (\ref{eq:comm_2}) follows from the fact that $\zeta:=\kappa_{i+1}T_i\cdot \lambda$ satisfies $\zeta_{i+1}=\zeta_i$ whenever $\zeta\neq 0$.
Equation (\ref{eq:comm_3}) is proved by noting that $\lambda_{i+1}<\lambda_{i}\Rightarrow \zeta = [\kappa_i,\kappa_{i+1}]\cdot \lambda= 0$ and $\lambda_{i+1}=\lambda_{i}\Rightarrow \zeta=[\kappa_i,\kappa_{i+1}]\cdot \lambda=-\kappa_{i+1}\kappa_i\cdot \lambda$.
To prove (\ref{eq:comm_4}), it suffices to check that $\Theta_{i+1}T_i\cdot \lambda=T_i\Theta_{i}\cdot \lambda=\alpha_{\lambda_{i}}\kappa_i\cdot \lambda$ whenever $T_i\cdot \lambda\neq 0$.
To prove (\ref{eq:comm_5}), we consider the following three cases: (a) If $T_i\cdot \lambda\neq 0$ ($\Leftrightarrow$ $\lambda_{i+1}=\lambda_i<\lambda_{i-1}$ $\Leftrightarrow$ $T_i\cdot \lambda=\kappa_i\cdot \lambda\neq 0$), we have $\kappa_i\Theta_{i+1}\cdot \lambda=\beta_i\kappa_i\cdot \lambda=\beta_iT_i\cdot \lambda$.
(b) If $\lambda_{i+1}<\lambda_i$, we have 
$\kappa_i\Theta_{i+1}\cdot \lambda=\Theta_{i+1}\kappa_i\cdot \lambda=\beta_i\kappa_i\cdot \lambda$.
(c) If $\lambda_{i}=\lambda_{i-1}$, we have 
$\kappa_i\Theta_{i+1}\cdot \lambda=\Theta_{i+1}\kappa_i\cdot \lambda=0$.
In each case, (\ref{eq:comm_5}) holds.

Thus we have
\begin{align*}
[U_i,U_{i+1}]=[\kappa_i+\Theta_i,\kappa_{i+1}+\Theta_{i+1}]
\stackrel{(\ref{eq:kappa_Theta})}{=}
[\kappa_i,\kappa_{i+1}]+[\kappa_{i},\Theta_{i+1}]
\stackrel{(\ref{eq:comm_3}),(\ref{eq:comm_5})}{=}(\beta_i-U_{i+1})T_i,
\end{align*}
\begin{align*}
[U_i,U_{i+1}]U_i
=(\beta_i-U_{i+1})T_iU_i
\stackrel{(\ref{eq:comm_1})}{=}(\beta_i-U_{i+1})T_i\Theta_i
\stackrel{(\ref{eq:comm_4})}{=}(\beta_i-U_{i+1})\Theta_{i+1}T_i.
\end{align*}
and
\begin{align*}
U_{i+1}[U_{i},U_{i+1}]
=U_{i+1}(\beta_i-U_{i+1})T_i
=(\beta_i-U_{i+1})U_{i+1}T_i
=(\beta_i-U_{i+1})(\kappa_{i+1}+\Theta_{i+1})T_i.
\end{align*}
Therefore, we obtain
\begin{align*}
[U_i+U_{i+1},U_{i+1}U_i]
&=[U_i,U_{i+1}]U_i-U_{i+1}[U_{i},U_{i+1}]
=(\beta_i-U_{i+1})\kappa_{i+1}T_i\\
&
\stackrel{(\ref{eq:comm_1})}{=}(\beta_i-\Theta_{i+1})\kappa_{i+1}T_i
\stackrel{(\ref{eq:comm_2})}{=}0,
\end{align*}
which implies (\ref{eq:weak_relation}).
\end{proof}

\begin{thm}
\label{thm:noncommutative_blocking}
We have
\begin{equation}\label{eq:to_prove}
{}^{[\bal,\bbe]} \bra{\lambda} S_{\mu}(a_1, a_2, \ldots) = {}^{[\bal,\bbe]} \bra{s_{\mu}(\UU/\bbe) \cdot \lambda}.
\end{equation}
\end{thm}

To prove Theorem~\ref{thm:noncommutative_blocking}, we need the following computations.
For simplicity, let $v(\lambda; \sigma) := {}^{[\bal,\bbe]}\bra{\lambda}^{\sigma}$.
Let $\epsilon_k$ be the $k$-th standard basis vector in $\ZZ^N$ for some $N \gg 1$.
For any subset $K = \{k_1. \dotsc, k_m \} \subseteq [N]$, let $\epsilon_K = \epsilon_{k_1} + \cdots + \epsilon_{k_m}$.

\begin{lemma}
\label{lemma:big_rectify_sigma}
For any sequences $\lambda$ and $\sigma$ (not necessarily partitions), we have
\[
v(\lambda + \epsilon_k; \sigma) = v(\lambda + \epsilon_k; \sigma + \epsilon_k) - \alpha_{\sigma_k} v(\lambda; \sigma).
\]
\end{lemma}

\begin{proof}
This follows applying to the definition of $v(\lambda + \epsilon_k; \sigma)$ the relation
\[
\psi_{\lambda_k +1 - k}^* e^{H^*(\gamma)} = e^{H^*(\gamma)} \psi_{\lambda_k + 1 - k}^* + \gamma \psi^*_{\lambda_k - k} e^{H^*(\gamma)},
\]
for $\gamma = -\alpha_{\sigma_k}$, which follows from the $\ast$ version of Equation~\eqref{eq:eH_commute}.
\end{proof}

\begin{lemma}
\label{lemma:big_rectify_lambda}
Suppose $\lambda_k = \lambda_{k-1}$ and $\sigma_k = \sigma_{k-1}$, then
\[
v(\lambda + \epsilon_k; \sigma) = \beta_{k-1} v(\lambda; \sigma).
\]
\end{lemma}

\begin{proof}
This follows from the definition of $v(\lambda; \sigma)$ and the rectification lemma~\cite[Lemma~3.6]{IMS22}.
Indeed, similar to the previous lemma, we have
\[
\psi^*_{\lambda_k - (k-1)} e^{-H^*(\beta_{k-1})} \psi^*_{\lambda_k + 1 - k} = \beta_{k-1} \psi^*_{\lambda_k - (k-1)} e^{-H^*(\beta_{k-1})} \psi^*_{\lambda_k - k}
\]
by the $\ast$ version of Equation~\eqref{eq:eH_commute} and recalling $(\psi^*_{\lambda_k - k})^2 = 0$.
\end{proof}

\begin{lemma}
\label{lemma:big_rectify_subset}
Let $\lambda$ be a partition.
If $K = \{k_1 < k_2 < \dotsc < k_m\}$, then
\[
{}^{[\bal,\bbe]} \bra{\lambda + \epsilon_K}_{\lambda} = \sum_{\mu} C_{\lambda}^{\mu} \cdot {}^{[\bal,\bbe]} \bra{\mu},
\]
where the coefficients $C_{\lambda}^{\mu}$ are given by
\[
U_{k_n} \dotsm U_{k_2} U_{k_1} \cdot \lambda = \sum_{\mu} C_{\lambda}^{\mu} \cdot \mu.
\]
\end{lemma}

\begin{proof}
This follows by repeatedly applying Lemma~\ref{lemma:big_rectify_sigma} and Lemma~\ref{lemma:big_rectify_lambda} and noting that these two cases precisely encode the action of the operator $U_k$.
\end{proof}

\begin{proof}[Proof of Theorem~\ref{thm:noncommutative_blocking}]
Because the operators $\{U_i\}_{i=1}^{\infty}$ satisfy the weak Knuth relations (Lemma~\ref{lemma:weak_Knuth}), Fomin-Greene's theorem~\cite{FG98} implies that the noncommutative Schur functions $s_{\lambda}(\UU / \bbe)$ are a commutative algebra (under the natural multiplication) as it is generated (as an algebra) by $e_{i}(\UU / \bbe)$ for $i=0,1,2,\dots$.
Therefore, to prove the theorem, it suffices to show \eqref{eq:to_prove} for the case when $s_\lambda=e_i$:
\[
{}^{[\bal,\bbe]}\bra{\lambda} E_i(a_1, a_2, \ldots) = {}^{[\bal,\bbe]}\bra{e_{i}(\UU / \bbe) \cdot \lambda}.
\]
By applying the commutator relation in~\eqref{eq:ea_commute}, we have
\[
{}^{[\bal, \bbe]}\bra{\lambda} a_i = \sum_{k=1}^{\infty} {}^{[\bal,\bbe]}\bra{\lambda + i \epsilon_k}_{\lambda} - \sum_{k=1}^{\infty} \beta_k^i \cdot {}^{[\bal,\bbe]}\bra{\lambda} = \sum_{k=1}^{\infty} {}^{[\bal,\bbe]}\bra{\lambda + i \epsilon_k}_{\lambda} - p_i(\bbe) \cdot {}^{[\bal,\bbe]}\bra{\lambda}.
\]
Next, write $E_i(p_1, p_2, \ldots; \bbe) = e_i(\xx \sqcup \bbe)$, and by noting the above equation comes from $p_i(\xx \sqcup \bbe)$ under the identification $p_i(\xx) \equiv a_i$, we have
\begin{equation}
\label{eq:U_ops_only}
{}^{[\bal,\bbe]}\bra{\lambda} E_i(a_1, a_2, \ldots; \bbe) = \sum_{\substack{K \subseteq \ZZ \\ \abs{K} = i}} {}^{[\bal,\bbe]}\bra{\lambda + \epsilon_K}_{\lambda} = {}^{[\bal,\bbe]}\bra{e_i(\UU) \cdot \lambda},
\end{equation}
where the last equality is by Lemma~\ref{lemma:big_rectify_subset}.
By expanding the plethsym for the left hand side and solving for ${}^{[\bal,\bbe]}\bra{\lambda} E_i(a_1, a_2, \ldots)$, we obtain the desired result.
\end{proof}

\subsection{Noncommutative pushing operators}

We define an operator $u_j^{(\bal,\bbe)}$ recursively as follows.
Consider a partition $\lambda$, and let $k$ be minimal such that $\lambda_k = \lambda_j$.
Let $\nu := \overline{\mu + \epsilon_j}$ be the smallest partition that contains $\mu + \epsilon_j$ (a box added to row $j$); that is, we have added a box to all rows $k \leq i \leq j$.
Then
\begin{equation}
\label{eq:uj_action_expansion}
u_j^{(\bal,\bbe)} \cdot \mu = \beta_k \cdots \beta_{j-1} \nu + \alpha_{\mu_j+1} \sum_{i=k}^j \prod_{a=k}^{i-1} (\alpha_{\mu_j+1} + \beta_a) \prod_{a=i}^{j-1} \beta_a u_i^{(\bal,\bbe)} \cdot \nu.
\end{equation}
Note that the result is well-defined in the completion (by the degree) $\field [\![ \mcP ]\!]$ since each partition only has finitely many contributions.
As we will always be looking for a specific term in this sum, working in the completion will not be consequential.
We will show the operators $\uu^{(\bal,\bbe)} = (u_1^{(\bal,\bbe)}, u_2^{(\bal,\bbe)}, \ldots)$ corresponds to separating out the action of the current operator $a_1$ on ${}_{[\bal,\bbe]} \bra{\mu}$.

\begin{lemma}
For any sequences $\mu$ and $\sigma$ (not necessarily a partition), we have
\begin{equation}
\label{eq:sigma_rectification}
{}_{[\bal,\bbe]} \bra{\mu}_{\sigma} = {}_{[\bal,\bbe]}\bra{\mu}_{\sigma + \epsilon_j} + \alpha_{\sigma_j+1} {}_{[\bal,\bbe]}\bra{\mu + \epsilon_j}_{\sigma + \epsilon_j}.
\end{equation}
\end{lemma}

\begin{proof}
We rewrite~\eqref{eq:seH_commute} as
\[
\psi_{\mu_j-j}^* = e^{-H(-\alpha_{\sigma_j+1})} \psi^*_{\mu_j-j} e^{H(-\alpha_{\sigma_j+1})} + \alpha_{\sigma_j+1} e^{-H(-\alpha_{\sigma_j+1})} \psi^*_{\mu_j+1-j} e^{H(-\alpha_{\sigma_j+1})},
\]
and so the claim follows.
\end{proof}

\begin{lemma}
For any sequences $\mu$ and $\sigma$ such that $\mu_j = \mu_{j-1}$ and $\sigma_j = \sigma_{j-1}$, we have
\begin{equation}
\label{eq:push_rectification}
{}_{[\bal,\bbe]} \bra{\mu + \epsilon_j}_{\sigma} = \beta_{j-1} \cdot {}_{[\bal,\bbe]}\bra{\mu + \epsilon_{j-1} + \epsilon_j}_{\sigma}.
\end{equation}
\end{lemma}

\begin{proof}
The claim follows from (the star version of) the rectification lemma~\cite[Lemma~3.6]{IMS22}.
\end{proof}

\begin{lemma}
\label{lemma:dual_op_action}
We have
\[
{}_{[\bal, \bbe]} \bra{\mu + \epsilon_j} = {}_{[\bal, \bbe]} \bra{u_j^{(\bal,\bbe)} \cdot \mu}.
\]
\end{lemma}

\begin{proof}
Consider $k \leq j$ minimal such that $\mu_j = \mu_k$ and $\sigma_j = \sigma_m$ for all $k \leq m \leq j$.
Recall that for $X = \{i_1, \dotsc, i_m\}$, we set $\epsilon_X = \epsilon_{i_1} + \cdots + \epsilon_{i_m}$.
Let $\nu = \mu + \epsilon_{[j,k]} = \overline{\mu + \epsilon_j}$ and $v(\mu; \sigma) := {}_{[\bal,\bbe]} \bra{\mu}_{\sigma}$.
By repeated applications of~\eqref{eq:push_rectification}, we have
\[
v(\mu + \epsilon_j; \mu) = \prod_{i=k}^{j-1} \beta_i v(\nu; \mu).
\]
Then by applying~\eqref{eq:sigma_rectification} to each $i \in [k, j]$ and then using~\eqref{eq:push_rectification}, we compute
\begin{align*}
v(\mu + \epsilon_j; \mu) & = \prod_{i=k}^{j-1} \beta_i \sum_{X \subseteq [k,j]} \alpha_{\mu_j+1}^{\abs{X}} v(\nu + \epsilon_X; \nu)
\\ & = \prod_{i=k}^{j-1} \beta_i \sum_{X \subseteq [k,j]} \alpha_{\mu_j+1}^{\abs{X}} \prod_{i \in [k,j) \setminus X} \beta_i v(\overline{\nu + \epsilon_X}; \nu)
\\ & = \prod_{i=k}^{j-1} \beta_i \cdot \nu + \prod_{a=k}^{j-1} \beta_a \cdot \alpha_{\mu_j+1} \sum_{i=k}^j \prod_{m=k}^{i-1} (\alpha_{\mu_j+1} + \beta_m) v(\overline{\nu + \epsilon_i}; \nu)
\\ & = \prod_{i=k}^{j-1} \beta_i \cdot \nu + \alpha_{\mu_j+1} \sum_{i=k}^j \prod_{m=k}^{i-1} (\alpha_{\mu_j+1} + \beta_m) \prod_{a=i}^{j-1} \beta_a v(\nu + \epsilon_i; \nu).
\end{align*}
However, this is precisely the recursion formula defining $u_j^{(\bal,\bbe)}$, and the claim follows.
\end{proof}

For brevity, we will simply write $u_j := u_j^{(\bal,\bbe)}$ until noted otherwise.

\begin{ex}
\ytableausetup{boxsize=.6em}%
We compute the action of $u_3$ on the empty partition using the operator definition:
\[
u_3 \cdot \emptyset = \beta_1 \beta_2 \mu + \alpha_1 \beta_1 \beta_2 u_1 \cdot \mu + \alpha_1 (\alpha_1 + \beta_1) \beta_2 u_2 \cdot \mu + \alpha_1 (\alpha_1 + \beta_1) (\alpha_1 + \beta_2) u_3 \cdot \mu.
\]
Now we want to compare it to ${}_{[\bal,\bbe]} \bra{\emptyset + \epsilon_3}_{\emptyset} = {}_{[\bal,\bbe]} \bra{\epsilon_3}_{\emptyset}$.
By repeated applications of~\eqref{eq:sigma_rectification} and~\eqref{eq:push_rectification}, we compute
\begin{align*}
{}_{[\bal,\bbe]} \bra{\epsilon_3}_{\emptyset} & = \beta_1 \beta_2 \Bigl( {}_{[\bal,\bbe]} \bra{\mu} + \alpha_1 \cdot {}_{[\bal,\bbe]} \bra{\mu+\epsilon_1}_{\mu} + \alpha_1 (\alpha_1 + \beta_1) \cdot {}_{[\bal,\bbe]} \bra{\mu+\epsilon_{\{1,2\}}}_{\mu}
\\ & \hspace{50pt} + \alpha_1 (\alpha_1 + \beta_1) (\alpha_1 + \beta_2) \cdot {}_{[\bal,\bbe]} \bra{\mu+\epsilon_{\{1,2,3\}}}_{\mu} \Bigr)
\end{align*}
Now by using~\eqref{eq:push_rectification} (in reverse) and letting $\mu = (1,1,1) = \ydiagram{1,1,1}$, we can rewrite
\begin{align*}
{}_{[\bal,\bbe]} \bra{\epsilon_3} & = \beta_1 \beta_2 \cdot {}_{[\bal,\bbe]} \bra{\mu} + \alpha_1 \beta_1 \beta_2 \cdot {}_{[\bal,\bbe]} \bra{\mu + \epsilon_1}
+ \alpha_1 (\alpha_1 + \beta_1) \beta_2 \cdot {}_{[\bal,\bbe]} \bra{\mu + \epsilon_2}
\\ & \hspace{20pt} + \alpha_1 (\alpha_1 + \beta_1) (\alpha_1 + \beta_2) \cdot {}_{[\bal,\bbe]} \bra{\mu + \epsilon_3}
\end{align*}
\end{ex}

\begin{ex}
\label{ex:dual_non_Knuth}
\ytableausetup{boxsize=.6em}%
Now we look at the Knuth relations of the operators $\uu^{(\bal,\bbe)}$; specifically we will consider $u_2 u_3 u_1 \cdot \emptyset$ and $u_2 u_1 u_3 \cdot \emptyset$.
For simplicity we will truncate our computation below for any partition that contains more than 6 boxes.
First, $u_1 \cdot \emptyset = \sum_{i=1}^{\infty} \prod_{j=1}^{i-1} \alpha_j (i) = \ydiagram{1} + \alpha_1 \ydiagram{2} + \alpha_1 \alpha_2 \ydiagram{3} + \cdots$.
Next, we will restrict ourselves to partitions of size at most $6$ in the computation of $u_2 u_3 u_1 \cdot \emptyset$, which is given as follows:
\begin{align*}
u_3 \cdot \ydiagram{1} & = \beta_2 \ydiagram{1,1,1} + \alpha_1 \beta_2 u_2 \ydiagram{1,1,1} + \alpha_1 (\alpha_1 + \beta_2) u_3 \ydiagram{1,1,1}
\\ & = \beta_2 \ydiagram{1,1,1} + \alpha_1 \beta_2 \left( \beta_1 \ydiagram{2,2,1} + \alpha_2 \beta_1 u_1 \ydiagram{2,2,1} + \alpha_2 (\alpha_2 + \beta_1) u_2 \ydiagram{2,2,1} \right) + \alpha_1 (\alpha_1 + \beta_2) \beta_1 \beta_2 \ydiagram{2,2,2} + \cdots
\\ & = \beta_2 \ydiagram{1,1,1} + \alpha_1 \beta_1 \beta_2 \ydiagram{2,2,1} + \alpha_1 \alpha_2 \beta_1 \beta_2 \ydiagram{3,2,1} + \alpha_1 (\alpha_1 + \beta_2) \beta_1 \beta_2 \ydiagram{2,2,2} + \cdots,
\allowdisplaybreaks \\
u_3 \cdot \ydiagram{2} & = \beta_2 \ydiagram{2,1,1} + \alpha_1 \beta_2 u_2 \cdot \ydiagram{2,1,1} + \alpha_1 (\alpha_1 + \beta_2) u_3 \cdot \ydiagram{2,1,1} + \cdots
\\ & = \beta_2 \ydiagram{2,1,1} + \alpha_1 \beta_2 \ydiagram{2,2,1} + \alpha_1 (\alpha_1 + \beta_2) \beta_2 \ydiagram{2,2,2} + \cdots,
\allowdisplaybreaks \\
u_3 \cdot \ydiagram{3} & = \beta_2 \ydiagram{3,1,1} + \alpha_1 \beta_2 \ydiagram{3,2,1} + \cdots,
\qquad\qquad
u_3 \cdot \ydiagram{4} = \beta_2 \ydiagram{4,1,1} + \cdots.
\end{align*}
Putting this together, we have
\begin{align*}
u_2 u_3 u_1 \cdot \emptyset & = u_2 \cdot \Bigl( \beta_2 \ydiagram{1,1,1} + \alpha_1 \beta_2 \ydiagram{2,1,1} + \alpha_1 (\alpha_1 + \beta_1) \beta_2 \ydiagram{2,2,1} + \alpha_1 (\alpha_1 + \beta_2) (\alpha_1 + \beta_1) \beta_2 \ydiagram{2,2,2}
\\ & \hspace{50pt} + \alpha_1 \alpha_2  (\alpha_1 + \beta_1) \beta_2 \ydiagram{3,2,1} + \alpha_1 \alpha_2 \alpha_3 \beta_2 \ydiagram{4,1,1} + \cdots \Bigr)
\\ & = \beta_1 \beta_2 \ydiagram{2,2,1} + \alpha_2 \beta_1 \beta_2 \ydiagram{3,2,1} + \alpha_1 \beta_2 \ydiagram{2,2,1} + \cdots
\end{align*}

Compare with
\begin{align*}
u_2 u_1 u_3 \cdot \emptyset & = u_2 u_1 \cdot \left( \beta_1 \beta_2 \ydiagram{1,1,1} + \alpha_1 \beta_1 \beta_2 u_1 \cdot \ydiagram{1,1,1} + \alpha_1 (\alpha_1 + \beta_1) \beta_2 u_2 \cdot \ydiagram{1,1,1} + \alpha_1 (\alpha_1 + \beta_1) (\alpha_1 + \beta_2) u_3 \cdot \ydiagram{1,1,1} \right)
\\ & = u_2 u_1 \cdot \Bigl( \beta_1 \beta_2 \ydiagram{1,1,1} + \alpha_1 \beta_1 \beta_2 \left(\ydiagram{2,1,1} + \alpha_2 \ydiagram{3,1,1} + \alpha_2 \alpha_3 \ydiagram{4,1,1} \right)
\\ & \hspace{50pt} + \alpha_1 (\alpha_1 + \beta_1) \beta_2 \left( \beta_1 \ydiagram{2,2,1} + \alpha_2 \beta_1 \ydiagram{3,2,1} \right)
\\ & \hspace{50pt} + \alpha_1 (\alpha_1 + \beta_1) (\alpha_1 + \beta_2) \beta_1 \beta_2 \cdot \ydiagram{2,2,2} + \cdots \Bigr)
\allowdisplaybreaks
\\ & = u_2 \cdot \Bigl( \beta_1 \beta_2 \left( \ydiagram{2,1,1} + \alpha_2 \ydiagram{3,1,1} + \alpha_2 \alpha_3 \ydiagram{4,1,1} \right) + \alpha_1 \beta_1 \beta_2 \left(\ydiagram{3,1,1} + (\alpha_2 + \alpha_3) \ydiagram{4,1,1} \right)
\\ & \hspace{50pt} + \alpha_1 (\alpha_1 + \beta_1) \beta_1 \beta_2 \ydiagram{3,2,1} + \cdots \Bigr)
\\ & = u_2 \cdot \Bigl( \beta_1 \beta_2 \ydiagram{2,1,1} + (\alpha_1 + \alpha_2) \beta_1 \beta_2 \ydiagram{3,1,1} + (\alpha_2 \alpha_3 + \alpha_1 \alpha_2 + \alpha_1 \alpha_3) \beta_1 \beta_2 \ydiagram{4,1,1}
\\ & \hspace{50pt} + \alpha_1 (\alpha_1 + \beta_1) \beta_1 \beta_2 \ydiagram{3,2,1} + \cdots \Bigr)
\allowdisplaybreaks
\\ & = \beta_1 \beta_2 \ydiagram{2,2,1} + (\alpha_1 + \alpha_2) \beta_1 \beta_2 \ydiagram{3,2,1} + \cdots.
\end{align*}
As a consequence, we see that the (weak) Knuth relations do not hold for $\uu$.
\end{ex}

\begin{ex}
Let $\mu = \emptyset$ and $\ell(\lambda) \leq 2$.
We directly compute
\begin{align*}
{}_{[\bal,\bbe]} \bra{\mu} a_1 \ket{\lambda}_{[\bal,\bbe]} & =
\bra{\emptyset} a_1 e^{-H(A_{\lambda_1-1})} \psi_{\lambda_1-1} e^{H(\beta_1)} e^{H(A_{\lambda_1-1})} e^{-H(A_{\lambda_2-1})} \psi_{\lambda_2-2} e^{H(\beta_2)} e^{H(A_{\lambda_2-1})} \ket{-2}
\\ & = \bra{\emptyset} a_1 e^{-H(A_{\lambda_1-1})} \psi_{\lambda_1-1} e^{H(\beta_1)} e^{H(A_{[\lambda_2, \lambda_1)})} \psi_{\lambda_2-2}\ket{-2}
\allowdisplaybreaks
\\ & = \sum_{k=0}^{\infty} h_k(\beta_1, A_{[\lambda_2,\lambda_1)}) \cdot \bra{\emptyset} e^{-H(A_{\lambda_1-1})} \psi_{\lambda_1-1} \psi_{\lambda_2-3-k} \ket{-2}
\\ & \hspace{20pt} + \sum_{k=0}^{\infty} h_k(\beta_1, A_{[\lambda_2,\lambda_1)}) \cdot  \bra{\emptyset} e^{-H(A_{\lambda_1-1})} \psi_{\lambda_1-2} \psi_{\lambda_2-2-k}\ket{-2}
\allowdisplaybreaks
\\ & = \sum_{k=0}^{\infty} h_k(\beta_1, A_{[\lambda_2,\lambda_1)}) \cdot s_{(\lambda_1,\lambda_2-k-1)'}(-A_{\lambda_1-1})  
\\ & \hspace{20pt} + \sum_{k=0}^{\infty} h_k(\beta_1, A_{[\lambda_2,\lambda_1)}) \cdot s_{(\lambda_1-1,\lambda_2-k)'}(-A_{\lambda_1-1})
\allowdisplaybreaks
\\ & = \sum_{k=0}^{\infty} \prod_{i=1}^{\lambda_1-1} \alpha_i \cdot h_k(\beta_1, A_{[\lambda_2,\lambda_1)}) \cdot h_{\lambda_2-k}(\emptyset / A_{\lambda_1-1})
 = \prod_{i=1}^{\lambda_1-1} \alpha_i \cdot h_{\lambda_2}(\beta_1 / A_{\lambda_2-1})
\end{align*}
For $\lambda = (4,3)$, we obtain
\begin{align*}
{}_{[\bal,\bbe]} \bra{\mu} a_1 \ket{\lambda}_{[\bal,\bbe]} & = \alpha_1 \alpha_2 \alpha_3 \alpha_4 \bigl( \beta_1^3 - \beta_1^2 e_1(-\alpha_1, -\alpha_2) + \beta_1 e_2(-\alpha_1, -\alpha_2) \bigr)
\\ & = \alpha_1 \alpha_2 \alpha_3 \alpha_4 \bigl( \beta_1^3 + \beta_1^2 (\alpha_1 + \alpha_2) + \beta_1 \alpha_1 \alpha_2 \bigr)
\\ & = \alpha_1 \alpha_2 \alpha_3 \alpha_4 \beta_1 (\beta_1 + \alpha_1) (\beta_1 + \alpha_2).
\end{align*}
\end{ex}

\begin{remark}
Another way to see $\uu$ corresponds to the action of the current operator $a_i$ is to first note that $a_1 = \frac{d}{dt} [e^{H(t)}] \bigr\rvert_{t=0}$.
Therefore, in the expansion of ${}_{[\bal,\bbe]} \bra{\mu} a_1$, we can compute the coefficient of ${}_{[\bal,\bbe]} \bra{\lambda}$ by using~\eqref{eq:orthonormal_basis} and computing
\[
{}_{[\bal,\bbe]} \bra{\mu} a_1 \ket{\lambda}_{[\bal,\bbe]} = \frac{d}{dt} \left[ {}_{[\bal,\bbe]} \bra{\mu} e^{H(t)} \ket{\lambda}_{[\bal,\bbe]} \right] \Bigr\rvert_{t=0} = \frac{d}{dt} [ \dG_{\lambda/\mu}(t; \bal, \bbe) ] \Bigr\rvert_{t=0}.
\]
We can give a precise formula by using the combinatorial description given in~\cite[Def.~4.1]{HJKSS25} as a single marked reverse plane partition.
Indeed, in order for there to be a nonzero contribution, we can only have a single connected component such that the topmost-rightmost box contributes a $t$.
We leave the details to the interested reader.
\end{remark}

As Example~\ref{ex:dual_non_Knuth} demonstrated, the $\uu^{(\bal,\bbe)}$ operators do not satisfy the (weak) Knuth relations.
However, we can see in Example~\ref{ex:dual_non_Knuth} that $\uu^{(0,\bbe)}$ do, and in fact, this holds in general.
This is a straightforward direct computation that refines~\cite[Sec.~6.1]{Iwao20}.

\begin{lemma}
The operators $\uu^{(0,\bbe)}$ satisfy the Knuth relations.
\end{lemma}

\begin{thm}
\label{thm:noncommutative_pushing}
Recall ${}_{[\bbe]} \bra{\lambda} = {}_{[0,\bbe]} \bra{\lambda}$.
We have
\[
{}_{[\bbe]} \bra{\lambda} S_{\mu}(a_1, a_2, \ldots) \equiv {}_{[\bbe]} \bra{s_{\mu}(\uu^{(0,\bbe)}) \cdot \lambda} \mod{M_{[\bbe]}^{\ell\perp}}.
\]
\end{thm}

\begin{proof}
Similarly to the proof of Theorem~\ref{thm:noncommutative_blocking}, it suffices to prove the theorem for the case when $s_\lambda=e_i$ ($i=0,1,2,\dots$):
\[
{}_{[\bbe]}\bra{\lambda} E_i(a_1, a_2, \ldots) = {}_{[\bbe]} \bra{e_i(\uu^{(0,\bbe)}) \cdot \lambda}.
\]
To do so, we first compute
\[
{}_{[\bbe]}\bra{\lambda} a_i = {}_{[\bbe]}\bra{\lambda} P_i(a_1, a_2, \ldots) \equiv \sum_{j=1}^{\ell} v(\lambda + i \epsilon_j, \lambda) \mod{M_{[\bbe]}^{\ell\perp}}
\]
from the fact $[e^{H(\gamma)}, a_i] = 0$ for all $i > 0$.
Hence, Lemma~\ref{lemma:dual_op_action} implies
\[
{}_{[\bbe]}\bra{\lambda} E_i(a_1, a_2, \ldots) = \sum_{j_1 < \cdots < j_i} {}_{[\bbe]} \bra{u_{j_i}^{(0,\bbe)} \cdots u_{j_1}^{(0,\bbe)} \cdot \lambda}, 
\]
and the claim follows.
\end{proof}

\section{Operator dynamics}
\label{sec:operator_dynamics}

In this section, we describe the dynamics particle processes of Dieker and Warren~\cite{DW08} using the deformed Schur operators $U_i^{(0,\bbe)}$ and $u_i^{(0, \bbe)}$ defined in Section~\ref{sec:noncomm_operators} acting on the corresponding state of free fermions.

As we will see below, the geometric distribution will correspond to using homogeneous (noncommutative) symmetric functions in terms of these operators, whereas the Bernoulli distribution uses the elementary symmetric functions.

For this section, our proof of Theorem~\ref{thm:transition_prob} will consist of showing the claim for a single time step.
The general case will follow from the branching rules (Proposition~\ref{prop:branching_rules}) and the Markov property, where both of these say the general case can be described as a product of the single steps summed over all possible routes from $\mu$ to $\lambda$.
In more detail, consider Case~A as an example, but all other cases are similar.
We assume 
\begin{align}
\displaystyle \mathsf{P}(Y_{n}^{\mathrm{A}}=\lambda|Y_{0}^{\mathrm{A}}=\mu
)=\prod_{j=1}^{\ell} \prod_{i=1}^n (1-\pi_j x_i)
\boldsymbol{\pi}^{\lambda/\mu} g_{\lambda/\mu}(x_1,\dots,x_n;\boldsymbol{\pi}^{-1}),
\label{inductionPA}
\end{align}
and show $n$ replaced by $n+1$ holds.
In the next subsection, we show
\begin{align}
\mathsf{P}(Y_k^{\mathrm{A}}=\lambda|Y_{k-1}^{\mathrm{A}}=\mu)=
\prod_{j=1}^\ell (1-\pi_j x_k) \boldsymbol{\pi}^{\lambda/\mu} g_{\lambda/\mu}(x_k;\boldsymbol{\pi}^{-1}),
\label{inductionPAinitial}
\end{align}
which corresponds to $n=1$ case of \eqref{inductionPA}.
By the Markov property, the transition probabilities satisfy
\begin{align}
\mathsf{P}(Y_{n+1}^{\mathrm{A}}=\lambda|Y_{0}^{\mathrm{A}}=\mu
)=
\sum_{ \mu  \subseteq \nu \subseteq \lambda }
\mathsf{P}(Y_{n+1}^{\mathrm{A}}=\lambda|Y_{n}^{\mathrm{A}}=\nu)
\mathsf{P}(Y_n^{\mathrm{A}}=\nu|Y_{0}^{\mathrm{A}}=\mu). \label{PAMarkov}
\end{align}
Inserting \eqref{inductionPA} and \eqref{inductionPAinitial} into the right hand side of \eqref{PAMarkov}
and applying the branching rule for the refined dual Grothendieck polynomials, we have
\begin{align}
\mathsf{P}(Y_{n+1}^{\mathrm{A}}=\lambda|Y_{0}^{\mathrm{A}}=\mu
)=&
\sum_{ \mu  \subseteq \nu \subseteq \lambda }
\prod_{j=1}^\ell (1-\pi_j x_{n+1}) \boldsymbol{\pi}^{\lambda/\nu} g_{\lambda/\nu}(x_{n+1};\boldsymbol{\pi}^{-1})
\nonumber \\
&\times \prod_{j=1}^{\ell} \prod_{i=1}^n (1-\pi_j x_i)
\boldsymbol{\pi}^{\nu/\mu} g_{\nu/\mu}(x_1,\dots,x_n;\boldsymbol{\pi}^{-1}) \nonumber \\
=&\prod_{j=1}^{\ell} \prod_{i=1}^{n+1} (1-\pi_j x_i)  \boldsymbol{\pi}^{\lambda/\mu} 
\sum_{ \mu  \subseteq \nu \subseteq \lambda }
g_{\lambda/\nu}(x_{n+1};\boldsymbol{\pi}^{-1})
g_{\nu/\mu}(x_1,\dots,x_n;\boldsymbol{\pi}^{-1}) \nonumber \\
=&\prod_{j=1}^{\ell} \prod_{i=1}^{n+1} (1-\pi_j x_i)  \boldsymbol{\pi}^{\lambda/\mu} 
g_{\lambda/\mu}(x_1,\dots,x_n,x_{n+1};\boldsymbol{\pi}^{-1}),
\end{align}
which completes the induction.

\subsection{Pushing operators}

Suppose the particles are at positions given by the partition $\mu$.
Then it is easy to see that the action $u_j \cdot \mu$ corresponds to the $j$-th particle trying to move one step to the right at time $i$.
Indeed, if the $j$-th particle is also at a site containing smaller particles, then taking the smallest partition containing $\mu + \epsilon_j$ corresponds to pushing the smaller particles.
More specifically, if this pushes $s$ particles, then we obtain the scalar $\beta_{j-s} \cdots \beta_{j-1}$ (if $s = 0$, then it just is the resulting partition).

\begin{ex}
Consider $4$ particles.
The action
\[
\ytableausetup{boxsize=0.7em}%
u_2 u_4 \cdot \ydiagram{4,1,1,1}
=
\beta_2 \beta_3 u_2 \cdot \ydiagram{4,2,2,2}
=
\beta_2 \beta_3  \cdot \ydiagram{4,3,2,2}
\]
is identified with the particle motion
\[
\begin{tikzpicture}[scale=0.7,baseline=0cm]
\draw[-] (-0.2,0) -- (4.2,0);
\foreach \x in {0,...,4} {
  \draw[-] (\x,-.1) -- (\x,+.1);
  \draw (\x, 0) node[below] {\tiny $\x$};
}
\draw (4.7, 0) node {$\cdots$};
\fill[darkred] (4,0.1) circle (0.1);
\fill[blue] (1,0.1) circle (0.1);
\fill[UQpurple] (1,0.3) circle (0.1);
\fill[dgreencolor] (1,0.5) circle (0.1);
\end{tikzpicture}
\longmapsto
\begin{tikzpicture}[scale=0.7,baseline=0cm]
\draw[-] (-0.2,0) -- (4.2,0);
\foreach \x in {0,...,4} {
  \draw[-] (\x,-.1) -- (\x,+.1);
  \draw (\x, 0) node[below] {\tiny $\x$};
}
\draw (4.7, 0) node {$\cdots$};
\fill[darkred] (4,0.1) circle (0.1);
\fill[blue] (2,0.1) circle (0.1);
\fill[UQpurple] (2,0.3) circle (0.1);
\fill[dgreencolor] (2,0.5) circle (0.1);
\draw[->,blue] (1.2,0+.1) -- (1.8,0+.1);
\draw[->,UQpurple] (1.2,0+.3) -- (1.8,0+.3);
\end{tikzpicture}
\longmapsto
\begin{tikzpicture}[scale=0.7,baseline=0cm]
\draw[-] (-0.2,0) -- (4.2,0);
\foreach \x in {0,...,4} {
  \draw[-] (\x,-.1) -- (\x,+.1);
  \draw (\x, 0) node[below] {\tiny $\x$};
}
\draw (4.7, 0) node {$\cdots$};
\fill[darkred] (4,0.1) circle (0.1);
\fill[blue] (3,0.1) circle (0.1);
\fill[UQpurple] (2,0.1) circle (0.1);
\fill[dgreencolor] (2,0.3) circle (0.1);
\end{tikzpicture},
\]
where the arrows denote the particle being pushed.
\end{ex}

We rewrite the action of our noncommutative operators to match the form of Theorem~\ref{thm:transition_prob}.

\begin{lemma}
\label{lemma:push_op_to_prob}
Let $\bbe = \bpi^{-1}$ and $u_j = u_j^{(0, \pi^{-1})}$.
Then for any $(j_1, j_2, \dotsc, j_k)$, we have
\[
u_{j_1} u_{j_2} \cdots u_{j_k} \cdot \mu = \frac{\pi_{j_1} \pi_{j_2} \cdots \pi_{j_k}}{\bpi^{\lambda / \mu}} \cdot \lambda.
\]
\end{lemma}

\begin{proof}
It is sufficient to prove this when $k = 1$ since $\bpi^{\lambda / \mu} = \bpi^{\lambda/\nu} \bpi^{\nu / \mu}$ for any $\nu$.
By definition, $U_j \cdot \mu = \beta_{j-s} \cdots \beta_{j-2} \beta_{j-1} \cdot \lambda$ for some $s$ (that is determined by $\mu$).
Since $\bpi^{\lambda / \mu} = \pi_{j-s} \cdots \pi_{j-2} \pi_{j-1} \pi_j \cdot \lambda$, we have $u_j \cdot \mu = \frac{\pi_j}{\bpi^{\lambda/\mu}} \cdot \lambda$ as desired.
\end{proof}

Now let us consider the Case~A transition probability for a single time step $\prob_A(\lambda|\mu)$ at time $i$.
We can write this as
\begin{equation}
\label{eq:single_transition_A}
\prob_A(\lambda | \mu) = \pi_{j_1} \pi_{j_2} \cdots \pi_{j_k} x_i^k \prod_{j=1}^{\infty} (1 - \pi_j x_i),
\end{equation}
where $k$ is the number of columns in $\lambda / \mu$ and $j_1 \geq j_2 \geq \cdots \geq j_k$ (which is necessarily unique).
To match the notation in Theorem~\ref{thm:transition_prob}, we specialize $\bbe = \bpi^{-1}$.
Since we update the particles from largest-to-smallest, the above discussion yields that we can write our time evolution with $\bpi = 1$ as
\[
\mcT_A = \sum_{k=0}^{\infty} h_k(x_i \uu) = \sum_{k=0}^{\infty} x_i^k h_k(\uu),
\]
where we have scaled all the $\uu$ operators by $x_i$ to introduce our time-dependent parameters.
Indeed, if we apply $h_k(x_i \uu)$, then in terms of the particle dynamics, we are moving all particle a total number of $k$ steps with probability~\eqref{eq:single_transition_A}.
To see again the particle-dependent parameters for the time evolution, if we bring the $\mcT_A$ inside the matrix coefficient, we must factor out the total position change factors $\bpi^{\lambda/\mu}$ as this cancels with the $\pi_{j_1}^{-1} \cdots \pi_{j_k}^{-1}$ by applying the various $u_j$ to $\mu$ until we get $\lambda$ (see~\eqref{eq:uj_action_expansion}).
As an example, $u_j \cdot \mu = \pi_{j-s} \cdots \pi_j \cdot {}_{[\bpi^{-1}]} \bra{u_j \cdot \mu}$.
Therefore, we can write the transition probability~\eqref{eq:single_transition_A} as
\begin{equation}
\label{eq:prob_A_operator}
\begin{aligned}
\prob_A(\lambda|\mu) & =  \bpi^{\lambda/\mu} \prod_{j=1}^{\infty} (1 - \pi_j x_i) \cdot {}_{[\bpi^{-1}]} \braket{\mcT_A \cdot \mu}{\lambda}_{[\bpi^{-1}]}
\\ & = \bpi^{\lambda/\mu} \prod_{j=1}^{\infty} (1 - \pi_j x_i) \sum_{k=0}^{\infty} x_i^k \cdot {}_{[\bpi^{-1}]} \braket{h_k(\uu) \cdot \mu}{\lambda}_{[\bpi^{-1}]}.
\end{aligned}
\end{equation}
Alternatively we can see~\eqref{eq:prob_A_operator} by noting in the second line, only the term $u_{j_1} u_{j_2} \cdots u_{j_k}$ is nonzero in the pairing by~\eqref{eq:orthonormal_basis} and taking this together with Lemma~\ref{lemma:push_op_to_prob}.
Next, we apply Theorem~\ref{thm:noncommutative_pushing},~\eqref{eq:exp_current_sum}, and~\eqref{eq:free_fermion_grothendiecks} to rewrite Equation~\eqref{eq:prob_A_operator} as
\begin{align*}
\prob_A(\lambda|\mu) & = \bpi^{\lambda/\mu} \prod_{j=1}^{\infty} (1 - \pi_j x_i) \sum_{k=0}^{\infty} x_i^k \cdot {}_{[\bpi^{-1}]} \bra{\mu} H_k(a_1, a_2, \dotsc) \ket{\lambda}_{[\bpi^{-1}]}
\\ & = \bpi^{\lambda/\mu} \prod_{j=1}^{\infty} (1 - \pi_j x_i) \cdot {}_{[\bpi^{-1}]} \bra{\mu} e^{H(x_i)} \ket{\lambda}_{[\bpi^{-1}]}
 = \bpi^{\lambda/\mu} \prod_{j=1}^{\infty} (1 - \pi_j x_i) \dG_{\lambda/\mu}(\xx; \bpi^{-1}).
\end{align*}
This is precisely the claim of Theorem~\ref{thm:transition_prob} for a single time step.

\begin{ex}
\label{ex:geometric_op_push}
Let us consider the case when at most three particles move, so we can restrict ourselves to $\uu_3 = (u_1, u_2, u_3)$.
This is equivalent to setting $\pi_j = 0$ for all $j > 3$ or considering the case with exactly three particles in the system.
So the first noncommutative homogeneous symmetric functions are
\begin{equation}
\label{eq:ncsym_h123}
\begin{gathered}
h_1(\uu_3) = u_1 + u_2 + u_3,
\qquad\qquad
h_2(\uu_3) = u_1^2 + u_1 u_2 + u_1 u_3 + u_2^2 + u_2 u_3 + u_3^2,
\\
h_3(\uu_3) = u_1^{3} + u_1^{2} u_2 + u_1^{2} u_3 + u_1 u_2^{2} + u_1 u_2 u_3 + u_1 u_3^{2} + u_2^{3}+ u_2^{2} u_3 + u_2 u_3^{2} + u_3^{3}.
\end{gathered}
\end{equation}
Let us take $\mu = (1,1) = \ydiagram{1,1}$, and we compute
\begin{align*}
h_1(\uu_3) \cdot \lambda & = \ydiagram{2,1}+ \pi_1^{-1} \ydiagram{2,2} + \ydiagram{1,1,1}\,,
\\
h_2(\uu_3) \cdot \lambda & = \ydiagram{3,1} + \pi_1^{-1} \ydiagram{3,2} + \ydiagram{2,1,1} + \pi_1^{-2} \ydiagram{3,3} + \pi_1^{-1} \ydiagram{2,2,1}  + \pi_1^{-1} \pi_2^{-1} \ydiagram{2,2,2}\,,
\\
h_3(\uu_3) \cdot \lambda & = \ydiagram{4,1} + \pi_1^{-1} \ydiagram{4,2} + \ydiagram{3,1,1} + \pi_1^{-2} \ydiagram{4,3} + \pi_1^{-1} \ydiagram{3,2,1}  + \pi_1^{-1} \pi_2^{-1} \ydiagram{3,2,2}
\\ & \hspace{20pt} + \pi_1^{-3} \ydiagram{4,4} + \pi_1^{-2} \ydiagram{3,3,1} + \pi_1^{-2} \pi_2^{-1} \ydiagram{3,3,2} + \pi_1^{-2} \pi_2^{-2} \ydiagram{3,3,3}\,.
\end{align*}
Next, we apply Theorem~\ref{thm:noncommutative_pushing} and~\eqref{eq:exp_current_sum} to compute
\begin{align*}
{}_{[\bpi^{-1}]} \bra{\mu} e^{H(x_i)} & = {}_{[\bpi^{-1}]} \bra{\mu} H_1(a_1, a_2, \ldots) + x_i \cdot {}_{[\bpi^{-1}]} \bra{\mu} H_1(a_1, a_2, \ldots)
\\ & \hspace{20pt} + x_i^2 \cdot {}_{[\bpi^{-1}]} \bra{\mu} H_2(a_1, a_2, \ldots) + x_i^3 \cdot {}_{[\bpi^{-1}]} \bra{\mu} H_3(a_1, a_2, \ldots) + \cdots
\allowdisplaybreaks
\\ & = {}_{[\bpi^{-1}]} \bra{\mu} + x_i \cdot {}_{[\bpi^{-1}]} \bra{h_1(\uu_3) \cdot \mu} + x_i^2 \cdot {}_{[\bpi^{-1}]} \bra{h_2(\uu_3) \cdot \mu} + x_i^3 \cdot {}_{[\bpi^{-1}]} \bra{h_3(\uu_3) \cdot \mu} + \cdots
\allowdisplaybreaks
\\ & = {}_{[\bpi^{-1}]} \bra{1,1} + x_i \cdot {}_{[\bpi^{-1}]} \bra{2,1} + x_i \cdot {}_{[\bpi^{-1}]} \bra{1,1,1}
\\ & \hspace{20pt} + x_i^2  \cdot {}_{[\bpi^{-1}]} \bra{3,1} + \frac{x_i}{\pi_1} \cdot {}_{[\bpi^{-1}]} \bra{2,2} + x_i^2 \cdot {}_{[\bpi^{-1}]} \bra{2,1,1}
\\ & \hspace{20pt} + x_i^3  \cdot {}_{[\bpi^{-1}]} \bra{4,1} + \frac{x_i^2}{\pi_1} \cdot {}_{[\bpi^{-1}]} \bra{3,2} + x_i^3 \cdot {}_{[\bpi^{-1}]} \bra{3,1,1} + \frac{x_i^2}{\pi_1} \cdot {}_{[\bpi^{-1}]} \bra{2,2,1}
+ \cdots
\allowdisplaybreaks
\\ & = {}_{[\bpi^{-1}]} \bra{1,1} + \frac{\pi_1 x_i }{\bpi^{(2,1)/\mu}} \cdot {}_{[\bpi^{-1}]} \bra{2,1} + \frac{\pi_3 x_i }{\bpi^{(1,1,1)/\mu}} \cdot {}_{[\bpi^{-1}]} \bra{1,1,1}
\\ & \hspace{20pt} + \frac{(\pi_1 x_i)^2}{\bpi^{(3,1)/\mu}}  \cdot {}_{[\bpi^{-1}]} \bra{3,1} + \frac{\pi_2 x_i}{\bpi^{(2,2)/\mu}} \cdot {}_{[\bpi^{-1}]} \bra{2,2} + \frac{(\pi_1 x_i)(\pi_2 x_i)}{\bpi^{(2,1,1)/\mu}} \cdot {}_{[\bpi^{-1}]} \bra{2,1,1}
\\ & \hspace{20pt} + \frac{(\pi_1 x_i)^3}{\bpi^{(4,1)/\mu}}  \cdot {}_{[\bpi^{-1}]} \bra{4,1} + \frac{(\pi_1 x_i)(\pi_2 x_i)}{\bpi^{(3,2)/\mu}} \cdot {}_{[\bpi^{-1}]} \bra{3,2}
\\ & \hspace{20pt} + \frac{(\pi_1 x_i)^2(\pi_3 x_i)}{\bpi^{(3,1,1)/\mu}} \cdot {}_{[\bpi^{-1}]} \bra{3,1,1} + \frac{(\pi_2 x_i)(\pi_3 x_i)}{\bpi^{(2,2,1)/\mu}} \cdot {}_{[\bpi^{-1}]} \bra{2,2,1}
+ \cdots
\end{align*}
Here, we have given all of the terms ${}_{[\bpi^{-1}]}\bra{\lambda}$ with $\abs{\lambda} \leq 5$.
Ignoring the normalization constant $C = \prod_{j=1}^3 (1 - \pi_j x_i)$, we see that all possible configurations and their probabilities (multiplied by $C$) we can obtain from moving three particles from $\mu$ such that the total distance the particles move from the step initial condition is at most $5$ are
\begin{align*}
1 & \cdot 
\begin{tikzpicture}[scale=0.7,baseline=0cm]
\draw[-] (-0.2,0) -- (4.2,0);
\foreach \x in {0,...,4} {
  \draw[-] (\x,-.1) -- (\x,+.1);
  \draw (\x, 0) node[below] {\tiny $\x$};
}
\draw (4.7, 0) node {$\cdots$};
\fill[darkred] (1,0.1) circle (0.1);
\fill[blue] (1,0.3) circle (0.1);
\fill[UQpurple] (0,0.1) circle (0.1);
\end{tikzpicture}
&
\pi_1 x_i & \cdot 
\begin{tikzpicture}[scale=0.7,baseline=0cm]
\draw[-] (-0.2,0) -- (4.2,0);
\foreach \x in {0,...,4} {
  \draw[-] (\x,-.1) -- (\x,+.1);
  \draw (\x, 0) node[below] {\tiny $\x$};
}
\draw (4.7, 0) node {$\cdots$};
\fill[darkred] (2,0.1) circle (0.1);
\fill[blue] (1,0.1) circle (0.1);
\fill[UQpurple] (0,0.1) circle (0.1);
\end{tikzpicture}
\\
\pi_3 x_i & \cdot 
\begin{tikzpicture}[scale=0.7,baseline=0cm]
\draw[-] (-0.2,0) -- (4.2,0);
\foreach \x in {0,...,4} {
  \draw[-] (\x,-.1) -- (\x,+.1);
  \draw (\x, 0) node[below] {\tiny $\x$};
}
\draw (4.7, 0) node {$\cdots$};
\fill[darkred] (1,0.1) circle (0.1);
\fill[blue] (1,0.3) circle (0.1);
\fill[UQpurple] (1,0.5) circle (0.1);
\end{tikzpicture}
&
(\pi_1 x_i)^2 & \cdot 
\begin{tikzpicture}[scale=0.7,baseline=0cm]
\draw[-] (-0.2,0) -- (4.2,0);
\foreach \x in {0,...,4} {
  \draw[-] (\x,-.1) -- (\x,+.1);
  \draw (\x, 0) node[below] {\tiny $\x$};
}
\draw (4.7, 0) node {$\cdots$};
\fill[darkred] (3,0.1) circle (0.1);
\fill[blue] (1,0.1) circle (0.1);
\fill[UQpurple] (0,0.1) circle (0.1);
\end{tikzpicture}
\\
\pi_2 x_i & \cdot 
\begin{tikzpicture}[scale=0.7,baseline=0cm]
\draw[-] (-0.2,0) -- (4.2,0);
\foreach \x in {0,...,4} {
  \draw[-] (\x,-.1) -- (\x,+.1);
  \draw (\x, 0) node[below] {\tiny $\x$};
}
\draw (4.7, 0) node {$\cdots$};
\fill[darkred] (2,0.1) circle (0.1);
\fill[blue] (2,0.3) circle (0.1);
\fill[UQpurple] (0,0.1) circle (0.1);
\draw[->,darkred] (1.2,0+.1) -- (1.8,0+.1);
\end{tikzpicture}
&
(\pi_1 x_i) (\pi_3 x_i) & \cdot 
\begin{tikzpicture}[scale=0.7,baseline=0cm]
\draw[-] (-0.2,0) -- (4.2,0);
\foreach \x in {0,...,4} {
  \draw[-] (\x,-.1) -- (\x,+.1);
  \draw (\x, 0) node[below] {\tiny $\x$};
}
\draw (4.7, 0) node {$\cdots$};
\fill[darkred] (2,0.1) circle (0.1);
\fill[blue] (1,0.1) circle (0.1);
\fill[UQpurple] (1,0.3) circle (0.1);
\end{tikzpicture}
\\
(\pi_1 x_i)^3 & \cdot 
\begin{tikzpicture}[scale=0.7,baseline=0cm]
\draw[-] (-0.2,0) -- (4.2,0);
\foreach \x in {0,...,4} {
  \draw[-] (\x,-.1) -- (\x,+.1);
  \draw (\x, 0) node[below] {\tiny $\x$};
}
\draw (4.7, 0) node {$\cdots$};
\fill[darkred] (4,0.1) circle (0.1);
\fill[blue] (1,0.1) circle (0.1);
\fill[UQpurple] (0,0.1) circle (0.1);
\end{tikzpicture}
&
(\pi_1 x_i) (\pi_2 x_i) & \cdot 
\begin{tikzpicture}[scale=0.7,baseline=0cm]
\draw[-] (-0.2,0) -- (4.2,0);
\foreach \x in {0,...,4} {
  \draw[-] (\x,-.1) -- (\x,+.1);
  \draw (\x, 0) node[below] {\tiny $\x$};
}
\draw (4.7, 0) node {$\cdots$};
\fill[darkred] (3,0.1) circle (0.1);
\fill[blue] (2,0.1) circle (0.1);
\fill[UQpurple] (0,0.1) circle (0.1);
\draw[->,darkred] (1.2,0+.1) -- (1.8,0+.1);
\end{tikzpicture}
\\
(\pi_1 x_i)^2 (\pi_3 x_i) & \cdot 
\begin{tikzpicture}[scale=0.7,baseline=0cm]
\draw[-] (-0.2,0) -- (4.2,0);
\foreach \x in {0,...,4} {
  \draw[-] (\x,-.1) -- (\x,+.1);
  \draw (\x, 0) node[below] {\tiny $\x$};
}
\draw (4.7, 0) node {$\cdots$};
\fill[darkred] (3,0.1) circle (0.1);
\fill[blue] (1,0.1) circle (0.1);
\fill[UQpurple] (1,0.3) circle (0.1);
\end{tikzpicture}
&
(\pi_2 x_i)(\pi_3 x_i) & \cdot 
\begin{tikzpicture}[scale=0.7,baseline=0cm]
\draw[-] (-0.2,0) -- (4.2,0);
\foreach \x in {0,...,4} {
  \draw[-] (\x,-.1) -- (\x,+.1);
  \draw (\x, 0) node[below] {\tiny $\x$};
}
\draw (4.7, 0) node {$\cdots$};
\fill[darkred] (2,0.1) circle (0.1);
\fill[blue] (2,0.3) circle (0.1);
\fill[UQpurple] (1,0.1) circle (0.1);
\draw[->,darkred] (1.2,0+.1) -- (1.8,0+.1);
\end{tikzpicture}
\end{align*}
where again an arrow denotes a particle that was pushed.
\end{ex}

For Case~D, we do the analogous proof using $\prob_D(\lambda|\mu)$ starting with the time evolution at $\brho = 1$
\[
\mcT_D = \sum_{k=0}^{\infty} e_k(x_i \uu) = \sum_{k=0}^{\infty} x_i e_k(\uu),
\]
where here we specialize $\bbe = \rho^{-1}$.
Indeed, after adding back in the particle-dependent parameters like before, we compute
\begin{align*}
\prob_D(\lambda|\mu) & = \frac{\brho^{\lambda/\mu} \cdot {}_{[\brho^{-1}]} \braket{\mcT_D \cdot \mu}{\lambda}_{[\brho^{-1}]}}{\prod_{j=1}^{\infty} (1 + \rho_j x_i)}
= \frac{\brho^{\lambda/\mu}}{\prod_{j=1}^{\infty} (1 + \rho_j x_i)} \sum_{k=0}^{\infty} x_i^k \cdot {}_{[\brho^{-1}]} \bra{\mu} E_k(a_1, a_2, \dotsc) \ket{\lambda}_{[\brho^{-1}]}
\\ & = \frac{\brho^{\lambda/\mu}}{\prod_{j=1}^{\infty} (1 + \rho_j x_i)} \cdot {}_{[\brho^{-1}]} \bra{\mu} e^{J(x_i)} \ket{\lambda}_{[\brho^{-1}]}
 = \frac{\brho^{\lambda/\mu}}{\prod_{j=1}^{\infty} (1 + \rho_j x_i)} \dwG_{\lambda'/\mu'}(\xx; \brho^{-1}).
\end{align*}
Note that the order of the operators from $e_k(\uu)$ is applied smallest-to-largest and matches the update rule.
We can also see the first equality by using Lemma~\ref{lemma:push_op_to_prob}.

\begin{ex}
Like Example~\ref{ex:geometric_op_push}, we consider the case with exactly $3$ particles, or equivalently $\rho_j = 0$ for all $j > 3$.
Therefore, we restrict to $\uu_3 = (u_1, u_2, u_3)$ (thus we consider $u_i = 0$ for all $i > 3$) and only need to consider
\begin{align}
\label{eq:ncsym_e123}
e_1(\uu_3) & = u_1 + u_2 + u_3,
&
e_2(\uu_3) & = u_2 u_1 + u_3 u_1 + u_3 u_2,
&
e_3(\uu_3) & = u_3 u_2 u_1.
\end{align}
Let us consider $\mu = (1,1,0)$, and we compute
\[
e_1(\uu_3) \cdot \mu = \ydiagram{2,1} + \rho_1 \ydiagram{2,2} + \ydiagram{1,1,1}\,,
\qquad
e_2(\uu_3) \cdot \mu = \ydiagram{2,2} + \ydiagram{2,1,1} + \rho_1 \ydiagram{2,2,1}\,,
\qquad
e_3(\uu_3) \cdot \mu = \ydiagram{2,2,1}\,.
\]
Next, by applying~\eqref{eq:exp_current_sum}, we have
\begin{align*}
{}_{[\brho^{-1}]} \bra{\mu} e^{J(x_i)} & = {}_{[\brho^{-1}]}\bra{1,1} + x_i \cdot {}_{[\brho^{-1}]}\bra{2,1} + \rho_1 x_i \cdot {}_{[\brho^{-1}]}\bra{2,2} + x_i \cdot {}_{[\brho^{-1}]}\bra{1,1,1}
\\ & \hspace{20pt} + x_i^2 \cdot {}_{[\brho^{-1}]}\bra{2,2} + x_i^2 \cdot {}_{[\brho^{-1}]}\bra{2,1,1} + \rho_1^{-1} x_i^2 \cdot {}_{[\brho^{-1}]}\bra{2,2,1}
+ x_i^3 \cdot {}_{[\brho^{-1}]}\bra{2,2,1}
\allowdisplaybreaks
\\ & = {}_{[\brho^{-1}]}\bra{1,1} + x_i \cdot {}_{[\brho^{-1}]}\bra{2,1} + (x_i^2 + \rho_1^{-1} x_i ) \cdot {}_{[\brho^{-1}]}\bra{2,2} + x_i \cdot {}_{[\brho^{-1}]}\bra{1,1,1}
\\ & \hspace{20pt} + x_i^2 \cdot {}_{[\brho^{-1}]}\bra{2,1,1} + (x_i^3 + \rho_1^{-1} x_i^2) \cdot {}_{[\brho^{-1}]}\bra{2,2,1}
\allowdisplaybreaks
\\ & = {}_{[\brho^{-1}]}\bra{1,1} + \frac{\rho_1 x_i }{\brho^{(2,1)/\mu}} \cdot {}_{[\brho^{-1}]} \bra{2,1} + \frac{(\rho_1 x_i) (\rho_2 x_i) + \rho_2 x_i}{\brho^{(2,2)/\mu}} \cdot {}_{[\brho^{-1}]}\bra{2,2}
\\ & \hspace{20pt} + \frac{\rho_3 x_i}{\brho^{(1,1,1)/\mu}} \cdot {}_{[\brho^{-1}]}\bra{1,1,1} + \frac{(\rho_1 x_i)(\rho_3 x_i)}{\brho^{(2,1,1)/\mu}} \cdot {}_{[\brho^{-1}]}\bra{2,1,1}
\\ & \hspace{20pt} + \frac{(\rho_1 x_i)(\rho_2 x_i)(\rho_3 x_i) + (\rho_2 x_i)(\rho_3 x_i)}{\brho^{(2,2,1)/\mu}} \cdot {}_{[\brho^{-1}]}\bra{2,2,1}.
\end{align*}
Ignoring the normalization constant $C = \prod_{j=1}^3 (1 + \rho_j x_i)^{-1}$, we see that all possible configurations and their probabilities (multiplied by $C$) we can obtain from moving the three particles from $\mu$ are
\begin{align*}
1 & \cdot 
\begin{tikzpicture}[scale=0.7,baseline=0cm]
\draw[-] (-0.2,0) -- (3.2,0);
\foreach \x in {0,...,3} {
  \draw[-] (\x,-.1) -- (\x,+.1);
  \draw (\x, 0) node[below] {\tiny $\x$};
}
\draw (3.7, 0) node {$\cdots$};
\fill[darkred] (1,0.1) circle (0.1);
\fill[blue] (1,0.3) circle (0.1);
\fill[UQpurple] (0,0.1) circle (0.1);
\end{tikzpicture}
&
\rho_1 x_i & \cdot 
\begin{tikzpicture}[scale=0.7,baseline=0cm]
\draw[-] (-0.2,0) -- (3.2,0);
\foreach \x in {0,...,3} {
  \draw[-] (\x,-.1) -- (\x,+.1);
  \draw (\x, 0) node[below] {\tiny $\x$};
}
\draw (3.7, 0) node {$\cdots$};
\fill[darkred] (2,0.1) circle (0.1);
\fill[blue] (1,0.1) circle (0.1);
\fill[UQpurple] (0,0.1) circle (0.1);
\end{tikzpicture}
\\
\rho_2 x_i & \cdot 
\begin{tikzpicture}[scale=0.7,baseline=0cm]
\draw[-] (-0.2,0) -- (3.2,0);
\foreach \x in {0,...,3} {
  \draw[-] (\x,-.1) -- (\x,+.1);
  \draw (\x, 0) node[below] {\tiny $\x$};
}
\draw (3.7, 0) node {$\cdots$};
\fill[darkred] (2,0.1) circle (0.1);
\fill[blue] (2,0.3) circle (0.1);
\fill[UQpurple] (0,0.1) circle (0.1);
\draw[->,darkred] (1.2,0+.1) -- (1.8,0+.1);
\end{tikzpicture}
&
(\rho_1 x_i)(\rho_2 x_i) & \cdot 
\begin{tikzpicture}[scale=0.7,baseline=0cm]
\draw[-] (-0.2,0) -- (3.2,0);
\foreach \x in {0,...,3} {
  \draw[-] (\x,-.1) -- (\x,+.1);
  \draw (\x, 0) node[below] {\tiny $\x$};
}
\draw (3.7, 0) node {$\cdots$};
\fill[darkred] (2,0.1) circle (0.1);
\fill[blue] (2,0.3) circle (0.1);
\fill[UQpurple] (0,0.1) circle (0.1);
\end{tikzpicture}
\\
\rho_3 x_i & \cdot 
\begin{tikzpicture}[scale=0.7,baseline=0cm]
\draw[-] (-0.2,0) -- (3.2,0);
\foreach \x in {0,...,3} {
  \draw[-] (\x,-.1) -- (\x,+.1);
  \draw (\x, 0) node[below] {\tiny $\x$};
}
\draw (3.7, 0) node {$\cdots$};
\fill[darkred] (1,0.1) circle (0.1);
\fill[blue] (1,0.3) circle (0.1);
\fill[UQpurple] (1,0.5) circle (0.1);
\end{tikzpicture}
&
(\rho_1 x_i)(\rho_3 x_i) & \cdot 
\begin{tikzpicture}[scale=0.7,baseline=0cm]
\draw[-] (-0.2,0) -- (3.2,0);
\foreach \x in {0,...,3} {
  \draw[-] (\x,-.1) -- (\x,+.1);
  \draw (\x, 0) node[below] {\tiny $\x$};
}
\draw (3.7, 0) node {$\cdots$};
\fill[darkred] (2,0.1) circle (0.1);
\fill[blue] (1,0.1) circle (0.1);
\fill[UQpurple] (1,0.3) circle (0.1);
\end{tikzpicture}
\\
(\rho_2 x_i)(\rho_3 x_i) & \cdot 
\begin{tikzpicture}[scale=0.7,baseline=0cm]
\draw[-] (-0.2,0) -- (3.2,0);
\foreach \x in {0,...,3} {
  \draw[-] (\x,-.1) -- (\x,+.1);
  \draw (\x, 0) node[below] {\tiny $\x$};
}
\draw (3.7, 0) node {$\cdots$};
\fill[darkred] (2,0.1) circle (0.1);
\fill[blue] (2,0.3) circle (0.1);
\fill[UQpurple] (1,0.1) circle (0.1);
\draw[->,darkred] (1.2,0+.1) -- (1.8,0+.1);
\end{tikzpicture}
&
(\rho_1 x_i)(\rho_2 x_i)(\rho_3 x_i) & \cdot 
\begin{tikzpicture}[scale=0.7,baseline=0cm]
\draw[-] (-0.2,0) -- (3.2,0);
\foreach \x in {0,...,3} {
  \draw[-] (\x,-.1) -- (\x,+.1);
  \draw (\x, 0) node[below] {\tiny $\x$};
}
\draw (3.7, 0) node {$\cdots$};
\fill[darkred] (2,0.1) circle (0.1);
\fill[blue] (2,0.3) circle (0.1);
\fill[UQpurple] (1,0.1) circle (0.1);
\end{tikzpicture}
\end{align*}
\end{ex}

\subsection{Blocking operators}
\label{sec:block_op_proof}

Suppose the particles are at positions given by the partition $\mu$.
Then it is easy to see that the action $U_j \cdot \mu$ corresponds to the $j$-th particle trying to move one step to the right at time $i$ and keeping the other particles fixed.
If the move is blocked by the $(j-1)$-th particle (being at the same position), then we obtain the scalar $\beta_{j-1}$.
Otherwise the particle moves and we simply obtain the resulting partition.

\begin{ex}
Consider $4$ particles.
The action
\[
U_2 U_4 \cdot \ydiagram{4,2,1,1}
=
\beta_3 U_2 \cdot \ydiagram{4,2,1,1}
=
\beta_3 \cdot \ydiagram{4,3,1,1}
\]
(which also equals $U_4 U_2 \cdot (4, 2, 1, 1)$)
is identified with the particle motion
\[
\begin{tikzpicture}[scale=0.7,baseline=0cm]
\draw[-] (-0.2,0) -- (4.2,0);
\foreach \x in {0,...,4} {
  \draw[-] (\x,-.1) -- (\x,+.1);
  \draw (\x, 0) node[below] {\tiny $\x$};
}
\draw (4.7, 0) node {$\cdots$};
\fill[darkred] (4,0.1) circle (0.1);
\fill[blue] (2,0.1) circle (0.1);
\fill[UQpurple] (1,0.1) circle (0.1);
\fill[dgreencolor] (1,0.3) circle (0.1);
\end{tikzpicture}
\longmapsto
\begin{tikzpicture}[scale=0.7,baseline=0cm]
\draw[-] (-0.2,0) -- (4.2,0);
\foreach \x in {0,...,4} {
  \draw[-] (\x,-.1) -- (\x,+.1);
  \draw (\x, 0) node[below] {\tiny $\x$};
}
\draw (4.7, 0) node {$\cdots$};
\fill[darkred] (4,0.1) circle (0.1);
\fill[blue] (2,0.1) circle (0.1);
\fill[UQpurple] (1,0.1) circle (0.1);
\fill[dgreencolor] (1,0.3) circle (0.1);
\end{tikzpicture}
\longmapsto
\begin{tikzpicture}[scale=0.7,baseline=0cm]
\draw[-] (-0.2,0) -- (4.2,0);
\foreach \x in {0,...,4} {
  \draw[-] (\x,-.1) -- (\x,+.1);
  \draw (\x, 0) node[below] {\tiny $\x$};
}
\draw (4.7, 0) node {$\cdots$};
\fill[darkred] (4,0.1) circle (0.1);
\fill[blue] (3,0.1) circle (0.1);
\fill[UQpurple] (1,0.1) circle (0.1);
\fill[dgreencolor] (1,0.3) circle (0.1);
\end{tikzpicture}.
\]
Note that the fourth particle is blocked but the second particle moves.
\end{ex}

We rewrite the noncommutative operator action to match Theorem~\ref{thm:transition_prob}.

\begin{lemma}
\label{lemma:block_op_to_prob}
Let $\beta_j = \rho_{j+1}$ and $U_j = U_j^{(0,\bbe)}$.
Then for any $(j_1, j_2, \dotsc, j_k)$, we have
\[
U_{j_1} U_{j_2} \cdots U_{j_k} \cdot \mu = \frac{\rho_{j_1} \rho_{j_2} \cdots \rho_{j_k}}{\brho^{\lambda / \mu}} \cdot \lambda.
\]
\end{lemma}

\begin{proof}
Like the proof of Lemma~\ref{lemma:push_op_to_prob}, we can reduce the proof to the case $k = 1$.
By definition, we either have (i) $U_j \cdot \mu = \beta_{j-1} \cdot \mu$ and $\lambda = \mu$ or (ii) $U_j \cdot \mu = \kappa_j \cdot \mu$ and $\lambda = \mu + \epsilon_j$.
In each case, it is clear the claim holds.
\end{proof}

Let us consider the transition probability for a single time step $\prob_B(\lambda' | \mu')$ in Case~B starting at time $i$.
This proof is largely analogous as for Case~A.
Here, we use $U_i := U_i^{(0,\bal)}$ and specialize $\alpha_j = \rho_{j+1}$.
Since we update particles from smallest-to-largest, the above description means one time evolution is given by
\[
\mcT_B = \sum_{k=0}^{\infty} e_k(x_i \UU) = \sum_{k=0}^{\infty} x_i^k e_k(\UU).
\]
From the analogous arguments as in Case~A or the above discussion, we need to multiply by $\brho^{\lambda / \mu}$ to move the time evolution inside the matrix coefficient and account for the movement of all of the particles.
Additionally, we scale each $U_j$ by $x_i$ to introduce the time parameters.
Therefore, we have
\begin{equation}
\label{eq:prob_B_operator}
\prob_B(\lambda' | \mu') = \frac{\brho^{\lambda/\mu} \cdot {}^{[\bal]} \braket{\mcT_B \cdot \mu}{\lambda}^{[\bal]}}{\prod_{j=1}^{\infty} (1 + \rho_j x_i)} = \sum_{k=0}^{\infty} \frac{\brho^{\lambda/\mu} x_i^k \cdot {}^{[\bal]} \braket{ e_k(\UU) \cdot \mu}{\lambda}^{[\bal]}}{\prod_{j=1}^{\infty} (1 + \rho_j x_i)}.
\end{equation}
We can also see Equation~\eqref{eq:prob_B_operator} by using Lemma~\ref{lemma:block_op_to_prob} with the fact for any $\ell \geq \ell(\lambda)$, we have
\[
\prob_B(\lambda | \mu) = \prod_{j=1}^{\infty} (1 + \rho_j x_i) \sum_{k=0}^{\infty} \sum_{\substack{j_1 < \cdots < j_k \\ U_{j_k} \cdots U_{j_1} \cdot \mu' = \ast \cdot \lambda'}} \rho_{j_k} \cdots \rho_{j_1} \cdot x_i^k,
\]
where $\ast$ represents any nonzero constant; note that $U_{j_1} \cdots U_{j_k} \cdot \mu' = \ast \cdot \lambda'$ is simply saying the movement of the particles from $\mu$ to $\lambda$ is given by moving (with blocking) the $j_1, j_2, \dotsc, j_k$ particles (in that order).
Next, we use~\eqref{eq:exp_current_sum} and Theorem~\ref{thm:noncommutative_blocking} to compute
\begin{align*}
\wG_{\lambda'\ds\mu'}(\xx; \bal) = {}^{[\bal]} \bra{\mu} e^{J(x_i)} \ket{\lambda}^{[\bal]}
& = \sum_{m=0}^{\infty} x_i^m \cdot {}^{[\bal]} \bra{\mu} E_m(a_1, a_2, \ldots) \ket{\lambda}^{[\bal]}
\\ & = \sum_{m=0}^{\infty} x_i^m \cdot {}^{[\bal]} \braket{e_m(\UU / \bal) \cdot \mu}{\lambda}^{[\bal]}
\\ & = \sum_{k=0}^{\infty} \sum_{m=k}^{\infty} (-1)^{k-m}  x_i^m h_{m-k}(\bal) \cdot {}^{[\bal]} \braket{e_k(\UU) \cdot \mu}{\lambda}^{[\bal]}
\\ & = \sum_{k=0}^{\infty} x_i^k \prod_{j=1}^{\infty} (1 + \alpha_j x_i)^{-1} \cdot {}^{[\bal]} \braket{e_k(\UU) \cdot \mu}{\lambda}^{[\bal]}
\end{align*}
Therefore, comparing this with Equation~\eqref{eq:prob_B_operator} (recall that $\alpha_j = \rho_{j+1}$), we have Theorem~\ref{thm:transition_prob},
\[
\prob_B(\lambda | \mu) = \prod_{i=1}^n (1 + \rho_1 x_i)^{-1} \brho^{\lambda/\mu} \wG_{\lambda'\ds\mu'}(\xx;\bal),
\]
for one time step.

Alternatively, let us examine Equation~\eqref{eq:prob_B_operator}.
Necessarily we must have $\lambda / \mu$ being a vertical strip (equivalently $\lambda' / \mu'$ being a horizontal strip) as otherwise both sides are $0$, so we now assume $\lambda / \mu$ is a vertical strip.
Suppose $m = \abs{\lambda} - \abs{\mu}$ particles move, and so each term in the sum is $0$ unless $k \geq m$.
Furthermore, when $k >m$, we only get a nonzero contribution from the $j$ such that $\lambda_{j-1} = \mu_j$, where necessarily $j > 1$.
Let $J = \{ j \in \ZZ_{>1} \mid \lambda_{j-1} = \mu_j \}$ be the (infinite) set of all such indices, and so we have
\begin{align*}
\sum_{k=0}^{\infty} x_i^k \cdot {}^{[\bal]} \braket{ e_k(\UU) \cdot \mu}{\lambda}^{[\bal]}
& = \sum_{k=m}^{\infty} x_i^k \sum_{\substack{X \subseteq J \\ \abs{X} = k-m}} \prod_{j \in X} \rho_j
= \sum_{k=m}^{\infty} x_i^k e_{k-m}(\brho_J)
\\
& = x_i^m \sum_{k=0}^{\infty} x_i^k e_k(\brho_J)
= x_i^m \prod_{j \in J} (1 + \rho_j x_i),
\end{align*}
where $\brho_J = \{\rho_j \mid j \in J\}$ (since $e_k$ is a symmetric function, we do not need to worry about the order).
Therefore, we have
\[
\prob_B(\lambda|\mu) = \frac{\brho^{\lambda/\mu} x_i^m}{1 + \rho_1 x_i} \prod_{j \in \widetilde{J}} (1 + \alpha_j x_i)^{-1},
\]
where $\widetilde{J} = \{ j \in \ZZ_{>0} \mid \lambda_j \neq \mu_{j+1} \}$ (note that we have also shifted the indices).
From the combinatorial description of $\wG_{\lambda' \ds \mu'}(\xx_1; \bal)$, we have obtained Theorem~\ref{thm:transition_prob} for a single time step.

\begin{ex}
\label{ex:blocking_B}
As in the previous examples, we consider a system with exactly three particles.
Similarly, we only consider $\UU_3 = (U_1, U_2, U_3)$, which is equivalent to setting $\rho_j = 0$ for all $j > 0$.
We begin by computing for any $\mu$
\begin{align*}
{}^{[\bal]}\bra{e_0(\UU_3 / \bal) \cdot \mu} & = h_0(\bal) \cdot {}^{[\bal]}\bra{e_0(\UU_3) \cdot \mu}
\allowdisplaybreaks
\\ {}^{[\bal]}\bra{e_1(\UU_3 / \bal) \cdot \mu} & = -h_1(\bal) \cdot {}^{[\bal]}\bra{e_0(\UU_3) \cdot \mu} + h_0(\bal) \cdot {}^{[\bal]}\bra{e_1(\UU_3) \cdot \mu}
\allowdisplaybreaks
\\ {}^{[\bal]}\bra{e_2(\UU_3 / \bal) \cdot \mu} & = h_2(\bal) \cdot {}^{[\bal]}\bra{e_0(\UU_3) \cdot \mu} - h_1(\bal) \cdot {}^{[\bal]}\bra{e_1(\UU_3) \cdot \mu} + h_0(\bal) \cdot {}^{[\bal]}\bra{e_2(\UU_3) \cdot \mu}
\allowdisplaybreaks
\\ {}^{[\bal]}\bra{e_k(\UU_3 / \bal) \cdot \mu} & = -h_k(\bal) \cdot {}^{[\bal]}\bra{e_0(\UU_3) \cdot \mu} + h_{k-1}(\bal) \cdot {}^{[\bal]}\bra{e_1(\UU_3) \cdot \mu} - h_{k-2}(\bal) \cdot {}^{[\bal]}\bra{e_2(\UU_3) \cdot \mu}
\\ & \hspace{20pt} + h_{k-3}(\bal) \cdot {}^{[\bal]}\bra{e_3(\UU_3) \cdot \mu}
\end{align*}
for all $k \geq 3$, and so we have
\begin{align*}
{}^{[\bal]} \bra{\mu} e^{J(\xx)}
& = \sum_{k=0}^{\infty} x_i^k {}^{[\bal]}\bra{e_k(\UU_3 / \bal) \cdot \mu}
\\ & = \sum_{m=0}^{\infty} h_m(\bal) x_i^m \Bigl( {}^{[\bal]}\bra{e_0(x_i \UU_3) \cdot \mu} + {}^{[\bal]}\bra{e_1(x_i \UU_3) \cdot \mu}
\\ & \hspace{100pt} + {}^{[\bal]}\bra{e_2(x_i \UU_3) \cdot \mu} + {}^{[\bal]}\bra{e_3(x_i \UU_3) \cdot \mu} \Bigr)
\\ & = \prod_{j=1}^{\infty} (1 + \alpha_j x_i) \cdot {}^{[\bal]}\bra{\mcT_B \cdot \mu}
\\ & = \prod_{j=2}^{\infty} (1 + \rho_j x_i) \cdot {}^{[\bal]}\bra{\mcT_B \cdot \mu}.
\end{align*}
Now we consider the initial positions of the particles to be $\mu = (1,1)$, and from~\eqref{eq:ncsym_e123}, we have
\[
e_1(\UU_3) \cdot \mu = \ydiagram{2,1} + \alpha_1 \ydiagram{1,1} + \ydiagram{1,1,1}\,,
\qquad
e_2(\UU_3) \cdot \mu = \ydiagram{2,2} + \ydiagram{2,1,1} + \alpha_1 \ydiagram{1,1,1}\,,
\qquad
e_3(\UU_3) \cdot \mu = \ydiagram{2,2,1}\,.
\]
Therefore, if we apply $\mcT_B$, we obtain
\begin{align*}
{}^{[\bal]} \bra{\mcT_B \cdot \mu} & = (1 + \alpha_1 x_i) \cdot {}^{[\bal]} \bra{1,1} + x_i \cdot {}^{[\bal]} \bra{2,1} + (x_i + \alpha_1 x_i^2) \cdot {}^{[\bal]} \bra{1,1,1}
\\ & \hspace{20pt} + x_i^2 \cdot {}^{[\bal]} \bra{2,2} + x_i^2 \cdot {}^{[\bal]} \bra{2,1,1} + x_i^3 \cdot {}^{[\bal]} \bra{2,2,1}
\\ & = (1 + \rho_2 x_i) \cdot {}^{[\bal]} \bra{1,1} + \frac{\rho_1 x_i}{\brho^{(2,1) / \mu}} \cdot {}^{[\bal]} \bra{2,1} + \frac{\rho_3 x_i + (\rho_2 x_i)(\rho_3 x_i)}{\brho^{(1,1,1)/\mu}} \cdot {}^{[\bal]} \bra{1,1,1}
\\ & \hspace{20pt} + \frac{(\rho_1 x_i)(\rho_2 x_i)}{\brho^{(2,2)/\mu}} \cdot {}^{[\bal]} \bra{2,2} + \frac{(\rho_1 x_i)(\rho_3 x_i)}{\brho^{(2,1,1)/\mu}} \cdot {}^{[\bal]} \bra{2,1,1}
\\ & \hspace{20pt} + \frac{(\rho_1 x_i)(\rho_2 x_i)(\rho_3 x_i)}{\brho^{(2,2,1)/\mu}} \cdot {}^{[\bal]} \bra{2,2,1}.
\end{align*}
As before, we ignore the normalization constant $C = \prod_{j=1}^3 (1 + \rho_j x_i)^{-1}$ and get the possible states with probabilities
\begin{align*}
1 & \cdot 
\begin{tikzpicture}[scale=0.7,baseline=0cm]
\draw[-] (-0.2,0) -- (3.2,0);
\foreach \x in {0,...,3} {
  \draw[-] (\x,-.1) -- (\x,+.1);
  \draw (\x, 0) node[below] {\tiny $\x$};
}
\draw (3.7, 0) node {$\cdots$};
\fill[darkred] (1,0.1) circle (0.1);
\fill[blue] (1,0.3) circle (0.1);
\fill[UQpurple] (0,0.1) circle (0.1);
\end{tikzpicture}
&
\rho_2 x_i & \cdot 
\begin{tikzpicture}[scale=0.7,baseline=0cm]
\draw[-] (-0.2,0) -- (3.2,0);
\foreach \x in {0,...,3} {
  \draw[-] (\x,-.1) -- (\x,+.1);
  \draw (\x, 0) node[below] {\tiny $\x$};
}
\draw (3.7, 0) node {$\cdots$};
\fill[darkred] (1,0.1) circle (0.1);
\fill[blue] (1,0.3) circle (0.1);
\fill[UQpurple] (0,0.1) circle (0.1);
\draw[->,blue] (1.5,0+.3) -- (1.2,0+.3);
\end{tikzpicture}
\\
\rho_1 x_i & \cdot 
\begin{tikzpicture}[scale=0.7,baseline=0cm]
\draw[-] (-0.2,0) -- (3.2,0);
\foreach \x in {0,...,3} {
  \draw[-] (\x,-.1) -- (\x,+.1);
  \draw (\x, 0) node[below] {\tiny $\x$};
}
\draw (3.7, 0) node {$\cdots$};
\fill[darkred] (2,0.1) circle (0.1);
\fill[blue] (1,0.1) circle (0.1);
\fill[UQpurple] (0,0.1) circle (0.1);
\end{tikzpicture}
&
\rho_3 x_i & \cdot 
\begin{tikzpicture}[scale=0.7,baseline=0cm]
\draw[-] (-0.2,0) -- (3.2,0);
\foreach \x in {0,...,3} {
  \draw[-] (\x,-.1) -- (\x,+.1);
  \draw (\x, 0) node[below] {\tiny $\x$};
}
\draw (3.7, 0) node {$\cdots$};
\fill[darkred] (1,0.1) circle (0.1);
\fill[blue] (1,0.3) circle (0.1);
\fill[UQpurple] (1,0.5) circle (0.1);
\end{tikzpicture}
\\
(\rho_2 x_i) (\rho_3 x_i) & \cdot 
\begin{tikzpicture}[scale=0.7,baseline=0cm]
\draw[-] (-0.2,0) -- (3.2,0);
\foreach \x in {0,...,3} {
  \draw[-] (\x,-.1) -- (\x,+.1);
  \draw (\x, 0) node[below] {\tiny $\x$};
}
\draw (3.7, 0) node {$\cdots$};
\fill[darkred] (1,0.1) circle (0.1);
\fill[blue] (1,0.3) circle (0.1);
\fill[UQpurple] (1,0.5) circle (0.1);
\draw[->,blue] (1.5,0+.3) -- (1.2,0+.3);
\end{tikzpicture}
&
(\rho_1 x_i)(\rho_2 x_i) & \cdot 
\begin{tikzpicture}[scale=0.7,baseline=0cm]
\draw[-] (-0.2,0) -- (3.2,0);
\foreach \x in {0,...,3} {
  \draw[-] (\x,-.1) -- (\x,+.1);
  \draw (\x, 0) node[below] {\tiny $\x$};
}
\draw (3.7, 0) node {$\cdots$};
\fill[darkred] (2,0.1) circle (0.1);
\fill[blue] (2,0.3) circle (0.1);
\fill[UQpurple] (0,0.1) circle (0.1);
\end{tikzpicture}
\\
(\rho_1 x_i)(\rho_3 x_i) & \cdot 
\begin{tikzpicture}[scale=0.7,baseline=0cm]
\draw[-] (-0.2,0) -- (3.2,0);
\foreach \x in {0,...,3} {
  \draw[-] (\x,-.1) -- (\x,+.1);
  \draw (\x, 0) node[below] {\tiny $\x$};
}
\draw (3.7, 0) node {$\cdots$};
\fill[darkred] (2,0.1) circle (0.1);
\fill[blue] (1,0.1) circle (0.1);
\fill[UQpurple] (1,0.3) circle (0.1);
\end{tikzpicture}
&
(\rho_1 x_i)(\rho_2 x_i)(\rho_3 x_i) & \cdot 
\begin{tikzpicture}[scale=0.7,baseline=0cm]
\draw[-] (-0.2,0) -- (3.2,0);
\foreach \x in {0,...,3} {
  \draw[-] (\x,-.1) -- (\x,+.1);
  \draw (\x, 0) node[below] {\tiny $\x$};
}
\draw (3.7, 0) node {$\cdots$};
\fill[darkred] (2,0.1) circle (0.1);
\fill[blue] (2,0.3) circle (0.1);
\fill[UQpurple] (1,0.1) circle (0.1);
\end{tikzpicture}
\end{align*}
where an arrow denotes a blocked particle.
\end{ex}

Now let us look at the dynamics for the geometric distribution given by Case~C.
The proof is similar to the above except we instead replace $e_k(\UU) \mapsto h_k(\UU)$ and $h_k(\bal) \mapsto e_k(\bbe)$, as well as specialize the operators to $\UU := \UU^{(0,\bbe)}$ with $\beta_j = \pi_{j+1}$.
So our time evolution operator is $\mcT_C = \sum_{k=0}^{\infty} h_k(x_i \UU)$.
Note that the operators are applied in reverse order for $h_k(\UU)$, encoding that we are now going from largest-to-smallest in the update order.

For the alternative proof using the combinatorial description of $\G_{\lambda\ds\mu}(\xx; \bbe)$, some slightly more detailed analysis about the motion of the particle is needed.
This is discussed in Section~\ref{sec:combinatorial_C}.
Note that it only depends on $\lambda$ and $\mu$, not on the motion of any other particles.

\begin{ex}
\label{ex:blocking_C}
We will follow the setup in Example~\ref{ex:blocking_B} with three particles, so we take $\pi_j = 0$ for all $j > 3$, with initial positions $\mu = (1, 1)$.
Using~\eqref{eq:ncsym_h123}, we compute
\begin{align*}
h_1(\UU_3) \cdot \mu & = \ydiagram{2,1} + \beta_1 \ydiagram{1,1} + \ydiagram{1,1,1}\,,
\\ h_2(\UU_3) \cdot \mu & = \ydiagram{3,1} + \beta_1 \ydiagram{2,1} + \ydiagram{2,1,1} + \beta_1^2 \ydiagram{1,1}+ \beta_1 \ydiagram{1,1,1} + \beta_2 \ydiagram{1,1,1}\,,
\\ h_3(\UU_3) \cdot \mu & = \ydiagram{4,1} + \beta_1 \ydiagram{3,1} + \ydiagram{3,1,1} + \beta_1^2 \ydiagram{2,1}+ \beta_1 \ydiagram{2,1,1} + \beta_2 \ydiagram{2,1,1}\,
\\ & \hspace{20pt} + \beta_1^3 \ydiagram{1,1} + \beta_1^2 \ydiagram{1,1,1} + \beta_1 \beta_2 \ydiagram{1,1,1} + \beta_2^2 \ydiagram{1,1,1}\,.
\end{align*}
Therefore, we have
\begin{align*}
{}^{[\bbe]} \bra{\mcT_C \cdot \mu} & = (1 + \pi_2 x_i + (\pi_2 x_i)^2 + (\pi_2 x_i)^3 + \cdots) \cdot {}^{[\bbe]} \bra{1,1}
\\ & \hspace{20pt} + x_i (1 + \pi_2 x_i + (\pi_2 x_i)^2 + \cdots) \cdot {}^{[\bbe]} \bra{2,1}
\\ & \hspace{20pt} + x_i (1 + (\pi_2 + \pi_3) x_i + (\pi_2^2 + \pi_2 \pi_3 + \pi_3^2) x_i^2 + \cdots) \cdot {}^{[\bbe]} \bra{1,1,1}
\\ & \hspace{20pt} + x_i^2 (1 + \pi_2 x_i + (\pi_2 x_i)^2 + \cdots) \cdot {}^{[\bbe]} \bra{3,1}
\\ & \hspace{20pt} + x_i^2 (1 + (\pi_2 + \pi_3) x_i + (\pi_2^2 + \pi_2 \pi_3 + \pi_3^2) x_i^2 + \cdots) \cdot {}^{[\bbe]} \bra{2,1,1}
\\ & \hspace{20pt} + x_i^3 (1 + (\pi_2 + \pi_3) x_i + (\pi_2^2 + \pi_2 \pi_3 + \pi_3^2) x_i^2 + \cdots) \cdot {}^{[\bbe]} \bra{3,1,1} + \cdots
\allowdisplaybreaks
\\ & = \frac{1}{1 - \pi_2 x_i} \cdot {}^{[\bbe]} \bra{1,1} + \frac{x_i}{1 - \pi_2 x_i} \cdot {}^{[\bbe]} \bra{2,1} + \frac{x_i}{(1 - \pi_2 x_i)(1 - \pi_3 x_i)} \cdot {}^{[\bbe]} \bra{1,1,1}
\\ & \hspace{20pt} + \frac{x_i^2}{1 - \pi_2 x_i} \cdot {}^{[\bbe]} \bra{3,1} + \frac{x_i^2}{(1 - \pi_2 x_i)(1 - \pi_3 x_i)} \cdot {}^{[\bbe]} \bra{2,1,1}
\\ & \hspace{20pt} + \frac{x_i^3}{(1 - \pi_2 x_i)(1 - \pi_3 x_i)} \cdot {}^{[\bbe]} \bra{3,1,1} + \cdots
\allowdisplaybreaks
\\ & = \frac{1}{1 - \pi_2 x_i} \cdot {}^{[\bbe]} \bra{1,1} + \frac{\pi_1 x_i}{(1 - \pi_2 x_i) \bpi^{(2,1) / \mu}} \cdot {}^{[\bbe]} \bra{2,1}
\\ & \hspace{20pt} + \frac{\pi_3 x_i}{(1 - \pi_2 x_i)(1 - \pi_3 x_i)\bpi^{(1,1,1) / \mu}} \cdot {}^{[\bbe]} \bra{1,1,1} + \frac{(\pi_1 x_i)^2}{(1 - \pi_2 x_i)\bpi^{(3,1) / \mu}} \cdot {}^{[\bbe]} \bra{3,1}
\\ & \hspace{20pt} + \frac{(\pi_1 x_i)(\pi_3 x_i)}{(1 - \pi_2 x_i)(1 - \pi_3 x_i)\bpi^{(2,1,1) / \mu}} {}^{[\bbe]} \bra{2,1,1}
\\ & \hspace{20pt} + \frac{(\pi_1 x_i)^2(\pi_3 x_i)}{(1 - \pi_2 x_i)(1 - \pi_3 x_i) \bpi^{(3,1,1)/\mu}} \cdot {}^{[\bbe]} \bra{3,1,1} + \cdots.
\end{align*}
Ignoring the normalization factor only for the first particle $C = 1 - \pi_1 x_i$, we see that some of the possible states with probabilities are
\begin{align*}
\frac{1}{1 - \pi_2 x_i} & \cdot 
\begin{tikzpicture}[scale=0.7,baseline=0cm]
\draw[-] (-0.2,0) -- (3.2,0);
\foreach \x in {0,...,3} {
  \draw[-] (\x,-.1) -- (\x,+.1);
  \draw (\x, 0) node[below] {\tiny $\x$};
}
\draw (3.7, 0) node {$\cdots$};
\fill[darkred] (1,0.1) circle (0.1);
\fill[blue] (1,0.3) circle (0.1);
\fill[UQpurple] (0,0.1) circle (0.1);
\draw[->,blue] (1.5,0+.3) -- (1.2,0+.3);
\end{tikzpicture}
&
\frac{\pi_1 x_i}{1 - \pi_2 x_i} & \cdot 
\begin{tikzpicture}[scale=0.7,baseline=0cm]
\draw[-] (-0.2,0) -- (3.2,0);
\foreach \x in {0,...,3} {
  \draw[-] (\x,-.1) -- (\x,+.1);
  \draw (\x, 0) node[below] {\tiny $\x$};
}
\draw (3.7, 0) node {$\cdots$};
\fill[darkred] (2,0.1) circle (0.1);
\fill[blue] (1,0.1) circle (0.1);
\fill[UQpurple] (0,0.1) circle (0.1);
\draw[->,blue] (1.5,0+.1) -- (1.2,0+.1);
\end{tikzpicture}
\\
\frac{\pi_3 x_i}{(1 - \pi_2 x_i)(1 - \pi_3 x_i)} & \cdot 
\begin{tikzpicture}[scale=0.7,baseline=0cm]
\draw[-] (-0.2,0) -- (3.2,0);
\foreach \x in {0,...,3} {
  \draw[-] (\x,-.1) -- (\x,+.1);
  \draw (\x, 0) node[below] {\tiny $\x$};
}
\draw (3.7, 0) node {$\cdots$};
\fill[darkred] (1,0.1) circle (0.1);
\fill[blue] (1,0.3) circle (0.1);
\fill[UQpurple] (1,0.5) circle (0.1);
\draw[->,blue] (1.5,0+.3) -- (1.2,0+.3);
\draw[->,UQpurple] (1.5,0+.5) -- (1.2,0+.5);
\end{tikzpicture}
&
\frac{(\pi_1 x_i)^2}{1 - \pi_2 x_i} & \cdot 
\begin{tikzpicture}[scale=0.7,baseline=0cm]
\draw[-] (-0.2,0) -- (3.2,0);
\foreach \x in {0,...,3} {
  \draw[-] (\x,-.1) -- (\x,+.1);
  \draw (\x, 0) node[below] {\tiny $\x$};
}
\draw (3.7, 0) node {$\cdots$};
\fill[darkred] (3,0.1) circle (0.1);
\fill[blue] (1,0.1) circle (0.1);
\fill[UQpurple] (0,0.1) circle (0.1);
\draw[->,blue] (1.5,0+.1) -- (1.2,0+.1);
\end{tikzpicture}
\\
\frac{(\pi_1 x_i)(\pi_3 x_i)}{(1 - \pi_2 x_i)(1 - \pi_3 x_i)} & \cdot 
\begin{tikzpicture}[scale=0.7,baseline=0cm]
\draw[-] (-0.2,0) -- (3.2,0);
\foreach \x in {0,...,3} {
  \draw[-] (\x,-.1) -- (\x,+.1);
  \draw (\x, 0) node[below] {\tiny $\x$};
}
\draw (3.7, 0) node {$\cdots$};
\fill[darkred] (2,0.1) circle (0.1);
\fill[blue] (1,0.1) circle (0.1);
\fill[UQpurple] (1,0.3) circle (0.1);
\draw[->,blue] (1.5,0+.1) -- (1.2,0+.1);
\draw[->,UQpurple] (1.5,0+.3) -- (1.2,0+.3);
\end{tikzpicture}
&
\frac{(\pi_1 x_i)^2(\pi_3 x_i)}{(1 - \pi_2 x_i)(1 - \pi_3 x_i)} & \cdot 
\begin{tikzpicture}[scale=0.7,baseline=0cm]
\draw[-] (-0.2,0) -- (3.2,0);
\foreach \x in {0,...,3} {
  \draw[-] (\x,-.1) -- (\x,+.1);
  \draw (\x, 0) node[below] {\tiny $\x$};
}
\draw (3.7, 0) node {$\cdots$};
\fill[darkred] (3,0.1) circle (0.1);
\fill[blue] (1,0.1) circle (0.1);
\fill[UQpurple] (1,0.3) circle (0.1);
\draw[->,blue] (1.5,0+.1) -- (1.2,0+.1);
\draw[->,UQpurple] (1.5,0+.3) -- (1.2,0+.3);
\end{tikzpicture}
\end{align*}
where an arrow denotes a blocked particle.
\end{ex}

\section{Bijective description}
\label{sec:bijection}

In this section, we provide a bijective proof of Theorem~\ref{thm:transition_prob}.
For Cases~A and~D, this is essentially translating the description of the noncommutative operators $\uu$ into tableaux through how the particles evolve.
For Cases~B and~C, a little more care is needed as the number of tableaux is not in bijection with the number of intermediate states.
However, as in Section~\ref{sec:operator_dynamics}, we reduce the general case to matching the behavior under one time step evolution.

\subsection{Case A: Geometric pushing}
\label{sec:bijection_A}

This case was discussed in~\cite{MS20} by going through the last passage percolation (LPP) model.
To make this combinatorially explicit, we simply note that the resulting matrix $[G(k,n)]_{k,n}$ is the analog of a Gelfand--Tsetlin pattern for the reverse plane partition.
More precisely, the $n$-th column gives the shape of the entries at most $n$.
From this description, we essentially have a combinatorially proof of Case~A of Theorem~\ref{thm:transition_prob}.
The only other ingredient needed is to note that in $\bpi^{\lambda/\mu} \dG_{\lambda/\mu}(\xx; \bpi^{-1})$, we get a contribution of $\pi_j x_i$ for a box with a $i$ in row $j$; recalling that we align the entries at the bottom of merged cells.
In terms of the more classical description of reverse plane partitions, we are only counting boxes with a $i$ that is not above another box with label $i$.
The remaining factor of $\prod_{j=1}^{\ell} \prod_{i=1}^n (1 - \pi_j x_i)$ comes from the normalization factor in the geometric distribution.

We can make this very precise with a direct correlation between the movement of particles and entries in the reverse plane partition.
An entry $i$ in row $j$ (that is not a merged box) corresponds to the $j$-th particle moving a step at time $i$.
Thus (ignoring the common normalization factor), by conditioning the $j$-th particle at time $i$ moving $k$ steps, then the reverse plane partition has $k$ entries with value $i$ at row $j$ (with no $i$ directly below it).
The general case of multiple particles moving follows fundamental facts of conditional probability.
Hence, we have shown the one step transition probability at time $i$
\[
\prob_{A,1}(\lambda | \mu) = \prod_{j=1}^{\ell}  (1 - \pi_j x_i) \bpi^{\lambda/\mu} \prod_{j=1}^{\ell-1} \pi_j^{-r_j} \wt(T)
=\prod_{j=1}^{\ell}  (1 - \pi_j x_i) \bpi^{\lambda/\mu} g_{\lambda/\mu}(x_i;\bpi^{-1})
,
\]
where $T$ is the \emph{unique} reverse plane partition of skew shape $\lambda / \mu$ with all boxes filled with $i$.
In some more detail about the expression,
recall $r_j$ is the number of boxes of $T$ in row $j$ that have been merged with the box below,
and those boxes correspond to the moves of the $j$-th particle, automatically pushed from the particle behind it.
This means that the number $r_j$ corresponds to the distance coming from the push and
has to be subtracted from the actual total distance $\lambda_j-\mu_j$ of the $j$-th particle's move in the power of $\pi_j$.
This gives the factor $
\prod_{j=1}^{\ell-1} \pi_j^{\lambda_j-\mu_j-r_j} \times \pi_\ell^{\lambda_\ell-\mu_\ell}=\bpi^{\lambda/\mu} \prod_{j=1}^{\ell-1} \pi_j^{-r_j}
$.
We also have a factor which is a power of $x_i$, and the power is equal to the total degree of $\bpi^{\lambda/\mu} \prod_{j=1}^{\ell-1} \pi_j^{-r_j}
$. We conclude the precise factor is $x_i^{|\lambda/\mu|-\sum_{j=1}^{\ell-1} r_j}$.
Since all boxes of the reverse plane partition are filled with the same number, we note $|\lambda/\mu|-\sum_{j=1}^{\ell-1} r_j$
is equal to the total number of columns of the skew shape $\lambda/\mu$,
and $x_i^{|\lambda/\mu|-\sum_{j=1}^{\ell-1} r_j}$ is exactly $\wt(T)$ for one variable case (recall we have fused boxes for the weight).

Hence, we have the reverse plane partition exactly encodes the movement of all of the particles at time $i = 0, 1, \dotsc, n$ by the branching rules (Proposition~\ref{prop:branching_rules}) and the Markov property as described at the beginning of Section~\ref{sec:operator_dynamics} (or basic facts of conditional probability).

\begin{ex}
We consider Case~A with $\lambda = 31$ and $n = 2$.
Hence, we are considering $\ell$ particles the move over two time steps.
Since all but the first two particles are fixed, we can ignore them.
We have the following reverse plane partitions and states, where we have drawn the merged boxes in gray.
An arrow denotes that a particle was pushed.
\[
\ytableausetup{boxsize=1.2em}
\begin{array}{*{8}{@{\hspace{5pt}}c}}
\ytableaushort{111,2}
&
\ytableaushort{112,2}
&
\ytableaushort{122,2}
&
\ytableaushort{{\color{gray}1}11,1}
&
\ytableaushort{{\color{gray}1}12,1}
&
\ytableaushort{{\color{gray}1}22,1}
&
\ytableaushort{{\color{gray}2}22,2}
\\[15pt]
\begin{tikzpicture}[scale=0.6]
\foreach \i in {0,1,2} {
  \draw[-] (-0.2,-\i) -- (3.2,-\i);
  \foreach \x in {0,1,2,3}
    \draw[-] (\x,-\i-.1) -- (\x,-\i+.1);
}
\fill[darkred] (0,-0+0.1) circle (0.1);
\fill[blue] (0,0.3) circle (0.1);
\fill[darkred] (3,-1+0.1) circle (0.1);
\fill[blue] (0,-1+0.1) circle (0.1);
\fill[darkred] (3,-2+0.1) circle (0.1);
\fill[blue] (1,-2+0.1) circle (0.1);
\end{tikzpicture}
&
\begin{tikzpicture}[scale=0.6]
\foreach \i in {0,1,2} {
  \draw[-] (-0.2,-\i) -- (3.2,-\i);
  \foreach \x in {0,1,2,3}
    \draw[-] (\x,-\i-.1) -- (\x,-\i+.1);
}
\fill[darkred] (0,-0+0.1) circle (0.1);
\fill[blue] (0,0.3) circle (0.1);
\fill[darkred] (2,-1+0.1) circle (0.1);
\fill[blue] (0,-1+0.1) circle (0.1);
\fill[darkred] (3,-2+0.1) circle (0.1);
\fill[blue] (1,-2+0.1) circle (0.1);
\end{tikzpicture}
&
\begin{tikzpicture}[scale=0.6]
\foreach \i in {0,1,2} {
  \draw[-] (-0.2,-\i) -- (3.2,-\i);
  \foreach \x in {0,1,2,3}
    \draw[-] (\x,-\i-.1) -- (\x,-\i+.1);
}
\fill[darkred] (0,-0+0.1) circle (0.1);
\fill[blue] (0,0.3) circle (0.1);
\fill[darkred] (1,-1+0.1) circle (0.1);
\fill[blue] (0,-1+0.1) circle (0.1);
\fill[darkred] (3,-2+0.1) circle (0.1);
\fill[blue] (1,-2+0.1) circle (0.1);
\end{tikzpicture}
&
\begin{tikzpicture}[scale=0.6]
\foreach \i in {0,1,2} {
  \draw[-] (-0.2,-\i) -- (3.2,-\i);
  \foreach \x in {0,1,2,3}
    \draw[-] (\x,-\i-.1) -- (\x,-\i+.1);
}
\fill[darkred] (0,-0+0.1) circle (0.1);
\fill[blue] (0,0.3) circle (0.1);
\draw[->,darkred] (0.2,-1+.1) -- (0.8,-1+.1);
\fill[darkred] (3,-1+0.1) circle (0.1);
\fill[blue] (1,-1+0.1) circle (0.1);
\fill[darkred] (3,-2+0.1) circle (0.1);
\fill[blue] (1,-2+0.1) circle (0.1);
\end{tikzpicture}
&
\begin{tikzpicture}[scale=0.6]
\foreach \i in {0,1,2} {
  \draw[-] (-0.2,-\i) -- (3.2,-\i);
  \foreach \x in {0,1,2,3}
    \draw[-] (\x,-\i-.1) -- (\x,-\i+.1);
}
\fill[darkred] (0,-0+0.1) circle (0.1);
\fill[blue] (0,0.3) circle (0.1);
\draw[->,darkred] (0.2,-1+.1) -- (0.8,-1+.1);
\fill[darkred] (2,-1+0.1) circle (0.1);
\fill[blue] (1,-1+0.1) circle (0.1);
\fill[darkred] (3,-2+0.1) circle (0.1);
\fill[blue] (1,-2+0.1) circle (0.1);
\end{tikzpicture}
&
\begin{tikzpicture}[scale=0.6]
\foreach \i in {0,1,2} {
  \draw[-] (-0.2,-\i) -- (3.2,-\i);
  \foreach \x in {0,1,2,3}
    \draw[-] (\x,-\i-.1) -- (\x,-\i+.1);
}
\fill[darkred] (0,-0+0.1) circle (0.1);
\fill[blue] (0,0.3) circle (0.1);
\draw[->,darkred] (0.2,-1+.1) -- (0.8,-1+.1);
\fill[darkred] (1,-1+0.1) circle (0.1);
\fill[blue] (1,-1+0.3) circle (0.1);
\fill[darkred] (3,-2+0.1) circle (0.1);
\fill[blue] (1,-2+0.1) circle (0.1);
\end{tikzpicture}
&
\begin{tikzpicture}[scale=0.6]
\foreach \i in {0,1,2} {
  \draw[-] (-0.2,-\i) -- (3.2,-\i);
  \foreach \x in {0,1,2,3}
    \draw[-] (\x,-\i-.1) -- (\x,-\i+.1);
}
\fill[darkred] (0,-0+0.1) circle (0.1);
\fill[blue] (0,0.3) circle (0.1);
\fill[darkred] (0,-1+0.1) circle (0.1);
\fill[blue] (0,-1+0.3) circle (0.1);
\draw[->,darkred] (0.2,-2+.1) -- (0.8,-2+.1);
\fill[darkred] (3,-2+0.1) circle (0.1);
\fill[blue] (1,-2+0.1) circle (0.1);
\end{tikzpicture}
\\
\pi_1^3 x_1^3 \cdot \pi_2 x_2
&
\pi_1^2 x_1^2 \cdot \pi_1 \pi_2 x_2^2
&
\pi_1 x_1 \cdot \pi_1^2 \pi_2 x_2^3
&
\pi_1^2 \pi_2 x_1^3 \cdot 1
&
\pi_1 \pi_2 x_1^2 \cdot \pi_1 x_2
&
\pi_2 x_1 \cdot \pi_1^2 x_2^2
&
1 \cdot \pi_1^2 \pi_2 x_2^3
\end{array}
\]
Below each configuration, we have written the conditional probability 
except without the normalization constant in the geometric distribution $(1 - \pi_j x_i)$.
Note that these are precisely the terms appearing in $\bpi^{\lambda} \dG_{\lambda}(\xx; \bpi^{-1})$.
\end{ex}

\subsection{Case D: Bernoulli pushing}

Algebraically, this case is simply applying the involution $\omega$ defined by $\omega s_{\lambda} = s_{\lambda'}$.
As such, we should expect the movement of the $j$-th particle to correspond to the entries in the $j$-th column, as opposed to the $j$-th row in Case~A.
Indeed, this is the case, but we need to reformulate the parameters for the particle process.
Recalling from~\cite[Thm.~5.11]{PP16} (see also~\cite[Prop.~3.4]{Yel17}), we can write $\wG_{\lambda/\mu}(\xx; \alpha) = \G_{\lambda/\mu}\bigl(\xx; \alpha / (1 + \alpha \xx) \bigr)$, it becomes natural to instead consider $\pi_j x_i$ as a rate the particle moves rather than the success probability, which becomes $\frac{\pi_j x_i}{1 + \pi_j x_i}$.
We also see this in the normalization factor as
\[
\omega \left( \prod_{j=1}^{\ell} \prod_{i=1}^n (1 - \pi_j x_i) \right) = \prod_{j=1}^{\ell} \prod_{i=1}^n (1 + \pi_j x_i)^{-1}.
\]
Since we are considering it as a rate, we use a different set of parameters $\brho$ since we simply require $\rho_j x_i > 0$ for all $i$ and $j$, as opposed to $\pi_j x_i \in (0, 1)$.
The remainder of the proof is that after we factor our the denominators, we make the same observation in Case~A that an entry $i$ in column $j$ corresponds to the $j$-th particle moving one step at time $i$ and contributes $\rho_j x_i$ to the (unnormalized) probability.

\begin{ex}
Let us consider Case~D for $\lambda = 31$ with $n = 2$.
Thus, three particles (all others are fixed) move over two time steps.
The possible terms, states, and valued-set tableaux are
\[
\begin{array}{*{8}{@{\hspace{2pt}}c}}
\begin{tikzpicture}[scale=0.5,baseline=-30]
\foreach \i/\r in {0/3,1/3,2/1}
    \draw (0,-\i) -- (\r,-\i);
\foreach \i in {0,1}
  \draw (\i,0) -- (\i,-2);
\foreach \i in {2,3}
  \draw (\i,0) -- (\i,-1);
\draw (0.5,-1.5) node {$\mathbf{\color{darkred}2}$};
\foreach \i in {0,1,2}
  \draw (\i+.5,-0.5) node {$\mathbf{\color{darkred}1}$};
\end{tikzpicture}
&
\begin{tikzpicture}[scale=0.5,baseline=-30]
\foreach \i/\r in {0/3,1/3,2/1}
    \draw (0,-\i) -- (\r,-\i);
\foreach \i in {0,1}
  \draw (\i,0) -- (\i,-2);
\foreach \i in {2,3}
  \draw (\i,0) -- (\i,-1);
\draw (0.5,-1.5) node {$\mathbf{\color{darkred}2}$};
\foreach \i in {0,1}
  \draw (\i+.5,-0.5) node {$\mathbf{\color{darkred}1}$};
\foreach \i in {2}
  \draw (\i+.5,-0.5) node {$\mathbf{\color{darkred}2}$};
\end{tikzpicture}
&
\begin{tikzpicture}[scale=0.5,baseline=-30]
\foreach \i/\r in {0/3,1/3,2/1}
    \draw (0,-\i) -- (\r,-\i);
\foreach \i in {0,1}
  \draw (\i,0) -- (\i,-2);
\foreach \i in {2,3}
  \draw (\i,0) -- (\i,-1);
\draw (0.5,-1.5) node {$\mathbf{\color{darkred}2}$};
\foreach \i in {0}
  \draw (\i+.5,-0.5) node {$\mathbf{\color{darkred}1}$};
\foreach \i in {1,2}
  \draw (\i+.5,-0.5) node {$\mathbf{\color{darkred}2}$};
\end{tikzpicture}
&
\begin{tikzpicture}[scale=0.5,baseline=-30]
\foreach \i/\r in {0/3,1/3,2/1}
    \draw (0,-\i) -- (\r,-\i);
\draw (0,0) -- (0,-2);
\draw (1,-1) -- (1,-2);
\foreach \i in {2,3}
  \draw (\i,0) -- (\i,-1);
\draw (0.5,-1.5) node {$\mathbf{\color{darkred}2}$};
\draw (0.5,-0.5) node[gray] {$1$};
\draw (1.5,-0.5) node {$\mathbf{\color{darkred}1}$};
\draw (2.5,-0.5) node {$\mathbf{\color{darkred}1}$};
\end{tikzpicture}
&
\begin{tikzpicture}[scale=0.5,baseline=-30]
\foreach \i/\r in {0/3,1/3,2/1}
    \draw (0,-\i) -- (\r,-\i);
\draw (0,0) -- (0,-2);
\draw (1,-1) -- (1,-2);
\foreach \i in {2,3}
  \draw (\i,0) -- (\i,-1);
\draw (0.5,-1.5) node {$\mathbf{\color{darkred}2}$};
\draw (0.5,-0.5) node[gray] {$1$};
\draw (1.5,-0.5) node {$\mathbf{\color{darkred}1}$};
\draw (2.5,-0.5) node {$\mathbf{\color{darkred}2}$};
\end{tikzpicture}
&
\begin{tikzpicture}[scale=0.5,baseline=-30]
\foreach \i/\r in {0/3,1/3,2/1}
    \draw (0,-\i) -- (\r,-\i);
\draw (0,0) -- (0,-2);
\draw (1,0) -- (1,-2);
\draw (3,0) -- (3,-1);
\draw (0.5,-1.5) node {$\mathbf{\color{darkred}2}$};
\draw (0.5,-0.5) node {$\mathbf{\color{darkred}1}$};
\draw (1.5,-0.5) node[gray] {$1$};
\draw (2.5,-0.5) node {$\mathbf{\color{darkred}1}$};
\end{tikzpicture}
&
\begin{tikzpicture}[scale=0.5,baseline=-30]
\foreach \i/\r in {0/3,1/3,2/1}
    \draw (0,-\i) -- (\r,-\i);
\draw (0,0) -- (0,-2);
\draw (1,0) -- (1,-2);
\draw (3,0) -- (3,-1);
\draw (0.5,-1.5) node {$\mathbf{\color{darkred}2}$};
\draw (0.5,-0.5) node {$\mathbf{\color{darkred}1}$};
\draw (1.5,-0.5) node[gray] {$2$};
\draw (2.5,-0.5) node {$\mathbf{\color{darkred}2}$};
\end{tikzpicture}
&
\begin{tikzpicture}[scale=0.5,baseline=-30]
\foreach \i/\r in {0/3,1/3,2/1}
    \draw (0,-\i) -- (\r,-\i);
\draw (0,0) -- (0,-2);
\draw (1,-1) -- (1,-2);
\draw (3,0) -- (3,-1);
\draw (0.5,-1.5) node {$\mathbf{\color{darkred}2}$};
\foreach \i in {0,1}
  \draw (\i+.5,-0.5) node[gray] {$1$};
\draw (2.5,-0.5) node {$\mathbf{\color{darkred}1}$};
\end{tikzpicture}
\\
\begin{tikzpicture}[scale=0.7]
\foreach \i in {0,1,2} {
  \draw[-] (-0.2,-\i) -- (2.2,-\i);
  \foreach \x in {0,1,2}
    \draw[-] (\x,-\i-.1) -- (\x,-\i+.1);
}
\foreach \t in {0,1,2}
  \fill[darkred] (\t,-\t+0.1) circle (0.1);
\fill[blue] (0,0.3) circle (0.1);
\fill[UQpurple] (0,0.5) circle (0.1);
\fill[blue] (1,-1+0.3) circle (0.1);
\fill[UQpurple] (1,-1+0.5) circle (0.1);
\fill[blue] (1,-2+0.1) circle (0.1);
\fill[UQpurple] (1,-2+0.3) circle (0.1);
\end{tikzpicture}
&
\begin{tikzpicture}[scale=0.7]
\foreach \i in {0,1,2} {
  \draw[-] (-0.2,-\i) -- (2.2,-\i);
  \foreach \x in {0,1,2}
    \draw[-] (\x,-\i-.1) -- (\x,-\i+.1);
}
\foreach \t in {0,1,2}
  \fill[darkred] (\t,-\t+0.1) circle (0.1);
\fill[blue] (0,0.3) circle (0.1);
\fill[UQpurple] (0,0.5) circle (0.1);
\fill[blue] (1,-1+0.3) circle (0.1);
\fill[UQpurple] (0,-1+0.1) circle (0.1);
\fill[blue] (1,-2+0.1) circle (0.1);
\fill[UQpurple] (1,-2+0.3) circle (0.1);
\end{tikzpicture}
&
\begin{tikzpicture}[scale=0.7]
\foreach \i in {0,1,2} {
  \draw[-] (-0.2,-\i) -- (2.2,-\i);
  \foreach \x in {0,1,2}
    \draw[-] (\x,-\i-.1) -- (\x,-\i+.1);
}
\foreach \t in {0,1,2}
  \fill[darkred] (\t,-\t+0.1) circle (0.1);
\fill[blue] (0,0.3) circle (0.1);
\fill[UQpurple] (0,0.5) circle (0.1);
\fill[blue] (0,-1+0.1) circle (0.1);
\fill[UQpurple] (0,-1+0.3) circle (0.1);
\fill[blue] (1,-2+0.1) circle (0.1);
\fill[UQpurple] (1,-2+0.3) circle (0.1);
\end{tikzpicture}
&
\begin{tikzpicture}[scale=0.7]
\foreach \i in {0,1,2} {
  \draw[-] (-0.2,-\i) -- (2.2,-\i);
  \foreach \x in {0,1,2}
    \draw[-] (\x,-\i-.1) -- (\x,-\i+.1);
}
\foreach \t in {0,1,2}
  \fill[darkred] (\t,-\t+0.1) circle (0.1);
\fill[blue] (0,0.3) circle (0.1);
\draw[->,darkred] (0.2,-1+.1) -- (0.8,-1+.1);
\fill[UQpurple] (0,0.5) circle (0.1);
\fill[blue] (1,-1+0.3) circle (0.1);
\fill[UQpurple] (1,-1+0.5) circle (0.1);
\fill[blue] (1,-2+0.1) circle (0.1);
\fill[UQpurple] (1,-2+0.3) circle (0.1);
\end{tikzpicture}
&
\begin{tikzpicture}[scale=0.7]
\foreach \i in {0,1,2} {
  \draw[-] (-0.2,-\i) -- (2.2,-\i);
  \foreach \x in {0,1,2}
    \draw[-] (\x,-\i-.1) -- (\x,-\i+.1);
}
\foreach \t in {0,1,2}
  \fill[darkred] (\t,-\t+0.1) circle (0.1);
\fill[blue] (0,0.3) circle (0.1);
\fill[UQpurple] (0,0.5) circle (0.1);
\draw[->,darkred] (0.2,-1+.1) -- (0.8,-1+.1);
\fill[blue] (1,-1+0.3) circle (0.1);
\fill[UQpurple] (0,-1+0.1) circle (0.1);
\fill[blue] (1,-2+0.1) circle (0.1);
\fill[UQpurple] (1,-2+0.3) circle (0.1);
\end{tikzpicture}
&
\begin{tikzpicture}[scale=0.7]
\foreach \i in {0,1,2} {
  \draw[-] (-0.2,-\i) -- (2.2,-\i);
  \foreach \x in {0,1,2}
    \draw[-] (\x,-\i-.1) -- (\x,-\i+.1);
}
\foreach \t in {0,1,2}
  \fill[darkred] (\t,-\t+0.1) circle (0.1);
\fill[blue] (0,0.3) circle (0.1);
\fill[UQpurple] (0,0.5) circle (0.1);
\draw[->,blue] (0.2,-1+.3) -- (0.8,-1+.3);
\fill[blue] (1,-1+0.3) circle (0.1);
\fill[UQpurple] (1,-1+0.5) circle (0.1);
\fill[blue] (1,-2+0.1) circle (0.1);
\fill[UQpurple] (1,-2+0.3) circle (0.1);
\end{tikzpicture}
&
\begin{tikzpicture}[scale=0.7]
\foreach \i in {0,1,2} {
  \draw[-] (-0.2,-\i) -- (2.2,-\i);
  \foreach \x in {0,1,2}
    \draw[-] (\x,-\i-.1) -- (\x,-\i+.1);
}
\foreach \t in {0,1,2}
  \fill[darkred] (\t,-\t+0.1) circle (0.1);
\fill[blue] (0,0.3) circle (0.1);
\fill[UQpurple] (0,0.5) circle (0.1);
\fill[blue] (0,-1+0.1) circle (0.1);
\fill[UQpurple] (0,-1+0.3) circle (0.1);
\draw[->,blue] (0.2,-2+.1) -- (0.8,-2+.1);
\fill[blue] (1,-2+0.1) circle (0.1);
\fill[UQpurple] (1,-2+0.3) circle (0.1);
\end{tikzpicture}
&
\begin{tikzpicture}[scale=0.7]
\foreach \i in {0,1,2} {
  \draw[-] (-0.2,-\i) -- (2.2,-\i);
  \foreach \x in {0,1,2}
    \draw[-] (\x,-\i-.1) -- (\x,-\i+.1);
}
\foreach \t in {0,1,2}
  \fill[darkred] (\t,-\t+0.1) circle (0.1);
\fill[blue] (0,0.3) circle (0.1);
\fill[UQpurple] (0,0.5) circle (0.1);
\draw[->,darkred] (0.2,-1+.1) -- (0.8,-1+.1);
\draw[->,blue] (0.2,-1+.3) -- (0.8,-1+.3);
\fill[blue] (1,-1+0.3) circle (0.1);
\fill[UQpurple] (1,-1+0.5) circle (0.1);
\fill[blue] (1,-2+0.1) circle (0.1);
\fill[UQpurple] (1,-2+0.3) circle (0.1);
\end{tikzpicture}
\\
\begin{array}{c@{}c@{}c} \pi_{11} & \pi_{21} & \pi_{31} \\ \pi_{12} & \sigma_{22} & \sigma_{32} \end{array}
&
\begin{array}{c@{}c@{}c} \pi_{11} & \pi_{21} & \sigma_{31} \\ \pi_{12} & \sigma_{22} & \pi_{32} \end{array}
&
\begin{array}{c@{}c@{}c} \pi_{11} & \sigma_{21} & \sigma_{31} \\ \pi_{12} & \pi_{22} & \pi_{32} \end{array}
&
\begin{array}{c@{}c@{}c} \sigma_{11} & \pi_{21} & \pi_{31} \\ \pi_{12} & \sigma_{22} & \sigma_{32} \end{array}
&
\begin{array}{c@{}c@{}c} \sigma_{11} & \pi_{21} & \sigma_{31} \\ \pi_{12} & \sigma_{22} & \pi_{32} \end{array}
&
\begin{array}{c@{}c@{}c} \pi_{11} & \sigma_{21} & \pi_{31} \\ \pi_{12} & \sigma_{22} & \sigma_{32} \end{array}
&
\begin{array}{c@{}c@{}c} \pi_{11} & \sigma_{21} & \sigma_{31} \\ \pi_{12} & \sigma_{22} & \pi_{32} \end{array}
&
\begin{array}{c@{}c@{}c} \sigma_{11} & \sigma_{21} & \pi_{31} \\ \pi_{12} & \sigma_{22} & \sigma_{32} \end{array}
\end{array}
\]
where $\pi_{ji} = \frac{\rho_j x_i}{1 + \rho_j x_i}$ and $\sigma_{ji} = 1 - \pi_{ji} = \frac{1}{1 + \rho_j x_i}$.
Once we factor out the denominators, we have $\brho^{\lambda'/\mu'} \dwG_{\lambda/\mu}(\xx; \brho^{-1})$.
In the factors considered below each state, taking it as a $\{0,1\}$-matrix by setting $\pi_{ij} = 1$ and $\sigma_{ij} = 0$, we obtain the transposed LPP $\{0,1\}$-matrix.
\end{ex}

Furthermore, the above description of the weight contributions makes a clear connection with $01$-matrices in~\cite{DW08}.
Indeed, we can equate each state with a $01$-matrix $[M_{ij}]_{i,j}$ by setting $M_{ij} = 1$ if and only if we use $\pi_{ij}$ (and necessarily $M_{ij} = 0$ corresponds to $\sigma_{ij}$).
Since all entries of the matrix and particle motions are done independently, the aforementioned correspondence means the transition probabilities are equal.

\subsection{Case C: Geometric blocking}
\label{sec:combinatorial_C}

Now we consider the TASEP with the more classical blocking behavior, and as before, we will proceed by conditioning on the motion of particles.
As we have the blocking behavior, we expect the first particle to behave differently than all of the subsequent particles.
This justifies why the normalization factor only involves $\pi_1$ rather than all $\pi_j$.
Like in Case~A, we will identify entries $i$ in row $j$ corresponding to the $j$-th particle moving at time $i$.
As the first particle will never be blocked, it simply moves at times $i_1 \leq i_2 \leq \cdots \leq i_{\lambda_1}$.

Next, let us consider the movement of the $j$-th particle for $j > 1$.
Suppose the first (resp.\ second) particle at time $i-1$ is $p_{j-1}$ (resp.\ $p_j$).
Let $p$ be the position of the second particle at time $i$.
Recall that we update the position of the $j$-th particle \emph{before} the $(j-1)$-th particle.
Therefore if $p_j = p_{j-1}$, then the second particle does not move, which can be phrased as the probability the second particle is at position $p = p_j$ at time $i$ is
$
1 = \sum_{k=0}^{\infty} P_G(w_{j,i} = k).
$
Now suppose $p_j < p_{j-1}$, and if $p < p_{j-1}$, then this occurs with probability $(1 - \pi_j x_i)(\pi_j x_i)^{p-p_j}$ as there is no blocking behavior.
Lastly, if $p = p_{j-1}$, then the probability is
\[
(1 - \pi_j x_i) \sum_{k=p_{j-1}-p_j}^{\infty} (\pi_j x_i)^k = (\pi_j x_i)^{p_{j-1}-p_j}.
\]

Therefore, we want to identify the motion of the particles with semistandard tableaux like for the classical TASEP (with geometric jumping).
Yet, whenever the $j$-th particle does \emph{not} move the maximal possible distance it can, there are two terms contributing to the probability.
We split this into two separate terms that we encode into tableaux, and we do so by having the $-(\pi_j x_i)^{p_{j-1} - p_j + 1}$ term corresponds to adding an extra $i$ to a box in row $j-1$ (since we take $\beta_{j-1} = \pi_j$, recalling $j > 1$).
As such, the motion of the particles is constructed from a set-valued tableau $T$ by using the semistandard tableau $\min(T)$ built from the smallest entries in each box of $T$.
Indeed, we can only add an $i$ a box in the $j$-th row whenever $\Lambda^{(i)}_{j-1} > \Lambda^{(i)}_j$, where $\Lambda^{(i)}$ is the positions of the particles at time $i$, which is equivalent to $(\Lambda^{(i)})_{i=1}^n$ being the Gelfand--Tselin pattern corresponding to $\min(T)$.
Hence, we still have the normalization factor $\prod_{i=1}^n (1 - \pi_1 x_i)$, which completes the proof of Theorem~\ref{thm:transition_prob} in this case.

As another way to see this, consider evolution of the particles from time $t=i-1$ with positions $(\mu_1,\dots,\mu_\ell)$ to time $t=i$ with positions $(\lambda_1,\dots,\lambda_\ell)$.
The preceding discussion means that when the $j$-th particle ($j=2,\dots,\ell$) is blocked by the $(j-1)$-th particle, which corresponds to the case $\lambda_j=\mu_{j-1}$, summing up conditional probabilities give $(\pi_j x_i)^{\lambda_j-\mu_j}$.
If it is not blocked ($\lambda_j \neq \mu_{j-1}$), the conditional probability is $(\pi_j x_i)^{\lambda_j-\mu_j}(1-\pi_j x_i)$.
The unified expression for the factor coming from the move of the $j$-th particle is $(\pi_j x_i)^{\lambda_j-\mu_j}(1-(1-\delta_{\lambda_j, \mu_{j-1}}) \pi_j x_i)$,
where $\delta_{ij}$ is the Kronecker delta: $\delta_{ij}=1$ if $i \neq j$ and 0 otherwise.
As for the first particle, there is nothing to block its motion and we have the factor $(1-\pi_1 x_i)(\pi_1 x_i)^{\lambda_1-\mu_1}$.
Hence we have
\begin{align*}
\prob_{C,1}(\lambda | \mu)&=
(1-\pi_1 x_i)(\pi_1 x_i)^{\lambda_1-\mu_1} \prod_{j=2}^\ell (\pi_j x_i)^{\lambda_j-\mu_j}(1-(1-\delta_{\lambda_j, \mu_{j-1}}) \pi_j x_i) \\
&=(1-\pi_1 x_i) \bpi^{\lambda/\mu} x_i^{|\lambda/\mu|}
\prod_{j=2}^\ell (1-(1-\delta_{\lambda_j, \mu_{j-1}}) \pi_j x_i).
 \end{align*}
We also have the refinement of~\cite[Prop.~8.8]{Yel17}:
\begin{equation}
\label{eq:single_var_double_slash}
 G_{\lambda \ds \mu}(x_i;\bpi)=x_i^{|\lambda/\mu|} \prod_{j=2}^\ell (1-(1-\delta_{\lambda_j, \mu_{j-1}}) \pi_j x_i).
\end{equation}
We can see~\eqref{eq:single_var_double_slash} holds by the analogous refinement of the combinatorial description in~\cite[Thm.~4.6]{Yel19} by using $\beta_j$ if a contribution to $\beta$ occurs on the $j$-th row.
Alternatively, we can deduce~\eqref{eq:single_var_double_slash} by applying~\eqref{eq:double_to_single_slash} and the combinatorial description of $G_{\lambda/\mu}(x_i; \bpi)$ (where there is a unique \emph{semistandard} tableau).

\begin{ex}
Consider $\lambda = 631$ with $n = 5$.
A minimal tableau for particle motions is
\[
\ytableausetup{boxsize=1.8em}
\min(T) = \ytableaushort{111245,244,4}
\qquad \longleftrightarrow \qquad
\begin{tikzpicture}[scale=0.7,baseline=-1.9cm]
\foreach \i in {0,1,2,3,4,5} {
  \draw[-] (-0.2,-\i) -- (6.2,-\i);
  \foreach \x in {0,1,2,3,4,5,6}
    \draw[-] (\x,-\i-.1) -- (\x,-\i+.1);
}
\foreach \y in {0,1,...,5}
  \draw (7.3,-\y) node {$i = \y$};
\fill[darkred] (0,0+0.1) circle (0.1);
\fill[blue] (0,0.3) circle (0.1);
\fill[UQpurple] (0,0.5) circle (0.1);
\fill[dgreencolor] (0,0.7) circle (0.1);
\fill[darkred] (3,-1+0.1) circle (0.1);
\fill[blue] (0,-1+0.1) circle (0.1);
\fill[UQpurple] (0,-1+0.3) circle (0.1);
\fill[dgreencolor] (0,-1+0.5) circle (0.1);
\fill[darkred] (4,-2+0.1) circle (0.1);
\fill[blue] (1,-2+0.1) circle (0.1);
\fill[UQpurple] (0,-2+0.1) circle (0.1);
\fill[dgreencolor] (0,-2+0.3) circle (0.1);
\fill[darkred] (4,-3+0.1) circle (0.1);
\fill[blue] (1,-3+0.1) circle (0.1);
\fill[UQpurple] (0,-3+0.1) circle (0.1);
\fill[dgreencolor] (0,-3+0.3) circle (0.1);
\fill[darkred] (5,-4+0.1) circle (0.1);
\fill[blue] (3,-4+0.1) circle (0.1);
\fill[UQpurple] (1,-4+0.1) circle (0.1);
\fill[dgreencolor] (0,-4+0.1) circle (0.1);
\fill[darkred] (6,-5+0.1) circle (0.1);
\fill[blue] (3,-5+0.1) circle (0.1);
\fill[UQpurple] (1,-5+0.1) circle (0.1);
\fill[dgreencolor] (0,-5+0.1) circle (0.1);
\end{tikzpicture}
\]
The set-valued tableau $T$ with the largest degree with this minimal tableau $\min(T)$ is
\[
\ytableausetup{boxsize=1.8em}
T = \ytableaushort{11{1\mathbf{\color{darkred}2}}{2\mathbf{\color{darkred}34}}{4\mathbf{\color{darkred}5}}5,{2\mathbf{\color{darkred}3}}4{4\mathbf{\color{darkred}5}},{4\mathbf{\color{darkred}5}}}\,,
\]
Every other set-valued tableau $T'$ with $\min(T') = \min(T)$ is formed by removing some of the extra (bold) entries in $T$;
in other words, each of these bold entires can be chosen to be added independently to yield a valid set-valued tableau $T'$ with $\min(T') = \min(T)$.
Note that each of the bold entries correspond to a case when a particle that does not move its maximal possible distance.
Note that at time $i = 4$, the third particle has moved its maximum possible distance since the update order is left-to-right and it becomes blocked by the second particle.
Thus, we do not have an bold $\mathbf{\color{darkred}4}$ appearing in the second row.
We can also consider infinitely other particles, which are all blocked, and so they do not contribute to the probability.
\end{ex}

\begin{remark}
This can also be compared with the solvable vertex model from~\cite{MS13,MS14} with $\bbe = \beta$ to see this choice.
We remark that the refined $\bbe$ version can be constructed from a ``trivially'' colored lattice model where the colored paths do not cross (which is different than the one used in~\cite[Thm.~3.6]{BSW20}) with $\beta_i$ corresponding to color $i$.
\end{remark}

\subsection{Case B: Bernoulli blocking}

Essentially this is modifying Case~C in the same way as we modified Case~A to obtain Case~D.
While in this case, things are seemingly very different since $\wG_{\lambda}(\xx_n; \bal)$ is always a formal power series, we see this on the probability side by considering the power series expansion of the rate
\[
\frac{1}{1 + \rho_j x_i} = \sum_{k=0}^{\infty} (-\rho_j x_i)^k.
\]
Thus, if there are $k$ extra entries of $i$ in column $j-1$, then this corresponds to choosing $(-\rho_j x_i)^k$ in the above expansion.
Like Case~C, the movement of the particles for a multiset-valued tableau $T$ is described by the semistandard tableaux $\min(T)$ formed by taking the smallest entries in each box.
In contrast to Case~C, we can always repeat any entry $i$ in $\min(T)$ (within its box) an arbitrary number of times and the result remains a multiset-valued tableau; this reflects that we multiply \emph{every} entry by $(1 + \rho_j x_i)^{-1}$.

\begin{ex}
Let us take $\lambda = 4311$ with $n =5$.
Then an example of the correspondence between a semistandard tableau $\min(T)$ and the particle motions is
\[
\ytableausetup{boxsize=1.8em}
\min(T) = \ytableaushort{1224,245,4,5}
\qquad \longleftrightarrow \qquad
\begin{tikzpicture}[scale=0.7,baseline=-1.9cm]
\foreach \i in {0,1,2,3,4,5} {
  \draw[-] (-0.2,-\i) -- (6.2,-\i);
  \foreach \x in {0,1,2,3,4,5,6}
    \draw[-] (\x,-\i-.1) -- (\x,-\i+.1);
}
\foreach \y in {0,1,...,5}
  \draw (7.3,-\y) node {$i = \y$};
\fill[darkred] (0,0+0.1) circle (0.1);
\fill[blue] (0,0.3) circle (0.1);
\fill[UQpurple] (0,0.5) circle (0.1);
\fill[dgreencolor] (0,0.7) circle (0.1);
\fill[darkred] (1,-1+0.1) circle (0.1);
\fill[blue] (0,-1+0.1) circle (0.1);
\fill[UQpurple] (0,-1+0.3) circle (0.1);
\fill[dgreencolor] (0,-1+0.5) circle (0.1);
\fill[darkred] (2,-2+0.1) circle (0.1);
\fill[blue] (1,-2+0.1) circle (0.1);
\fill[UQpurple] (1,-2+0.3) circle (0.1);
\fill[dgreencolor] (0,-2+0.1) circle (0.1);
\fill[darkred] (2,-3+0.1) circle (0.1);
\fill[blue] (1,-3+0.1) circle (0.1);
\fill[UQpurple] (1,-3+0.3) circle (0.1);
\fill[dgreencolor] (0,-3+0.1) circle (0.1);
\fill[darkred] (3,-4+0.1) circle (0.1);
\fill[blue] (2,-4+0.1) circle (0.1);
\fill[UQpurple] (1,-4+0.1) circle (0.1);
\fill[dgreencolor] (1,-4+0.3) circle (0.1);
\fill[darkred] (4,-5+0.1) circle (0.1);
\fill[blue] (2,-5+0.1) circle (0.1);
\fill[UQpurple] (2,-5+0.3) circle (0.1);
\fill[dgreencolor] (1,-5+0.1) circle (0.1);
\end{tikzpicture}
\]
Any multiset valued tableau with $\min(T)$ will have a corresponding set-valued tableau (which does not change $\min(T)$) whose entries are a subset of each box of
\[
\ytableausetup{boxsize=1.8em}
\overline{T} = \ytableaushort{{\mathbf{\color{darkred}1}}{\mathbf{\color{darkred}2}}{\mathbf{\color{darkred}234}}{\mathbf{\color{darkred}45}},{\mathbf{\color{darkred}23}}{\mathbf{\color{darkred}45}}{\mathbf{\color{darkred}5}},{\mathbf{\color{darkred}4}},{\mathbf{\color{darkred}5}}}\,.
\]
We have written every entry in bold to indicate (and emphasize) that we can repeat \emph{every} entry as many times as we desire.
\end{ex}

\section{Multi-point distributions}
\label{sec:multipoint}

In this section, we will compute certain multi-point distributions at a single time associated for the TASEPs we consider here.
We will specifically focus on the Case~A and Case~C as the other cases can be shown by applying the $\omega$ involution.
We give determinant formulas for the multi-point distributions for the general cases.
When starting from the step initial condition, we show our formulas reduce to specializations of (dual) Grothendieck polynomials.

\subsection{Pushing}

We begin by stating a straightforward extension of~\cite[Cor.~3.14]{MS20}, which can be proven using the lattice model given therein.
However, we will sketch a proof using our free fermion presentation.
To state the claim, we need the flagged Schur function, which we denote by $s_{\lambda,\phi}(\xx)$ for a flagging $\phi$.
Let $\lambda$ be a partition and $\phi$ a flagging (a sequence of integers).
The flagged Schur function
$s_{\lambda,\phi}(\mathbf{x})$
is defined by
\[
s_{\lambda,\phi}(\mathbf{x}) = \sum_T \wt(T),
\]
where the sum is over all semistandard Young tableaux $T$ of shape $\lambda$ such that the entries in row $i$ are at most $\phi_i$.

\begin{prop}
\label{prop:schur_branching}
Let $\ell = \ell(\lambda)$.
We have
\begin{equation}
\label{eq:schur_branching}
s_{\lambda,\phi}(\xx, \bbe_{\ell}) = \sum_{\mu \subseteq \lambda} \bbe^{\lambda-\mu} \dG_{\mu}(\xx; \bbe),
\end{equation}
where $\phi$ is the flagging such that the entries in the $i$-th row is at most $\beta_i$.
\end{prop}

\begin{proof}
From~\cite[Ex.~2.6]{Iwao21}, the flagged Schur function can be written as
\[
s_{\lambda,\phi}(\xx_n, \bbe_{\ell}) = \bra{\emptyset} e^{H(\xx_n)} \ket{\lambda}_{(\bbe)},
\qquad\qquad \text{ where } \ket{\lambda}_{(\bbe)} := \prod^{\rightarrow}_{1 \leq i \leq \ell} \left( e^{H(\beta_i)} \psi_{\lambda_i-i} \right) \ket{-\ell}.
\]
We want to show that
\begin{equation}
\label{eq:schur_dg_expansion}
\ket{\lambda}_{(\bbe)} = e^{H(\beta_1)} \ket{\lambda}_{[\vec{\bbe}]} = \sum_{\mu \subseteq \lambda} \bbe^{\lambda / \mu} \ket{\mu}_{[\bbe]},
\end{equation}
where $\vec{\bbe} = (\beta_2, \beta_3, \ldots)$ denotes the shifted parameters and the last equality is the branching rule~\cite[Cor.~3.19]{MS20} (which we will provide a free fermionic proof).
The first equality is immediate from the definitions.
We will prove the last equality in~\eqref{eq:schur_dg_expansion} by using~\eqref{eq:orthonormal_basis}, which reduces the claim to showing ${}_{[\bbe]} \braket{\mu}{\lambda}_{(\bbe)} = \bbe^{\lambda - \mu}$ for $\mu \subseteq \lambda$ and $0$ otherwise.
Indeed, the claim follows from a similar proof to~\cite[Thm.~3.10]{IMS22} except we now have the $i$-th diagonal entry equal to $\beta_i^{\lambda_i - \mu_i}$ due to the shift.
Hence, Equation~\eqref{eq:schur_dg_expansion} yields
\begin{equation}
\label{eq:beta_schur_g}
s_{\lambda,\phi}(\xx_n, \bbe_{\ell}) = \bra{\emptyset} e^{H(\xx_n)} \ket{\lambda}_{(\bbe)}
= \sum_{\mu \subseteq \lambda} \bbe^{\lambda - \mu}  \bra{\emptyset} e^{H(\xx_n)} \ket{\mu}_{[\bbe]}
= \sum_{\mu \subseteq \lambda} \bbe^{\lambda-\mu} \dG_{\mu}(\xx_n; \bbe),
\end{equation}
where we used~\eqref{eq:free_fermion_grothendiecks} with the fact ${}_{[\bbe]} \bra{\emptyset} = \bra{\emptyset}$ for the last equality.
\end{proof}

\begin{remark}
Our proofs of Proposition~\ref{prop:schur_branching} can be trivially generalized to give a formula for the expansion of $s_{\lambda,\phi}(\xx, \bbe_{\ell})$ into $\dG_{\mu}(\xx_n; \bbe)$ for an arbitrary flag $\phi$ with the coefficients both as a combinatorial formula (from~\cite{MS20}) or as a determinant (from Wick's theorem or the LGV lemma).
\end{remark}

Next we compute formulas for the multi-point distribution.
Using the fact that we must have $G(i, n) \geq G(i+1, n)$, for $\ell$ particles, the $k$-point distribution is equivalent to the $\ell$-point distribution.
In particular, we consider $1 = i_1 \geq i_2 \geq \cdots \geq i_k$, and we have
\[
\prob(G(i_k,n) \leq \lambda_{i_k}, \dotsc, G(i_1,n) \leq \lambda_{i_1}) = \prob(G(\ell,n) \leq \lambda_{\ell}, \dotsc, G(1,n) \leq \lambda_1),
\]
where $\lambda_j = \lambda_{i_k}$ with $k$ maximal such such $i_k \leq j$.
The ordering on $\{i_k\}$ does not lose any generality, but we do require knowledge about the maximum distance the first particle can move.
As such, we will use the notation
\[
\prob_{\leq,n}(\lambda | \nu) := \prob(G(\ell,n) \leq \lambda_{\ell}, \dotsc, G(1,n) \leq \lambda_1 | \bG(0) = \nu).
\]

\begin{thm}
\label{thm:mp_dual_g}
For Case~A, the multi-point distribution with $\ell$ particles is given by
\[
\prob_{\leq,n}(\lambda | \nu) = \bpi^{\lambda/\nu} \prod_{i=1}^{\ell} \prod_{j=1}^n (1 - \pi_i x_j) \cdot \det \bigl[ h_{\lambda_i - \nu_j + j - i}(\xx \sqcup \bpi_i^{-1} / \bpi_{j-1}^{-1}) \bigr]_{i,j=1}^{\ell}.
\]
\end{thm}

\begin{proof}
Using~\eqref{eq:schur_dg_expansion}, we compute
\begin{align*}
\prob_{\leq,n}(\lambda|\nu) & = \prod_{i=1}^{\ell} \prod_{j=1}^n (1 - \pi_i x_j) \sum_{\nu \subseteq \mu \subseteq \lambda} \bpi^{\mu / \nu} \dG_{\mu/\nu}(\xx_n; \bpi^{-1})
\\ & = \bpi^{\lambda/\nu} \prod_{i=1}^{\ell} \prod_{j=1}^n (1 - \pi_i x_j) \cdot {}_{[\bpi^{-1}]} \bra{\nu} e^{H(\xx_n)} \ket{\lambda}_{(\bpi^{-1})}.
\end{align*}
Next we apply Wick's theorem to compute ${}_{[\bpi^{-1}]} \bra{\nu} e^{H(\xx_n)} \ket{\lambda}_{(\bpi^{-1})} = \det[ \bra{-\ell} P_j Q_i \ket{-\ell} ]_{i,j=1}^{\ell}$, where
\begin{align*}
P_j & := e^{H(\bpi_{j-1}^{-1})} \psi^*_{\nu_j-j}e^{-H(\bpi_{j-1}^{-1})} = \sum_{m=0}^{\infty} h_m(\emptyset / \bpi_{j-1}^{-1}) \psi^*_{\nu_j-j+m},
\\
Q_i & := e^{H(\xx \sqcup \bpi_i^{-1})} \psi_{\lambda_i-i} e^{-H(\xx \sqcup \bpi_i^{-1})} = \sum_{k=0}^{\infty} h_k(\xx \sqcup \bpi_i^{-1}) \psi_{\lambda_i-i-k}.
\end{align*}
Therefore, the entries in the determinant are (\textit{cf.}~\cite[Prop.~2.4]{Iwao21})
\begin{align*}
\bra{-\ell} P_j Q_i \ket{-\ell}
& = \sum_{m=0}^{\infty} \sum_{k=0}^{\infty} h_m(\emptyset / \bpi_{j-1}^{-1}) h_k(\xx \sqcup \bpi_i^{-1}) \cdot \bra{-\ell} \psi^*_{\nu_j-j+m} \psi_{\lambda_i-i-k} \ket{-\ell}
\\ & = \sum_{k=0}^{\infty} h_{\lambda_i - \nu_j + j - i - k}(\emptyset / \bpi_{j-1}^{-1}) h_k(\xx \sqcup \bpi_i^{-1})
\\ & = h_{\lambda_i - \nu_j + j - i}(\xx \sqcup \bpi_i^{-1} / \bpi_{j-1}^{-1}),
\end{align*}
yielding the claim.
\end{proof}

\begin{proof}[Alternative proof]
We give another proof of Theorem~\ref{thm:mp_dual_g} using the skew Cauchy formula (Theorem~\ref{thm:skew_cauchy}).
If we take $\mu = \emptyset$ and the specializations $\bal = 0$ and $\xx = \bpi^{-1}$, we obtain
\[
\sum_{\lambda \supseteq \nu} \bpi^{\lambda} \dG_{\lambda/\nu}(\yy; \bpi) = \prod_i \frac{1}{1 - \pi_1 y_i} \dG_{\nu}(\xx; \bbe) \bpi^{\nu}
\]
since $\G_{\lambda}(\bpi^{-1}; \bpi) = \bpi^{\lambda}$ by directly examining the bi-alternate formula~\cite{HJKSS24,HJKSS25} (see also~\cite{FNS23} for a description of other ways to verify this identity).
The claim follows from the Jacobi--Trudi formula (see, \textit{e.g.},~\cite[Thm.~6.1]{HJKSS24} or~\cite[Thm.~4.1]{IMS22}).
\end{proof}

Unless otherwise stated, we will henceforth use the flagging $\phi$ from Proposition~\ref{prop:schur_branching} for our flagged Schur functions.
As a special case of Theorem~\ref{thm:mp_dual_g} using Proposition~\ref{prop:schur_branching}, we obtain
\begin{align*}
\prob_{\leq,n}(\lambda | \emptyset) & = \prod_{i=1}^{\ell} \prod_{j=1}^n (1 - \pi_i x_j) \sum_{\mu \subseteq \lambda} \bpi^{\mu} \dG_{\mu}(\xx_n; \bpi^{-1})
= \bpi^{\lambda} \prod_{i=1}^{\ell} \prod_{j=1}^n (1 - \pi_i x_j) s_{\lambda,\phi}(\xx_n, \bpi_{\ell}^{-1}).
\end{align*}
Furthermore, this recovers~\cite[Thm.~4.26]{MS20} by taking $\lambda$ to be an $\ell \times m$ rectangle, which allows us to forget about the flagging $\phi$.

\begin{remark}
\label{rem:JR_comparison}
We want to compare Theorem~\ref{thm:mp_dual_g} to~\cite[Thm.~2]{JR20}.
We begin by following~\cite[Thm.~4.18]{IMS22} with the substitution $w \mapsto z^{-1}$ to write the entries of the determinant as the integral
\[
h_{\lambda_i-\nu_j-i+j}(\xx_n \sqcup \bpi_i^{-1} / \bpi_{j-1}^{-1}) = \frac{1}{2\pi\ii} \oint_{\gamma_r} \frac{\prod_{k=1}^{j-1} (1 - \pi_k^{-1} z^{-1})}{\prod_{k=1}^i (1 - \pi_k^{-1} z^{-1})} \frac{z^{\lambda_i - \mu_j - i + j}}{\prod_{k=1}^n (1 - x_k z^{-1})} \, \frac{dz}{z},
\]
where $\gamma$ is a counterclockwise oriented circle of radius $r > \abs{x_1}, \abs{\pi_1^{-1}}, \abs{x_2}, \abs{\pi_2^{-1}}, \ldots$ and $\ii = \sqrt{-1}$.
Hence, we can express $\prob(G(\ell,n) \leq \lambda_{\ell}, \dotsc, G(1,n) \leq \lambda_1 | \bG(0) = \nu) = \det [F_{ij}]_{i,j=1}^{\ell}$, where
\begin{equation}
\label{eq:integral_mp_dual_g}
F_{ij} = \frac{1}{2\pi\ii} \oint_{\gamma_r} \frac{(\pi_i z)^{\lambda_i}}{(\pi_j z)^{\mu_j}} \frac{\prod_{k=1}^{j-1} (z - \pi_k^{-1})}{\prod_{k=1}^i (z - \pi_k^{-1})} \prod_{k=1}^n \frac{1 - x_k \pi_i}{1 - x_k z^{-1}} \, dz.
\end{equation}
This is the transpose of the expression in~\cite[Thm.~2]{JR20} after noting that their requirement is $G(i,n) < \lambda_i + 1$ (this is an equivalent condition since $\lambda_i \in \ZZ_{\geq 0}$) and they use weakly \emph{increasing} sequences (\textit{i.e.}, the order of the particles is reversed).
In particular, we must multiply the integrand by $\prod_{k=1}^{\ell} \frac{z - \pi_k^{-1}}{z - \pi_k^{-1}}$ and reindex the particles and matrix by $j \mapsto \ell + 1 - j$ (so $\pi_j \mapsto \pi_{\ell+1-j}$ as well).
A similar transformation shows that~\cite[Thm.~1]{JR20} is equivalent to the integral formula from~\cite[Thm.~4.18]{IMS22} with Theorem~\ref{thm:transition_prob}.
\end{remark}

\begin{ex}
Let us consider $\lambda = (2,1)$ and $\nu = \emptyset$ for Case~A.
Thus we take $\ell = 2$, and we assume that $n \geq 2$.
In this computation, we will essentially ignore the normalization factor $C = \prod_{i=1}^{\ell} \prod_{j=1}^n (1 - \pi_i x_j)$, as that factor clearly cancels.
By the definition and Theorem~\ref{thm:transition_prob},
\begin{align*}
C^{-1} \cdot \prob_{\leq, n}(\lambda | \emptyset)
& = C^{-1} \cdot \prob(G(2,n) \leq \lambda_2, G(1,n) \leq \lambda_1)
\\ & = C^{-1} \cdot \bigl( \prob(\bG(n) = (2,1)) + \prob(\bG(n) = (2)) + \prob(\bG = (1,1))
\\ & \hspace{40pt} + \prob(\bG(n) = (1)) + \prob(\bG(n) = \emptyset) \bigr)
\\ & = \pi_1^2 \pi_2 \dG_{21}(\xx_n, \bpi^{-1}) + \pi_1^2 \dG_{2}(\xx_n, \bpi^{-1}) + \pi_1 \pi_2 \dG_{11}(\xx_n, \bpi^{-1})
\\ & \hspace{20pt}  + \pi_1 \dG_{1}(\xx_n, \bpi^{-1}) + \dG_{\emptyset}(\xx_n, \bpi^{-1})
\\ & = \pi_1^2 \pi_2 (s_{21} + \pi_1^{-1} s_2) + \pi_1^2 s_2 + \pi_1 \pi_2 (s_{11} + \pi_1^{-1} s_1) + \pi_1 s_1 + 1
\\ & = \pi_1^2 \pi_2 s_{21} + (\pi_1 \pi_2 + \pi_1^2) s_2 + \pi_1 \pi_2 s_{11} + (\pi_1 + \pi_2) s_1 + 1,
\end{align*}
where we have written $s_{\lambda} = s_{\lambda}(\xx_n)$ for brevity.
Next, by the branching rule for flagged Schur functions, we see that
\begin{align*}
\pi_1^2 \pi_2 s_{21,\phi}(\xx_n, \bpi_2^{-1}) & = \pi_1^2 \pi_2 \bigl( s_{21}(\xx_n) s_{21/21,\phi}(\bpi_2^{-1}) + s_{2}(\xx_n) s_{21/2,\phi}(\bpi_2^{-1}) + s_{11}(\xx_n) s_{21/11,\phi}(\bpi_2^{-1})
\\ & \hspace{40pt} + s_{1}(\xx_n) s_{21/1,\phi}(\bpi_2^{-1}) + s_{\emptyset}(\xx_n) s_{21/\emptyset,\phi}(\bpi_2^{-1}) \bigr)
\\ & = \pi_1^2 \pi_2 \bigl( s_{21} + (\pi_1^{-1} + \pi_2^{-1}) s_{2} + \pi_1^{-1}s_{11} + (\pi_1^{-2} + \pi_1^{-1} \pi_2^{-1}) s_{1} + \pi_1^{-2} \pi_2^{-1} \bigr)
\\ & = \pi_1^2 \pi_2 s_{21} + (\pi_1 \pi_2 + \pi_1^2) s_2 + \pi_1 \pi_2 s_{11} + (\pi_1 + \pi_2) s_1 + 1.
\end{align*}
Finally, let us compute the determinant from Theorem~\ref{thm:mp_dual_g}:
\begin{align*}
\pi_1^2 \pi_2 \det \begin{bmatrix}
h_2(\xx, \pi_1^{-1}) & h_3(\xx) \\
h_0(\xx, \pi_1^{-1}, \pi_2^{-1}) & h_1(\xx, \pi_2^{-1})
\end{bmatrix}
& = 
\pi_1^2 \pi_2 ( h_2 + \pi_1^{-1} h_1 + \pi_1^{-2}) \cdot (h_1 + \pi_2^{-1}) - \pi_1^2 \pi_2 h_3 \cdot 1
\\ & = \pi_1^2 \pi_2 (h_{21} - h_3) + \pi_1 \pi_2 h_{11} + (\pi_2 + \pi_1) h_1 + \pi_1^2 h_2 + 1
\\ & = \pi_1^2 \pi_2 s_{21} + (\pi_1 \pi_2 + \pi_1^2) s_2 + \pi_1 \pi_2 s_{11} + (\pi_2 + \pi_1) s_1 + 1.
\end{align*}
\end{ex}

\begin{thm}
\label{thm:mp_dual_weak}
For Case~D, the multi-point distribution with $\ell$ particles is given by
\[
\prob_{\leq,n}(\lambda | \nu) = \brho^{\lambda/\nu} \prod_{i=1}^{\ell} \prod_{j=1}^n (1 - \rho_i x_j)^{-1} \cdot \det \bigl[ e_{\lambda_i - \nu_j + j - i}(\xx \sqcup -\brho_{j-1}^{-1} / {-\brho_i^{-1}}) \bigr]_{i,j=1}^{\ell}.
\]
\end{thm}

\begin{proof}
Analogous to the proof of Theorem~\ref{thm:mp_dual_g} except we are computing ${}_{[\brho^{-1}]} \bra{\nu} e^{J(\xx_n)} \ket{\lambda}_{(\brho^{-1})}$ and use~\eqref{eq:dual_free_fermion_grothendiecks} to obtain $\dwG_{\lambda/\nu}(\xx; \brho^{-1})$.
\end{proof}

\begin{ex}
For this example, we consider the Case~D TASEP.
Consider $\lambda = (2, 1, 1)$ and $\nu = (1)$, and so we take $\ell = 3$ and $n \geq 3$.
Note that $\lambda' = (3,1)$ and $\nu' = (1)$,
We will (again) ignore the normalization factor $C = \prod_{i=1}^{\ell} \prod_{j=1}^n (1 - \rho_i x_j)^{-1}$ as it clearly is present in all formulas.
By the definition and Theorem~\ref{thm:transition_prob}, we have
\begin{align*}
C^{-1} \cdot \prob_{\leq, n}(\lambda | \nu) & = \rho_1 \rho_2 \rho_3 \dwG_{31/1}(\xx; \brho^{-1}) + \rho_1 \rho_2 \dwG_{21/1}(\xx; \brho^{-1}) + \rho_2 \rho_3 \dwG_{3/1}(\xx; \brho^{-1})
\\ & \hspace{20pt} + \rho_2 \dwG_{2/1}(\xx; \brho^{-1}) + \rho_1 \dwG_{11/1}(\xx; \brho^{-1}) + \dwG_{1/1}(\xx; \brho^{-1})
\\ & = \rho_1 \rho_2 \rho_3 (s_3 + s_{21} + \rho_2^{-1} s_2 + \rho_2^{-1} s_{11}) + \rho_1 \rho_2 (s_2 + s_{11})
\\ & \hspace{20pt} + \rho_2 \rho_3 (s_2 + \rho_2^{-1} s_1) + \rho_2 s_1 + \rho_1 s_1 + 1
\\ & = \rho_1 \rho_2 \rho_3 s_3 + \rho_1 \rho_2 \rho_3 s_{21} + (\rho_1 \rho_2 + \rho_1 \rho_3 + \rho_2 \rho_3) s_2
\\ & \hspace{20pt} + (\rho_1 \rho_2 + \rho_1 \rho_3) s_{11} + (\rho_1 + \rho_2 + \rho_3) s_1 + 1.
\end{align*}
The determinant formula from Theorem~\ref{thm:mp_dual_weak} yields
\begin{align*}
& \rho_1 \rho_2 \rho_3 \det \begin{bmatrix}
e_1(\xx / (-\rho_1^{-1})) & e_3(\xx) & e_4(\xx, -\rho_2^{-1}) \\
e_{-1}(\xx / {-\brho_2^{-1}}) & e_1(\xx / (-\rho_2^{-1})) & e_2(\xx) \\
e_{-2}(\xx / {-\brho_3^{-1}}) & e_0(\xx / (-\rho_2^{-1}, -\rho_3^{-1})) & e_1(\xx / (-\rho_3^{-1}) \\
\end{bmatrix}
\\ & \qquad
= \rho_1 \rho_2 \rho_3 \det \begin{bmatrix}
e_1 + \rho_1^{-1} & e_3 & e_4 - \rho_2^{-1} e_3 \\
0 & e_1 + \rho_2^{-1} & e_2 \\
0 & 1 & e_1 + \rho_3^{-1} \\
\end{bmatrix}
\\ & \qquad
= \rho_1 \rho_2 \rho_3 (e_1 + \rho_1^{-1}) (e_1 + \rho_2^{-1}) (e_1 + \rho_3^{-1}) - \rho_1 \rho_2 \rho_3 e_2 (e_1 + \rho_1^{-1}) \cdot 1
\\ & \qquad
= \rho_1 \rho_2 \rho_3 e_{111} + (\rho_1 \rho_2 + \rho_1 \rho_3 + \rho_2 \rho_3) e_{11} + (\rho_1 + \rho_2 + \rho_3) e_1 + 1 - \rho_1 \rho_2 \rho_3  e_{21} - \rho_2 \rho_3 e_2
\\ & \qquad
= \rho_1 \rho_2 \rho_3 (s_3 + 2s_{21} + s_{111}) + (\rho_1 \rho_2 + \rho_1 \rho_3 + \rho_2 \rho_3) (s_2 + s_{11}) + (\rho_1 + \rho_2 + \rho_3) s_1 + 1
\\ & \qquad \hspace{20pt} - \rho_1 \rho_2 \rho_3  (s_{21} + s_{111}) - \rho_2 \rho_3 s_{11}
\\ & \qquad
= \rho_1 \rho_2 \rho_3 s_3 + \rho_1 \rho_2 \rho_3 s_{21} + (\rho_1 \rho_2 + \rho_1 \rho_3 + \rho_2 \rho_3) s_2 + (\rho_1 \rho_2 + \rho_1 \rho_3) s_{11} + (\rho_1 + \rho_2 + \rho_3) s_1 + 1.
\end{align*}
\end{ex}

Analogously to Equation~\eqref{eq:integral_mp_dual_g}, in Case~D we have $\prob_{\leq,n}(\lambda | \nu) = \det [F_{ij}]_{i,j=1}^{\ell}$ with
\begin{equation}
F_{ij} = \frac{1}{2\pi\ii} \oint_{\gamma_r} \frac{(\rho_i z)^{\lambda_i}}{(\rho_j z)^{\mu_j}} \frac{\prod_{k=1}^{j-1} (z - \rho_k^{-1})}{\prod_{k=1}^i (z - \rho_k^{-1})} \prod_{k=1}^n \frac{1 - x_k z^{-1}}{1 - x_k \rho_i} \, dz.
\end{equation}

\subsection{Blocking}

Next, let us consider Case~C, and recall that $\beta_i = \pi_{i+1}$.
We begin with some preparatory formulas.

Note that $\dG_{\lambda}(\pi_1; \bbe) = \bpi^{\lambda}$ from the combinatorial description.
Taking the skew Cauchy formula (Theorem~\ref{thm:skew_cauchy}) with the specializations $\bal = 0$ and $\yy = \bpi_1$, we obtain
\begin{equation}
\label{eq:specialized_skew_cauchy}
\sum_{\lambda \supseteq \nu, \mu} \G_{\lambda\ds\mu}(\xx; \bbe) \bpi^{\lambda/\nu} = \prod_i \frac{1}{1 - \pi_1 x_i} \sum_{\eta \subseteq \nu \cap \mu} \G_{\nu\ds\eta}(\xx; \bbe) \bpi^{\mu/\eta}.
\end{equation}

We will use the notation
\[
\prob_{\geq,n}(\nu | \mu) := \prob(G(\ell,n) \geq \nu_{\ell}, \dotsc, G(1,n) \geq \nu_1 | \bG(0) = \mu).
\]
Using Equation~\eqref{eq:specialized_skew_cauchy}, we obtain an expression for the multi-point distribution for Case~C as
\begin{subequations}
\begin{align}
\prob_{\geq,n}(\nu | \mu) 
 & =
\prod_{j=1}^n (1 - \pi_1 x_j) \sum _{\lambda \supseteq \nu,\mu} \bpi^{\lambda/\mu} \G_{\lambda \ds \mu}(\xx_n; \bbe) \label{eq:mp_G_larger}
\\ & = \sum _{\eta \subseteq \nu \cap \mu} \bpi^{\nu/\eta} \G_{\nu \ds \eta}(\xx_n; \bbe). \label{eq:mp_G_smaller}
\end{align}
\end{subequations}
Note that if $\nu = \mu$, then we obtain a Cauchy--Littlewood type identity with Grothendieck polynomials from the first equality~\eqref{eq:mp_G_larger} since the total probability is $1$.
That is, we have
\begin{equation}
\label{eq:specialized_skew_pieri}
\sum_{\lambda} \G_{\lambda \ds \mu}(\xx_n; \bbe) \bpi^{\lambda} = \prod_{i=1}^n \frac{1}{1 - \pi_1 x_i} \bpi^{\mu},
\end{equation}
which is also the skew Pieri rule~\cite[Eq.~(4.7a)]{IMS22} (which refines~\cite[Thm.~7.10]{Yel19}) with the skew Pieri $\nu = \emptyset$ and the same specializations $\bal = 0$ and $\yy = \bpi_1$ we used for the skew Cauchy formula.
Another special case is when $\mu = \emptyset$, where we obtain a single $\bpi^{\nu} \G_{\nu}(\xx_n; \bbe)$ in~\eqref{eq:mp_G_smaller}.
This can be expressed as a determinant by the Jacobi--Trudi formula~\cite[Thm.~4.1]{IMS22}.
Moreover, this is just the specialization $\bal = 0$ and $\yy = \bpi_1$ of the skew Pieri formula~\cite[Eq.~(4.7b)]{IMS22}.

Noting that we have used the skew Cauchy identity in computing~\eqref{eq:mp_G_smaller}, we want to evaluate
\begin{equation}
\label{eq:skew_cauchy_free_fermion}
{}^{[\bbe]} \bra{\mu} e^{H^*(\pi_1)} e^{H(\xx_n)} \ket{\nu}^{[\bbe]} 
= \prod_{i=1}^{\ell} \prod_{j=1}^n \frac{1}{1-\beta_i x_j} \cdot {}^{[\bbe]} \bra{\mu} e^{H^*(\pi_1 / \bbe_{\ell})} e^{H(\xx_n)} e^{H^*(\bbe_{\ell})} \ket{\nu}^{[\bbe]}.
\end{equation}
By applying Wick's theorem similar to the proof of Theorem~\ref{thm:mp_dual_g} (\textit{cf.}\ the proof of~\cite[Thm.~4.1]{IMS22}), we obtain the following formula for the general case.

\begin{thm}
\label{thm:mp_case_C}
For Case~C, the multi-point distribution with $\ell$ particles is given by
\[
\prob_{\geq,n}(\nu|\mu) = 
\prod_{j=2}^{\ell} \prod_{i=1}^n (1-\pi_j x_i)^{-1} \det \big[ h_{\nu_i - \mu_j - i + j}\bigl(\xx \ds (\bpi_i / \bbe_j \bigr) \bigr]_{i,j=1}^{\ell}.
\]
\end{thm}


We remark that in the left hand side of~\eqref{eq:skew_cauchy_free_fermion} is using the vector obtained by applying the $*$ anti-involution to~\eqref{eq:schur_dg_expansion}:
\[
{}^{(\bpi)} \bra{\nu} := \sum_{\lambda \subseteq \nu} \bpi^{\nu/\lambda} \cdot {}^{[\bpi]} \bra{\lambda}.
\]
However our computation for Theorem~\ref{thm:mp_case_C} is not simply the $*$ version of Proposition~\ref{prop:schur_branching} as we need to take the pairing with $\ket{\mu}^{[\bbe]}$, not $\ket{\mu}^{[\bpi]}$.

We also provide another determinant formula for the multipoint distribution by using the standard probability theoretic computation to sum over determinants with matrix elements given by integrals.

\begin{thm}
\label{thm:mp_case_C_alt}
For Case~C, the multi-point distribution with $\ell$ particles is given by
\begin{gather*}
\mathrm{P}_{\ge,n}(\nu|\mu) = \prod_{i=1}^\ell \prod_{j=1}^n (1-\pi_i x_j) \bpi^{\nu/\mu} \det [ C_{ij} ]_{i,j=1}^\ell,
\\
\text{where } C_{ij} = \begin{cases}
\displaystyle
\oint_{\widetilde{\gamma}_r}
\frac{1 }{ (1-\beta_1 w^{-1}) \prod_{k=1}^{j-1} (1-\beta_k w^{-1}) \prod_{m=1}^n (1-x_m w) w^{\nu_1-\mu_j-1+j}   }
 \frac{dw}{2 \pi {\bf i} w} & \text{if } i = 1, \\[15pt]
\displaystyle
\oint_{\widetilde{\gamma}_r}
\frac{\prod_{k=1}^{i-2} (1-\beta_k w^{-1}) }{ \prod_{k=1}^{j-1} (1-\beta_k w^{-1}) \prod_{m=1}^n (1-x_m w) w^{\nu_i-\mu_j-i+j}   }
 \frac{dw}{2 \pi {\bf i} w} & \text{if } i \ge 2,
\end{cases}
\end{gather*}
with the contour $\widetilde{\gamma}_r$ being a circle centered at the origin with radius $r$ satisfying $0 < r < \abs{x_m^{-1}}$ for $m = 1, \dotsc, n$ and $r > \abs{\beta_1}, \abs{\beta_2}, \ldots$.
\end{thm}

\begin{proof}
Using Theorem~\ref{thm:transition_prob} and the integral formula~\cite[Thm.~4.19]{IMS22}, we write the multi-point distribution as
\begin{align*}
\mathrm{P}_{\ge,n}(\nu|\mu) & = \prod_{i=1}^\ell \prod_{j=1}^n (1-\pi_i x_j) \sum_{\lambda_1=\nu_1}^\infty \sum_{\lambda_2=\nu_2}^{\lambda_1} \cdots \sum_{\lambda_\ell=\nu_\ell}^{\lambda_{\ell-1}} \bpi^{-\mu} \\
& \hspace{20pt} \times \det \Bigg[ 
\oint_{\widetilde{\gamma}_r}
\frac{\pi_i^{\lambda_i}  \prod_{k=1}^{i-1} (1-\beta_k w^{-1}) }{ \prod_{k=1}^{j-1} (1-\beta_k w^{-1}) \prod_{m=1}^n (1-x_m w) w^{\lambda_i-\mu_j-i+j}   }
 \frac{dw}{2 \pi {\bf i} w} \Bigg]_{i,j=1}^\ell.
\end{align*}
Inserting the sum $\displaystyle \sum_{\lambda_\ell=\nu_\ell}^{\lambda_{\ell-1}}$ into the $\ell$-th row,
the matrix element in the $j$-th column becomes
\begin{align}
\sum_{\lambda_\ell=\nu_\ell}^{\lambda_{\ell-1}}
\oint_{\widetilde{\gamma}_r} &
\frac{\pi_\ell^{\lambda_\ell}  \prod_{k=1}^{\ell-1} (1-\beta_k w^{-1}) }{ \prod_{k=1}^{j-1} (1-\beta_k w^{-1}) \prod_{m=1}^n (1-x_m w) w^{\lambda_\ell-\mu_j-\ell+j}   }
 \frac{dw}{2 \pi {\bf i} w} \nonumber
\\ = & \oint_{\widetilde{\gamma}_r}
\frac{\pi_\ell^{\nu_\ell}  \prod_{k=1}^{\ell-2} (1-\beta_k w^{-1}) }{ \prod_{k=1}^{j-1} (1-\beta_k w^{-1}) \prod_{m=1}^n (1-x_m w) w^{\nu_\ell-\mu_j-\ell+j}   }
 \frac{dw}{2 \pi {\bf i} w} \nonumber
\\ & \hspace{20pt} - \oint_{\widetilde{\gamma}_r}
\frac{\pi_\ell^{\lambda_{\ell-1}+1}  \prod_{k=1}^{\ell-2} (1-\beta_k w^{-1}) }{ \prod_{k=1}^{j-1} (1-\beta_k w^{-1}) \prod_{m=1}^n (1-x_m w) 
w^{\lambda_{\ell-1}-\mu_j-(\ell-1)+j}   }
 \frac{dw}{2 \pi {\bf i} w}. \label{onesum}
\end{align}
The second term in \eqref{onesum} can be eliminated using the $(\ell-1)$-th row, hence the matrix elements in the $\ell$-th row after performing the first sum can be written as
\[
\oint_{\widetilde{\gamma}_r}
\frac{\pi_\ell^{\nu_\ell}  \prod_{k=1}^{\ell-2} (1-\beta_k w^{-1}) }{ \prod_{k=1}^{j-1} (1-\beta_k w^{-1}) \prod_{m=1}^n (1-x_m w) w^{\nu_\ell-\mu_j-\ell+j}   }
 \frac{dw}{2 \pi {\bf i} w}.
\]
Iterating this process, we obtain our claim.
\end{proof}

For Case~B, we again apply the $\omega$ involution, which replaces $e^{H(\xx_n)}$ with $e^{J(\xx_n)}$, but otherwise the proof is similar (compare the proofs of Theorem~\ref{thm:mp_dual_g} and Theorem~\ref{thm:mp_dual_weak}).
Therefore, we obtain the following.

\begin{thm}
For Case~B, the multi-point distribution with $\ell$ particles is given by
\[
\prob_{\geq,n}(\nu|\mu) = 
\prod_{i=2}^{\ell} \prod_{j=1}^n (1-\rho_i x_j) \det \big[ e_{\nu_i - \mu_j - i + j}(\xx_n \ds ({-\brho_j}/{-\brho_i})) \bigr]_{i,j=1}^{\ell}.
\]
\end{thm}

In both of these cases, we can obtain contour integral formulas for the entries in matrices of the multi-point distribution determinants by following~\cite[Thm.~4.19]{IMS22} (see also Remark~\ref{rem:JR_comparison}).

\section{Continuous time limit}
\label{sec:continuous}

In this section, we will examine the continuous time limit of these processes.
In order to do this, we will take the geometric jumping with $\xx_{\lfloor t/p \rfloor} = p$ as $p \to 0$, which takes the geometric distribution with rate $\pi_j x_i$ to an exponential distribution with rate $\pi_j$.
As an example, compare Figure~\ref{fig:discrete_samples} with Figure~\ref{fig:continuous_samples}.
We will be able to see from the other integral formulas in~\cite[Sec.~4.8]{IMS22}, that using the geometric jumping will result in the same limit with Bernoulli rates using Theorem~\ref{thm:transition_prob} with noting the positions of the particles (in the bosonic form) are given by the conjugate shapes.
Therefore, we only consider the geometric jumping cases; that is, we only consider Case~A and Case~C.

\begin{figure}
\[
\includegraphics[width=.45\textwidth]{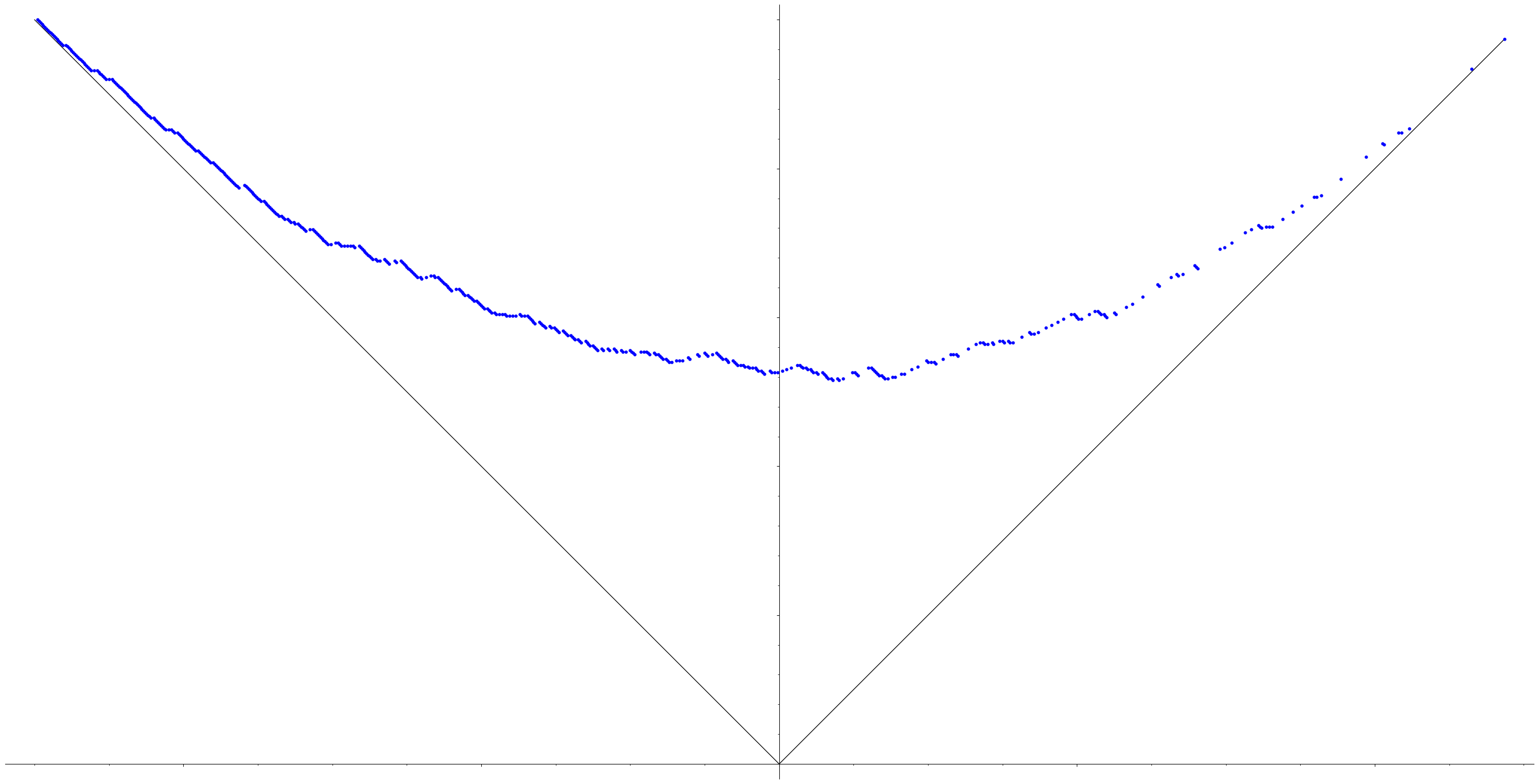}
\qquad
\includegraphics[width=.45\textwidth]{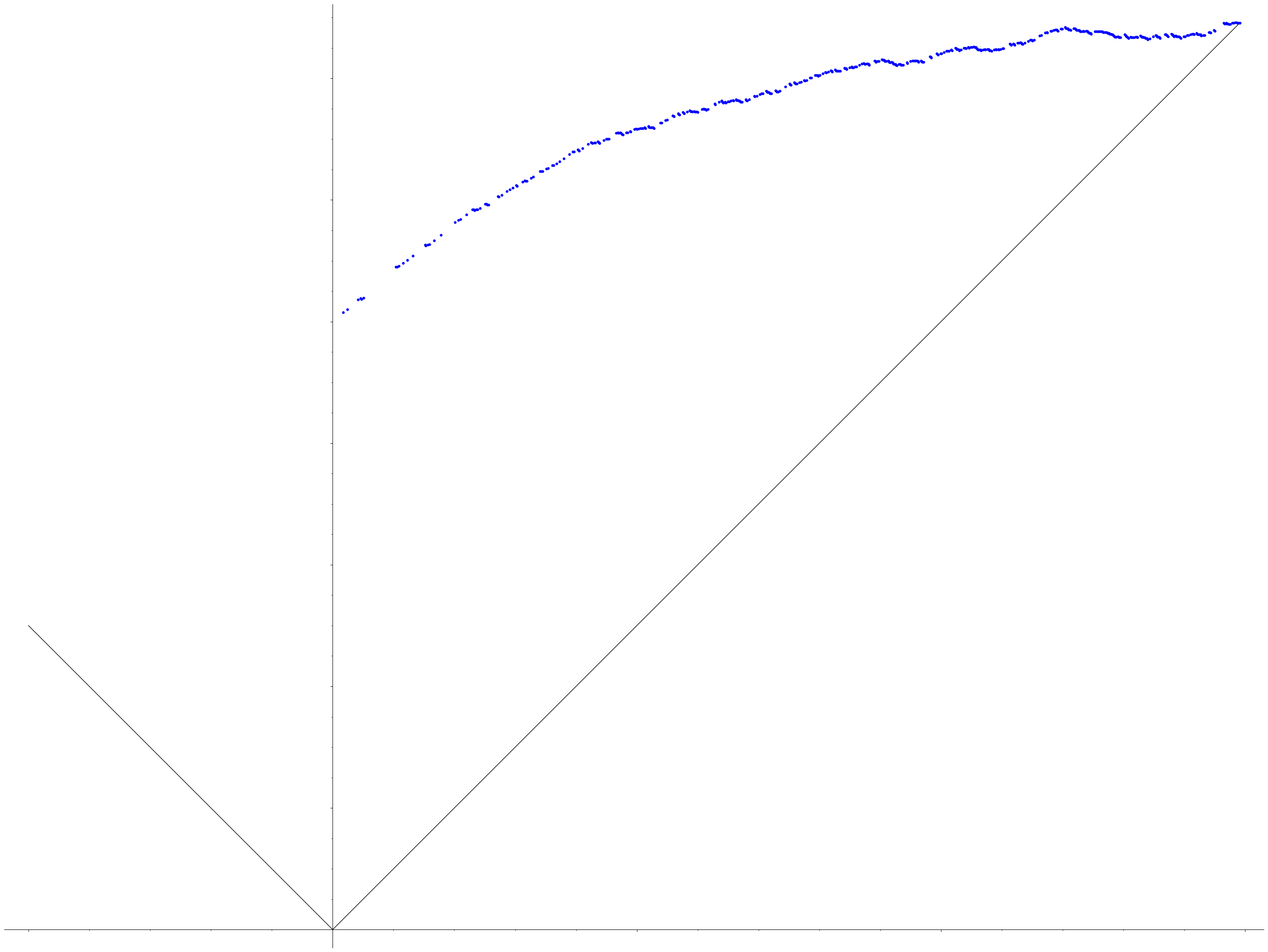}
\]
\caption{Samples of the continuous time limit TASEP with $\ell = 500$ particles at time $t = 500$ with rate $\bpi = 1$ under the blocking (left) and pushing (right) behavior.}
\label{fig:continuous_samples}
\end{figure}

\begin{figure}
\[
\includegraphics[width=.45\textwidth]{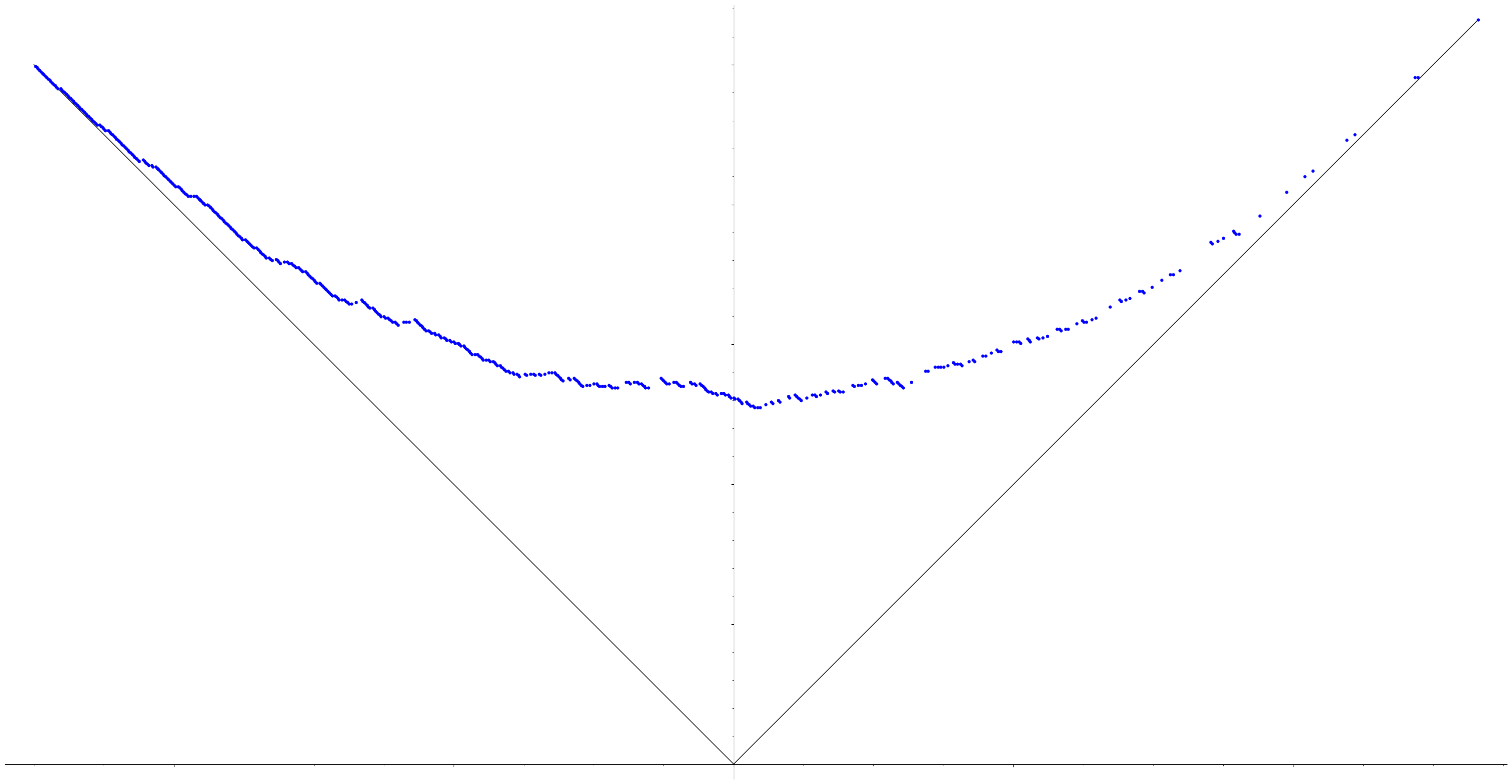}
\qquad
\includegraphics[width=.45\textwidth]{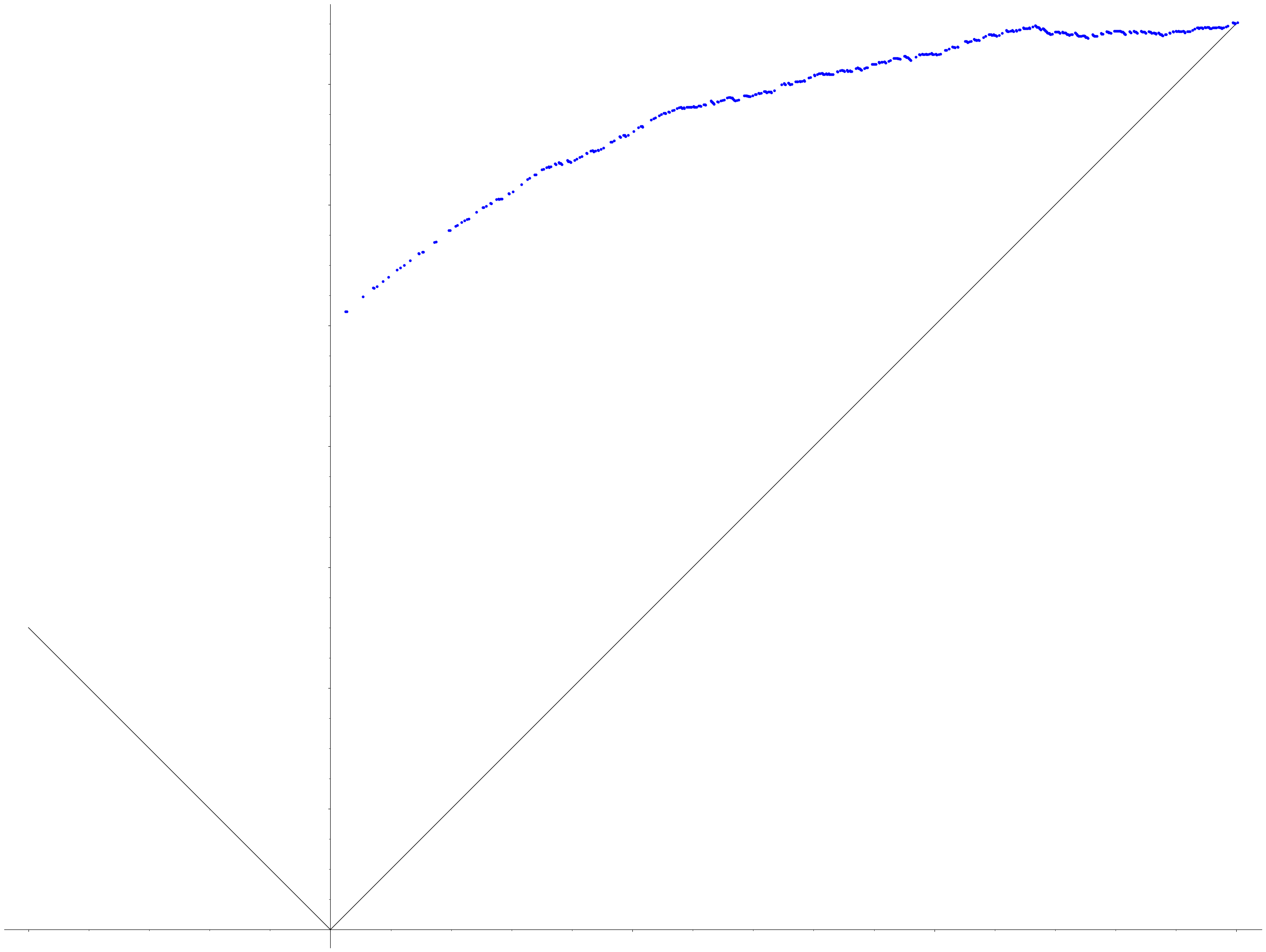}
\]
\caption{Samples of TASEP with $\ell = 500$ particles with $\bpi = 1$ and $\xx = p = 0.01$ with $n = \lfloor 500 / p \rfloor$ under the blocking behavior (left) and pushing behavior (right).}
\label{fig:discrete_samples}
\end{figure}

\subsection{Blocking}

We will consider continuous Markov process where particles have independent exponential clocks, where 
the $j$-th particle's clock has rate $\pi_j$, and when the clock rings, if the $(j-1)$-th particle is not at the same site, then the $j$-th particle jumps one step to the right.
We take the limit $p \to 0$ in the integral formula~\cite[Thm.~4.19]{IMS22}, and using the classical formula $e^x = \lim_{n\to\infty} (1 + x/n)^n$, we obtain the following.

\begin{cor}
\label{cor:continuous_C}
The continuous time limit of Case~C, for $t \in (0, \infty)$, is
\[
\prob_C(\bG(t) = \lambda| \bG(0) = \mu) = \prod_{j=1}^{\ell} e^{-\pi_j t} \bpi^{\lambda/\mu} \det\left[ \oint_{\widetilde{\gamma}_r} \frac{\prod_{k=1}^{i-1} (1 - \beta_k w^{-1}) e^{tw}}{\prod_{k=1}^{j-1} (1 - \beta_k w^{-1}) w^{(\lambda_i - i) - (\mu_j - j)}} \frac{dw}{2\pi\ii w} \right]_{i,j=1}^{\ell},
\]
where the contour $\widetilde{\gamma}_r$ is a circle centered at the origin with radius $r$ satisfying $r > \abs{\beta_1}, \abs{\beta_2}, \ldots$.
\end{cor}

In Corollary~\ref{cor:continuous_C}, if we substitute $w^{-1} = z$, note the $j$-th particle in the fermionic positioning is at $\nu_j - j$ from the bosonic positions $\nu$, reindexing the particles $j \mapsto \ell + 1 - j$ (as in Remark~\ref{rem:JR_comparison}), and taking the transpose of the determinant, we recover~\cite[Thm.~1]{RS06}.
Furthermore, by taking the same limit of the multi-point distribution from Theorem~\ref{thm:mp_case_C_alt}, we recover~\cite[Cor.]{RS06}.
On the other hand, we can also derive the TASEP master equation~\cite[Eq.~(3)]{RS06} from Theorem~\ref{thm:transition_prob} and verify that the boundary conditions~\cite[Eq.~(2)]{RS06} hold in the limit.

\begin{thm}
\label{thm:continuous_blocking}
The continuous time limit of Case~C $\prob_C(\bG(t) = \lambda| \bG(0) = \mu) $ satisfies
\begin{gather}
\label{eq:master_TASEP}
\frac{d}{dt} \prob_C(\bG(t) = \lambda | \bG(0) = \mu) = \sum_{s=1}^{\ell} \pi_s \prob_C(\bG(t) = \lambda | \bG(0) = \mu) - \sum_{s=1}^{\ell} \pi_s \prob_C(\bG(t) = \lambda - \epsilon_s | \bG(0) = \mu),
\\
\label{eq:boundary_condition_C}
\pi_s \prob_C(\bG(t) = \lambda - \epsilon_s | \bG(0) = \mu) = \pi_{s+1} \prob_C(\bG(t) = \lambda | \bG(0) = \mu)
\qquad \text{ if } \lambda_s = \lambda_{s+1}.
\end{gather}
\end{thm}

\begin{proof}
Define $t_n := pn$, and we will take the limit $n \to \infty$ such that $t_n \to t$, which necessarily means that $p \to 0$.
Consider the quantity
\begin{equation}
\label{eq:derivative_defn}
\frac{\prob_C(\bG(t_{n+1}) = \lambda \mid \bG(0) = \mu) - \prob_C(\bG(t_n) = \lambda \mid \bG(0) = \mu)}{t_{n+1} - t_n},
\end{equation}
and under the limit $n \to \infty$, this converges to
\begin{equation}
\label{eq:derivative_limit}
\frac{d}{dt} \prob_C(\bG(t) = \lambda \mid \bG(0) = \mu)
\end{equation}
as $t_{n+1} - t_n = p \to 0$.
Since we are taking our time scale to be $p$, Theorem~\ref{thm:transition_prob} with~\cite[Thm.~4.19]{IMS22} yields
\[
\prob_C(\bG(t_n) = \lambda | \bG(0) = \mu) = \prod_{i=1}^{\ell} (1 - \pi_i p)^n \bpi^{\lambda/\mu} \det \left[ \oint_{\widetilde{\gamma}_r} F_{ij}(n) \frac{dw}{2\pi\ii w} \right]_{i,j=1}^{\ell},
\]
where
\begin{equation}
\label{eq:integrand_caseC}
F_{ij}(n) = \frac{\prod_{k=1}^{i-1} (1 - \beta_k w^{-1})}{\prod_{k=1}^{j-1} (1 - \beta_k w^{-1})} \frac{(1 - p w)^{-n}}{w^{(\lambda_i-i)-(\mu_j-j)}}.
\end{equation}
(Note that we have \emph{not} included the integral in the definition of $F_{ij}(n)$.)
We can rewrite
\begin{equation}
\label{eq:split_derivative}
\begin{aligned}
\eqref{eq:derivative_defn} & = p^{-1} \left( \prod_{i=1}^{\ell} (1 - \pi_i p) - 1) \right) \prod_{i=1}^{\ell} (1 - \pi_i p)^n \bpi^{\lambda/\mu} \det \left[ \oint_{\widetilde{\gamma}_r} F_{ij}(n+1) \frac{dw}{2\pi\ii w} \right]_{i,j=1}^{\ell}
\\
&\hspace{20pt} + p^{-1} \prod_{i=1}^{\ell} (1 - \pi_i p)^n \bpi^{\lambda/\mu} \left( \det \left[ \oint_{\widetilde{\gamma}_r} F_{ij}(n+1)  \frac{dw}{2\pi\ii w} \right]_{i,j=1}^{\ell} - \det \left[ \oint_{\widetilde{\gamma}_r} F_{ij}(n) \frac{dw}{2\pi\ii w} \right]_{i,j=1}^{\ell} \right).
\end{aligned}
\end{equation}

Now we can take the limit $n \to \infty$ and $p \to 0$ in the first term of~\eqref{eq:split_derivative} as
\begin{align*}
p^{-1} \left( \prod_{i=1}^{\ell} (1 - \pi_i p) - 1) \right) = -\sum_{s=1}^{\ell} \pi_s + O(p) \xrightarrow{p\to0} -\sum_{s=1}^{\ell} \pi_s,
\\
(1 - \pi_i p)^n = (1 - \pi_i t / n)^n \xrightarrow{n\to\infty} e^{-\pi_i t}.
\end{align*}
and so the first term in the limit becomes
\begin{equation}
\label{eq:blocking_first_term}
-\sum_{s=1}^{\ell} \pi_i \prob_C(\bG(t) = \lambda | \bG(0) = \mu).
\end{equation}
Using the multilinearity of the determinant with $(1 - pw)^{-n} = (1- pw)(1 - pw)^{-n-1}$, we can rewrite the second term of~\eqref{eq:split_derivative} as
\[
\sum_{s=0}^{\ell} \prod_{i=1}^{\ell} (1 - \pi_i p)^n \bpi^{\lambda/\mu} \det \left[ \oint_{\widetilde{\gamma}_r} w^{\delta_{is}} F_{ij}(n+1) \frac{dw}{2\pi\ii w} \right]_{i,j=1}^{\ell} + O(p),
\]
which in the limit becomes
\begin{equation}
\label{eq:blocking_second_term}
\sum_{s=1}^{\ell} \pi_s \prob_C(\bG(t) = \lambda - \epsilon_s | \bG(0) = \mu).
\end{equation}
Therefore, evaluating~\eqref{eq:derivative_defn} in two different ways from~\eqref{eq:derivative_limit},~\eqref{eq:blocking_first_term}, and~\eqref{eq:blocking_second_term} yields the claim~\eqref{eq:master_TASEP}.

By the multilinearity of the determinant, we can rewrite the boundary condition~\eqref{eq:boundary_condition_C} as a single determinant after applying Corollary~\ref{cor:continuous_C} (which we want to show is equal to $0$).
The result is exactly the determinant for $\prob_C(\bG(t) = \lambda | \bG(0) = \mu)$ from Corollary~\ref{cor:continuous_C} except in the $s$-th row, where the entries are
\begin{align*}
& \oint_{\widetilde{\gamma}_r} \frac{\prod_{k=1}^{s-1} (1 - \beta_k w^{-1}) e^{tw}}{\prod_{k=1}^{j-1} (1 - \beta_k w^{-1})} \left( \frac{1}{w^{(\lambda_s - 1 - s) - (\mu_j - j)}} - \frac{\pi_{s+1}}{w^{(\lambda_s - s) - (\mu_j - j)}} \right) \frac{dw}{2\pi\ii w}
\\ & \hspace{50pt} = \oint_{\widetilde{\gamma}_r} \frac{\prod_{k=1}^{s-1} (1 - \beta_k w^{-1}) e^{tw}}{\prod_{k=1}^{j-1} (1 - \beta_k w^{-1})} \frac{1 - \beta_s w^{-1}}{w^{(\lambda_s - (s+1)) - (\mu_j - j)}} \frac{dw}{2\pi\ii w}.
\end{align*}
Since $\lambda_{s+1} = \lambda_s$, we see that the $s$-th row is precisely the $(s+1)$-th row, and hence the determinant is $0$ as desired.
\end{proof}

\subsection{Pushing}

We will consider continuous Markov process where particles have independent exponential clocks, where the $j$-th particle's clock has rate $\pi_j$, and when the clock rings, all particles $j' \leq j$ at the site of the $j$-th particle move one step to the right.
We can similarly describe the transition probability as a determinant of contour integrals for the pushing behavior.

\begin{cor}
The continuous time limit of Case~A, for $t \in (0, \infty)$, is
\[
\prob_A(\bG(t) = \lambda| \bG(0) = \mu) = \prod_{j=1}^{\ell} e^{-\pi_j t} \bpi^{\lambda/\mu} \det\left[ \oint_{\gamma_r} \frac{\prod_{k=1}^{j-1} (1 - \pi_k^{-1} w) e^{tw}}{\prod_{k=1}^{i-1} (1 - \pi_k^{-1} w) w^{(\lambda_i - i) - (\mu_j - j)}} \frac{dw}{2\pi\ii w} \right]_{i,j=1}^{\ell}.
\]
\end{cor}

By the same computations as in the blocking behavior case, we can also show the continuous time limit of Case~A also satisfies the master equation~\cite[Eq.~(3)]{RS06}, but the boundary conditions are different.
The proof of the boundary condition is again using the multilinearity of the determinant instead combining the $(s+1)$-th rows (but still showing it equals the $s$-th row).

\begin{thm}
\label{thm:continuous_pushing}
The continuous time limit $\prob_A(\bG(t) = \lambda | \bG(0) = \mu)$ satisfies~\eqref{eq:master_TASEP} but with the boundary conditions
\begin{equation*}
\label{eq:boundary_condition_A}
\pi_s \prob_A(\bG(t) = \lambda + \epsilon_{s+1} | \bG(0) = \mu) = \pi_{s+1} \prob_A(\bG(t) = \lambda | \bG(0) = \mu)
\qquad \text{ if } \lambda_s = \lambda_{s+1}.
\end{equation*}
\end{thm}

\section{Canonical particle process}
\label{sec:canonical_process}

The goal of this section is to describe the particle process whose transition kernel naturally uses the canonical Grothendieck polynomials.
We will start with explicitly defining the stochastic process, and then we will show how to interpret it using the noncommutative operators $\UU^{(\bal,\bbe)}$ (in contrast to Section~\ref{sec:operator_dynamics}).

Recall that $G(j, i)$ denotes the position of the $j$-th particle at time $i$.
The positions of the particles is defined recursively by the formula
\begin{equation}
\label{eq:geometric_blocking_rule}
G(j,i) = \min\bigl( G(j,i-1) + w_{ji}, G(j-1,i-1) \bigr),
\end{equation}
by convention $G(0,i-1) := \infty$, where the random variable $w_{ij}$ --- which now depends on $G(j, i-1)$ --- is determined by the \defn{inhomogeneous geometric distribution} defined by
\begin{equation}
\label{eq:refined_geometric}
\prob_{\mcG}(w_{ji} = m' \mid G(j,i-1) = m) := \frac{1 - \pi_j x_i}{1 + \alpha_{m+m'} x_i} \prod_{k=m}^{m+m'-1} \frac{(\alpha_k + \pi_j) x_i}{1 + \alpha_k x_i}.
\end{equation}
In other words, the $j$-th particle at time $i$ attempts to jump $w_{ji}$ steps, but can be blocked by the $(j-1)$-th particle, which updates its position after the $j$-th particle moves.

Let us digress slightly on why~\eqref{eq:refined_geometric} is called an inhomogeneous geometric distribution.
We can realize it as the waiting time for a failure in sequence of Bernoulli variables (\textit{i.e.}, weighted coin flips), but the $k$-th trial given a probability of success $(\alpha_k + \pi_j) x_i (1 + \alpha_k x_i)^{-1}$.
Indeed, we note that the probability of a failure is
\[
1 - \frac{\alpha_k x_i + \pi_j x_i}{1 + \alpha_k x_i} = \frac{1 - \pi_j x_i}{1 + \alpha_k x_i}.
\]
Hence, this gives us a sampling algorithm for the distribution $\prob_{\mcG}$.
We illustrate the effectiveness of this sampling in Figure~\ref{fig:sample_comparison}.
This perspective also allows us to easily see that we have a probability measure on $\ZZ_{\geq m}$ for any fixed $m$.
The case when $\bpi = 0$ can also be seen as a projection of the Warren--Windridge dynamics~\cite{WW09}; see also~\cite[Sec.~2.2]{Assiotis23}.

\begin{figure}
\[
\includegraphics[width=.9\textwidth]{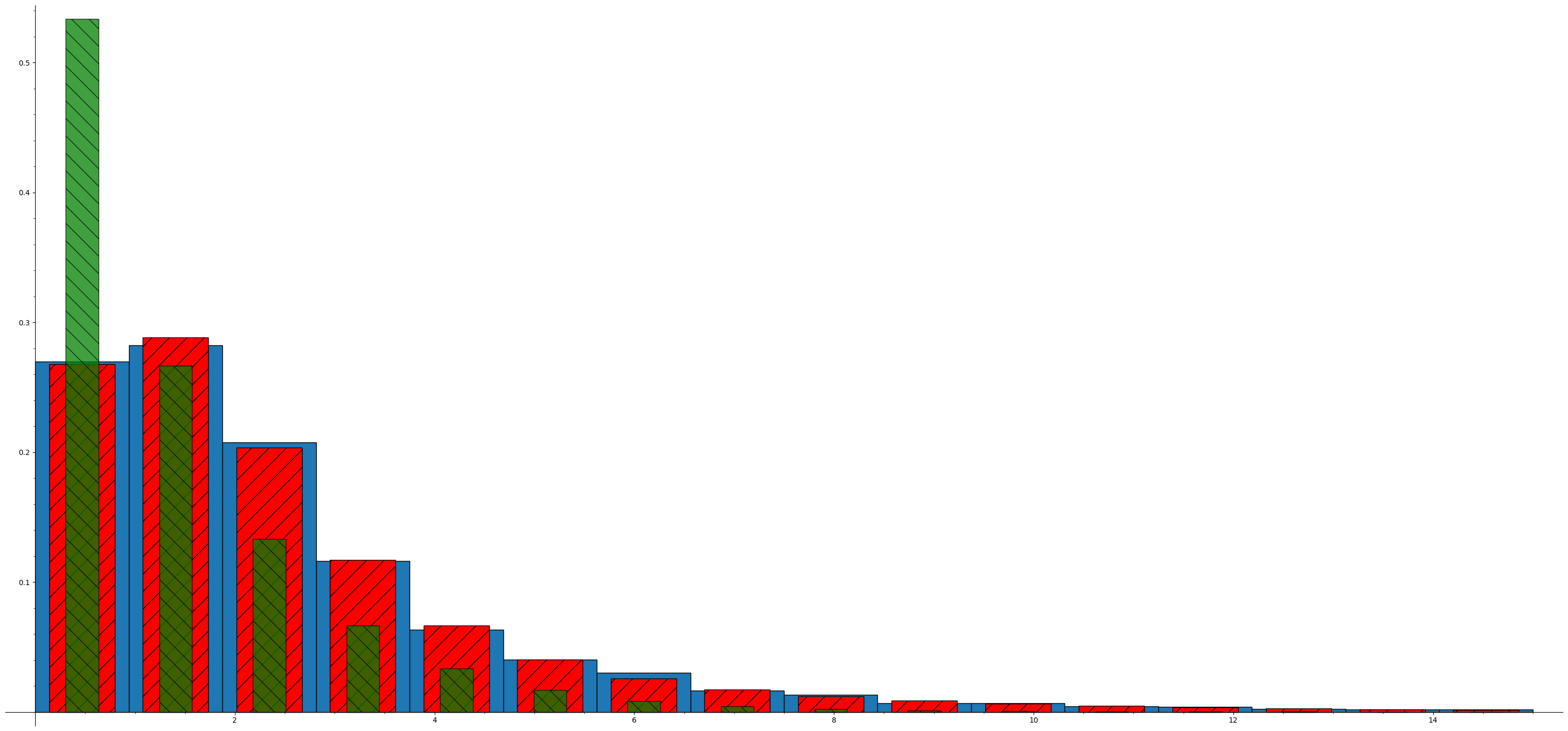}
\]
\caption{A sampling using $10000$ samples of the modified geometric distribution $\prob_{\mcG}$ for $x_i = 1$, $\pi_j = .5$, and $\alpha_k = 1 - k e^{-k/2}$ (blue) under the exact distribution (red), which is under the geometric distribution (green).}
\label{fig:sample_comparison}
\end{figure}

We will give some remarks on the meaning of the $\bal$ parameters.
From the behavior of the operators $\UU^{(\bal,\bbe)}$, it would be tempting to consider the $\bal$ parameters as a viscosity, but for $\bal > 0$, we have $\prob_{\mcG}(w_{ji} = k) > \prob_{Ge}(w_{ji} = k)$.
Thus, in this case, the $\bal$ parameters act as a current being applied to the system, the strength (and direction) of which can vary at each position.
On the other hand, when $\bal < 0$, we have $\prob_{\mcG}(w_{ji} = k) < \prob_{Ge}(w_{ji} = k)$, and so indeed $\bal$ then acts as (position-based) viscosity.
See Figure~\ref{fig:discrete_canonical} and compare with Figure~\ref{fig:discrete_samples}(left).
We can also introduce locations where certain particles must stop by having $-\alpha_k = \pi_j$ since this would have $\mcP_{\mcG}(w_{ij} = k') = 0$ for all $k'$ that would move the $j$-th particle past position $k$.

\begin{figure}
\[
\includegraphics[width=.45\textwidth]{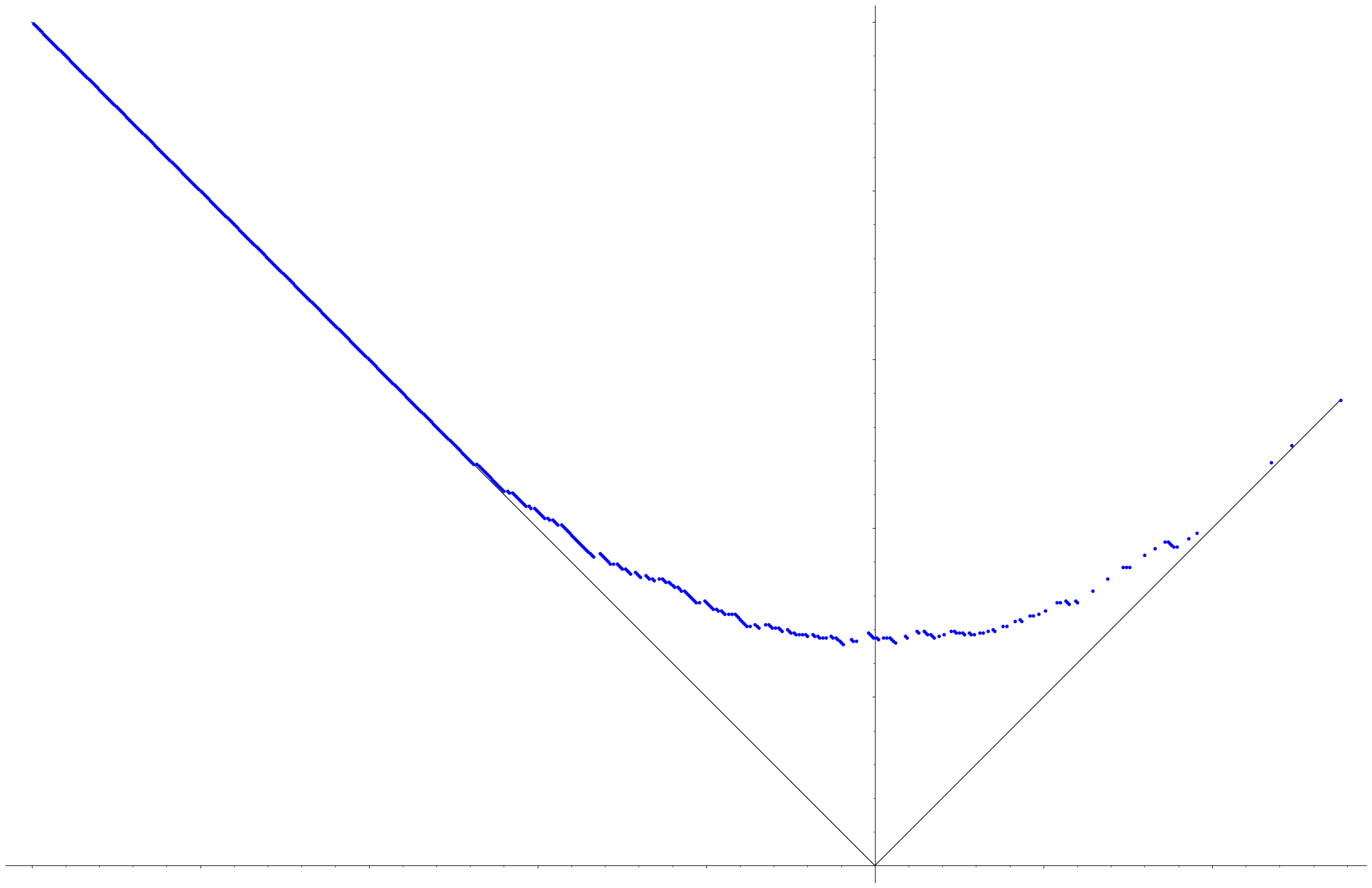}
\qquad
\includegraphics[width=.45\textwidth]{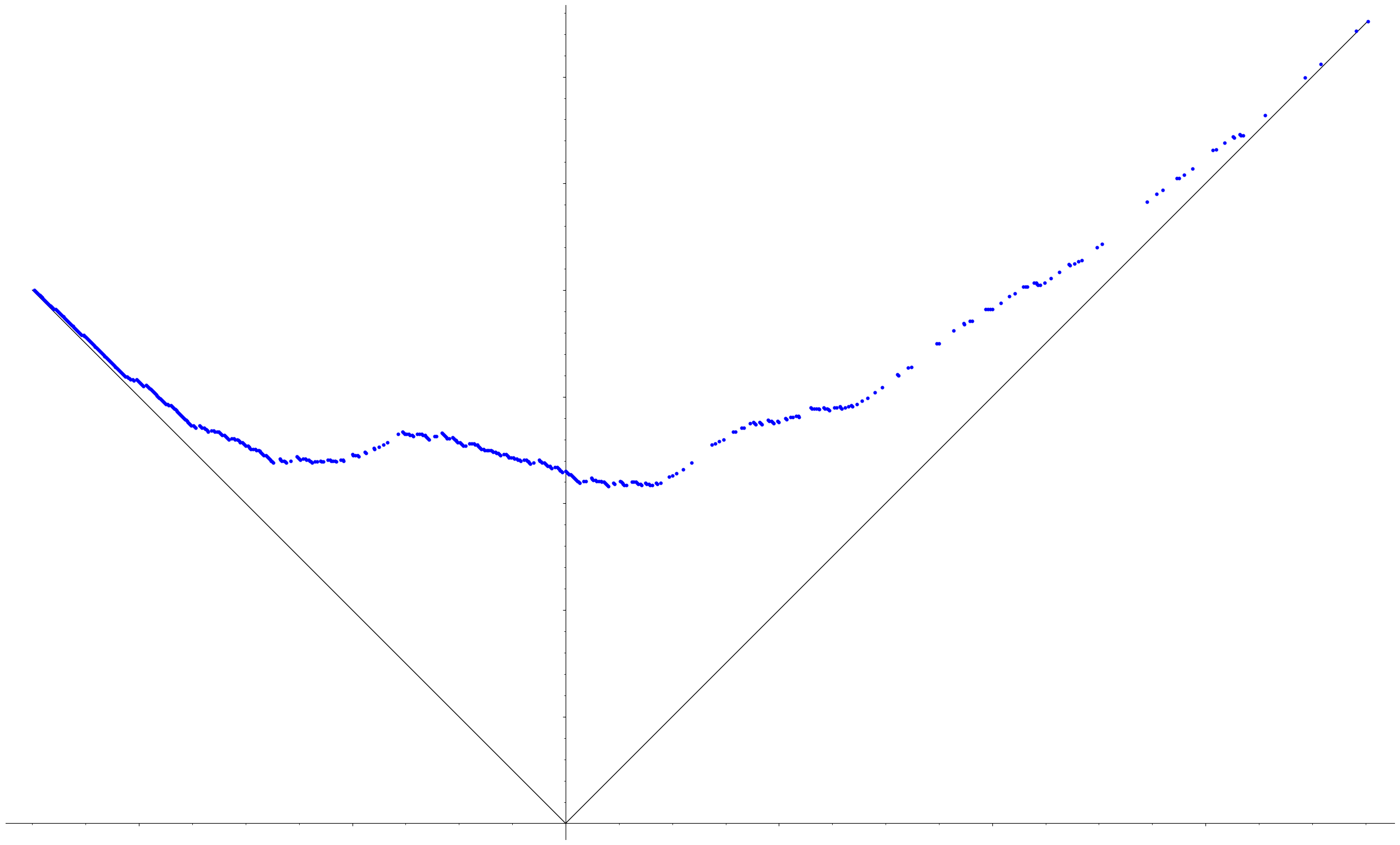}
\]
\caption{Samples of blocking TASEP with $\ell = 500$ particles after $n = 50000$ time steps with (left) $\bpi = 1$, $\xx = 0.01$, and $\bal = -0.5$; (right) $\bpi = 0.5$, $\xx = .2$, and $\alpha_k = 0.5 \sin(k/50)^6$.}
\label{fig:discrete_canonical}
\end{figure}

To see how to obtain this process using the noncommutative operators $\UU^{(\bal,\bbe)}$, we initiate by taking the skew Cauchy formula (Theorem~\ref{thm:skew_cauchy}) with $\nu = \emptyset$ and with the specializations $\yy = \bpi_1$ and $\beta_j = \pi_{j+1}$, yielding
\begin{equation}
\label{eq:skew_cauchy_specialized}
\sum_{\lambda} \G_{\lambda\ds\mu}(\xx_n; \bal, \bbe) \dG_{\lambda}(\bpi_1; \bal, \bbe) = \prod_i (1 - \pi_1 x_i)^{-1} g_{\mu}(\bpi_1; \bal, \bbe).
\end{equation}
In particular, if we let $\widehat{\lambda}_i = \lambda_i - 1$ for all $1 \leq i \leq \ell(\lambda)$, then from the combinatorial description of~\cite[Thm.~4.2]{HJKSS25}, we have
\[
g_{\lambda}(\bpi_1; \bal, \bbe) = \bpi^{1^{\ell(\lambda)}} \prod_{(i,j) \in \widehat{\lambda}} (\alpha_i + \pi_j).
\]
Hence, Equation~\eqref{eq:skew_cauchy_specialized} can be considered a Littlewood-type identity for canonical Grothendieck polynomials.
Dividing this by the factor on the right hand side and taking the term corresponding to $\lambda$, we obtain a probability distribution for $n$ step random growth process (since we must have $\mu \subseteq \lambda$ and currently the interpretation we have described is only on partitions) given by
\begin{equation}
\label{eq:canonical_prob}
\prob_{\mcC,n}(\lambda | \mu) = \prod_{i=1}^n (1 - \pi_1 x_i) \bpi^{1^{\ell(\lambda)} / 1^{\ell(\mu)}} \prod_{(i,j) \in \widehat{\lambda} / \widehat{\mu}} (\alpha_i + \pi_j) \G_{\lambda\ds\mu}(\xx_n; \bal, \bbe).
\end{equation}
Note that Equation~\eqref{eq:skew_cauchy_specialized} is equivalent to $\sum_{\lambda} \prob_{\mcC,n}(\lambda | \mu) = 1$ for any fixed $\mu$ and $n$.

Rephrasing Equation~\eqref{eq:canonical_prob} and adding an $\alpha_0 = 0$ parameter in order to simplify the product in $\dG_{\lambda}(\bpi_1; \bal, \bbe)$, what we have computed are coefficients
\[
C_{\lambda\mu} = \prod_{i=1}^n (1 - \pi_1 x_i) (\vec{\bal} + \bbe)^{\lambda/\mu}, \qquad \text{where } (\vec{\bal} + \bbe)^{\lambda/\mu} := \prod_{(i,j) \in \lambda / \mu} (\alpha_{i-1} + \pi_j)
\]
that is defined to be $0$ if $\lambda \not\supseteq \mu$, such that
\begin{equation}
\label{eq:canonical_prob_bra}
C_{\lambda\mu} \cdot {}^{[\bal,\bbe]} \bra{\mu} e^{H(\xx_n)} \ket{\lambda}^{[\bal,\bbe]} = \prob_{\mcC,n}(\lambda|\mu)
\quad \Longleftrightarrow \quad
{}^{[\bal,\bbe]}\bra{\mu} e^{H(\xx_n)} = \sum_{\lambda \supseteq \mu} \frac{\prob_{\mcC,n}(\lambda|\mu)}{C_{\lambda\mu}} \cdot {}^{[\bal,\bbe]} \bra{\lambda},
\end{equation}
where the equivalence of the two formulas is given by the orthonormality~\eqref{eq:orthonormal_basis}.

We now restrict ourselves to a single timestep at time $i$ in order to encode the growth process as a particle process by using the operators $\UU^{(\bal,\bbe)}$.
This incurs no loss of generality as $\prob_{\mcC,n+n'}(\lambda|\mu) = \sum_{\nu} \prob_{\mcC,n}(\lambda|\nu) \prob_{\mcC,n'}(\nu|\mu)$ by the branching rules (Proposition~\ref{prop:branching_rules}) and we have a Markov process.
Recall the operator $\mcT_C$ from Section~\ref{sec:block_op_proof}, and we define $\mcT_{\mcC}$ as $\mcT_C$ except using the operators $\UU^{(\bal,\bbe)}$.
Since Theorem~\ref{thm:noncommutative_blocking} holds for $\UU^{(\bal,\bbe)}$, we have
\[
{}^{[\bal,\bbe]}\bra{\mu} e^{H(x_i)} = \prod_{j=2}^{\infty} \frac{1}{1 - \pi_j x_i} \cdot {}^{[\bal,\bbe]} \bra{\mcT_{\mcC} \cdot \mu}
\]
by the same argument in Section~\ref{sec:block_op_proof}.
Thus, if we consider the expansion
\begin{equation*}
\bra{\mcT_{\mcC} \cdot \mu} = \sum_{\lambda} B_{\lambda\mu} \cdot {}^{[\bal,\bbe]} \bra{\lambda},
\end{equation*}
and matching coefficients in~\eqref{eq:canonical_prob_bra} (equivalently, pairing with $\ket{\lambda}^{[\bal,\bbe]}$), we obtain
\[
\prob_{\mcC}(\lambda|\mu) = \frac{B_{\lambda\mu}}{(\vec{\bal} + \bbe)^{\lambda/\mu}} \prod_{j=1}^{\infty} (1 - \pi_j x_i)^{-1}.
\]

\begin{ex}
\label{ex:canonical_op_prob}
\ytableausetup{boxsize=.6em}%
We will redo the computation in Example~\ref{ex:blocking_C} except now using the general $\UU^{(\bal,\bbe)}$ operators.
Recall that $\mu = (1, 1)$ and $\pi_j = 0$ for all $j > 3$.
Using~\eqref{eq:ncsym_h123} and recalling we consider $\alpha_0 = 0$, we compute
\begin{align*}
h_1(\UU_3) \cdot \mu & = \left(-\alpha_1 \ydiagram{1,1} + \ydiagram{2,1}\right) + \pi_1 \ydiagram{1,1} + \ydiagram{1,1,1}\,,
\\ h_2(\UU_3) \cdot \mu & = \left(\alpha_1^2 \ydiagram{1,1} - (\alpha_1 + \alpha_2) \ydiagram{2,1} + \ydiagram{3,1} \right) 
 + \pi_1 \left( -\alpha_1 \ydiagram{1,1} + \ydiagram{2,1} \right)  
 + \left( -\alpha_1 \ydiagram{1,1,1} + \ydiagram{2,1,1}\right) 
\\ & \hspace{20pt} + \pi_1^2 \ydiagram{1,1}  
 + \pi_1 \ydiagram{1,1,1}  
 + \pi_2 \ydiagram{1,1,1}\,,  
\\ h_3(\UU_3) \cdot \mu & = \left(-\alpha_1^3 \ydiagram{1,1} + (\alpha_1^2 + \alpha_1 \alpha_2 + \alpha_2^2) \ydiagram{2,1} - (\alpha_1 + \alpha_2 + \alpha_3) \ydiagram{3,1} + \ydiagram{4,1} \right) 
\\ & \hspace{20pt} + \pi_1 \left( \alpha_1^2 \ydiagram{1,1} - (\alpha_1 + \alpha_2) \ydiagram{2,1} + \ydiagram{3,1} \right)  
 + \left( \alpha_1^2 \ydiagram{1,1,1} - (\alpha_1 + \alpha_2) \ydiagram{2,1,1} + \ydiagram{3,1,1} \right) 
\\ & \hspace{20pt} + \pi_1^2 \left( -\alpha_1 \ydiagram{1,1} + \ydiagram{2,1} \right)  
 + \pi_1 \left( -\alpha_1 \ydiagram{1,1,1} + \ydiagram{2,1,1} \right)  
 + \pi_2 \left( -\alpha_1 \ydiagram{1,1,1} + \ydiagram{2,1,1} \right)  
\\ & \hspace{20pt} + \pi_1^3 \ydiagram{1,1}  
 + \pi_1^2 \ydiagram{1,1,1}  
 + \pi_1 \pi_2 \ydiagram{1,1,1}  
 + \pi_2^2 \ydiagram{1,1,1}\,.  
 \end{align*}
 Recall that $A_k = -\bal_k$.
Therefore, we have
\begin{align*}
{}^{[\bal,\bbe]} \bra{\mcT_{\mcC} \cdot \mu} & = (1 + h_1(\bbe_1 \sqcup A_1) x_i + h_2(\bbe_1 \sqcup A_1) x_i^2 + h_3(\bbe_1 \sqcup A_1) x_i^3 + \cdots) \cdot {}^{[\bal,\bbe]} \bra{1,1}
\\ & \hspace{20pt} + x_i (1 + h_1(\bbe_1 \sqcup A_2) x_i + h_2(\bbe_1 \sqcup A_2) x_i^2 + \cdots) \cdot {}^{[\bal,\bbe]} \bra{2,1}
\\ & \hspace{20pt} + x_i (1 + h_1(\bbe_2 \sqcup A_1) x_i + h_2(\bbe_2 \sqcup A_1) x_i^2 + \cdots) \cdot {}^{[\bal,\bbe]} \bra{1,1,1}
\\ & \hspace{20pt} + x_i^2 (1 + h_1(\bbe_1 \sqcup A_3) x_i  + \cdots) \cdot {}^{[\bal,\bbe]} \bra{3,1}
\\ & \hspace{20pt} + x_i^2 (1 + h_1(\bbe_2 \sqcup A_2) x_i  \cdots) \cdot {}^{[\bal,\bbe]} \bra{2,1,1} + \cdots
\allowdisplaybreaks
\\ & = \frac{(1 + \alpha_1 x_i)^{-1}}{1 - \pi_2 x_i} \cdot {}^{[\bal,\bbe]} \bra{1,1} + \frac{(\alpha_1 x_i + \pi_1 x_i) (1 + \alpha_1 x_i)^{-1} (1 + \alpha_2 x_i)^{-1}}{(1 - \pi_2 x_i) (\vec{\bal} + \bpi)^{(2,1) / \mu}} \cdot {}^{[\bal,\bbe]} \bra{2,1}
\\ & \hspace{20pt} + \frac{(\alpha_0 x_i + \pi_3 x_i) (1 + \alpha_1 x_i)^{-1}}{(1 - \pi_2 x_i)(1 - \pi_3 x_i) (\vec{\bal} + \bpi)^{(1,1,1) / \mu}} \cdot {}^{[\bal,\bbe]} \bra{1,1,1}
\\ & \hspace{20pt} + \frac{(\alpha_1 x_i + \pi_1 x_i) (\alpha_2 x_i + \pi_1 x_i) (1 + \alpha_1 x_i)^{-1}(1 + \alpha_2 x_i)^{-1}(1 + \alpha_3 x_i)^{-1}}{(1 - \pi_2 x_i) (\vec{\bal} + \bpi)^{(3,1) / \mu}} \cdot {}^{[\bal,\bbe]} \bra{3,1}
\\ & \hspace{20pt} + \frac{(\alpha_1 x_i + \pi_1 x_i) (\alpha_0 x_i + \pi_3 x_i) (1 + \alpha_1 x_i)^{-1} (1 + \alpha_2 x_i)^{-1}}{(1 - \pi_2 x_i)(1 - \pi_3 x_i) (\vec{\bal} + \bpi)^{(2,1,1) / \mu}} {}^{[\bal,\bbe]} \bra{2,1,1} + \cdots.
\end{align*}
If we include $\alpha_0$ in the $\UU^{(\bal,\bbe)}$ operators, then all terms will be multiplied by $(1 + \alpha_0 x_i)^{-1}$ since the third particle can move from position $0$.
With this factor, some of the transition probabilities are
\begin{align*}
\prob_{\mcC}(1,1|\mu) & = \frac{(1 - \pi_1 x_i) (1 - \pi_3 x_i)}{(1 + \alpha_0 x_i) (1 + \alpha_1 x_i)},
\allowdisplaybreaks \\
\prob_{\mcC}(2,1|\mu) & = \frac{(\alpha_1 x_i + \pi_1 x_i) (1 - \pi_1 x_i) (1 - \pi_3 x_i)}{(1 + \alpha_0 x_i) (1 + \alpha_1 x_i) (1 + \alpha_2 x_i)},
\allowdisplaybreaks \\
\prob_{\mcC}(1,1,1|\mu) & = \frac{(\alpha_0 x_i + \pi_3 x_i) (1 - \pi_1 x_i)}{(1 + \alpha_0 x_i) (1 + \alpha_1 x_i)},
\allowdisplaybreaks \\
\prob_{\mcC}(3,1|\mu) & = \frac{(\alpha_1 x_i + \pi_1 x_i) (\alpha_2 x_i + \pi_1 x_i) (1 - \pi_1 x_i) (1 - \pi_3 x_i)}{(1 + \alpha_0 x_i) (1 + \alpha_1 x_i) (1 + \alpha_2 x_i) (1 + \alpha_3 x_i)},
\allowdisplaybreaks \\
\prob_{\mcC}(2,1,1|\mu) & = \frac{(\alpha_1 x_i + \pi_1 x_i) (\alpha_0 x_i + \pi_3 x_i) (1 - \pi_1 x_i)}{(1 + \alpha_0 x_i) (1 + \alpha_1 x_i) (1 + \alpha_2 x_i)}.
\end{align*}
\end{ex}

Like at $\bal = 0$, which is the Case~C blocking behavior with geometric jumps, any individual (free) particle motion is (up to changing $\pi_j \mapsto \pi_1$) equivalent to the first particle's motion.
Thus, let us consider $\lambda$ with $\ell(\lambda) = 1$, and a straightforward computation (say, at time $i$) using either the operators $\UU^{(\bal,\bbe)}$ or the combinatorial description of $\G_{\lambda\ds\mu}(x_i; \bal, \bbe)$ yields
\[
\prob_{\mcC}\bigl( m'| m \bigr) = \frac{1 - \pi_j x_i}{1 + \alpha_{m'+m} x_i} \prod_{k=m}^{m+m'-1} \frac{(\alpha_k + \pi_j) x_i}{1 + \alpha_k x_i},
\]
which is precisely the measure specified in~\eqref{eq:canonical_prob}.
By~\eqref{eq:canonical_prob}, for any fixed $m$ this is a probability measure for all $\alpha_k + \pi_j \geq 0$ with the natural assumptions $0 \leq \pi_j x_i < 1$ and $\alpha_k x_i \geq -1$.
This can also be extended to include generic parameters $(\alpha_k)_{k \in \ZZ}$ by shifting the parameters $\alpha_k \mapsto \alpha_{k\pm1}$.
Therefore, we can perform the same analysis as in Section~\ref{sec:block_op_proof} to show the following.

\begin{thm}
\label{thm:canonical_process}
Suppose $\ell(\lambda) \leq \ell$, $\pi_j x_i \in (0, 1)$, $\alpha_k x_i > -1$, and $\alpha_k + \pi_j \geq 0$ for all $i, j, k$.
Set $\beta_j = \pi_{j+1}$.
Let $\prob_{\mcC,n}(\lambda|\mu)$ denote the $n$-step transition probability for the Case~C particle system except using the distribution~\eqref{eq:refined_geometric} for the jump probability of the particles, as given by~\eqref{eq:geometric_blocking_rule}.
Then the $n$-step transition probability is given by
\[
\prob_{\mcC,n}(\lambda|\mu) = \prod_{i=1}^n (1 - \pi_1 x_i) (\vec{\bal} + \bpi)^{\lambda/\mu} G_{\lambda\ds\mu}(\xx_n; \bal, \bbe).
\]
\end{thm}

\begin{remark}
\label{rem:boson_pos_param}
Since the $\bal$ parameters used, and hence the probabilities, now depend on the positions of the particles, we can only work with the bosonic model.
Indeed, switching to the fermionic model will require us to introduce additional parameters $\alpha_k$ for $k < 0$, in which case Theorem~\ref{thm:canonical_process} no longer holds, or to account for the shifting of positions by replacing $\alpha_k \mapsto \alpha_{k+j}$ for the $j$-th particle distribution $\prob_{\mcG}$.
\end{remark}

We could also prove Theorem~\ref{thm:canonical_process} by using the combinatorics of hook-valued tableaux~\cite{HJKSS24,HJKSS25,Yel17} as in Section~\ref{sec:combinatorial_C}.
The key observation is that we have a factor $x_i (1 - \alpha_k x_i)^{-1}$ for every box in the $k$-th column that would normally contain an $i$ in the set-valued tableaux (over all $k$).
In more detail, we take the minimal entries of each hook (the corner entry) in the tableau to describe the basic motion of the particles.
The leg (the column part except for the corner) corresponds to the choice between $1$ and $-\pi_i x_j$ in the numerator of the normalization constant as before.
The arm (the row part except for the corner) comes from waiting at that particular position and contributes an $-\alpha x_i$, which contributes a factor of $(1 + \alpha x_i)^{-1}$ as in the Case~B combinatorial proof.
The associated combinatorics when $\bbe = \beta = -\alpha = -\bal$, where no particles will move, was studied in~\cite[Sec.~13.4]{Yel17}.

From~\cite[Thm.~4.1]{IMS22}, we obtain determinant formulas for $\prob_{\mcC,n}(\lambda|\mu)$, where we can write the entries of the matrix as contour integrals~\cite[Thm.~4.19]{IMS22}.
We can also redo he computation in Theorem~\ref{thm:mp_case_C} at this level of generality to obtain a multi-point distribution for this process.

\begin{thm}
\label{thm:mp_canonical}
The multi-point distribution for Case~C inhomogeneous process with $\ell$ particles is given by
\[
\prob_{\geq,n}(\nu|\mu) = 
\prod_{j=2}^{\ell} \prod_{i=1}^n (1-\pi_j x_i)^{-1} \det \big[ h_{\nu_i - \mu_j - i + j}\bigl(\xx \ds (A_{(\mu_j, \nu_i]} \sqcup \bpi_i / \bbe_j \bigr) \bigr]_{i,j=1}^{\ell}.
\]
\end{thm}

We can give another, more simple, proof for the case when $\bal = \alpha$.
This will follow from a straightforward generalization of the unrefined case~\cite[Prop.~3.4]{Yel17}, noting our sign convention means we need to substitute $-\alpha$.

\begin{prop}
We have
\begin{equation}
\label{eq:canonical_from_regular}
G_{\lambda}(\xx; \alpha, \bbe) = G_{\lambda}(\xx/(1 + \alpha \xx); 0, \alpha + \bbe),
\end{equation}
where we substitute $x_i \mapsto x_i / (1 + \alpha x_i)$ and $\beta_i \mapsto \alpha + \beta_i$.
\end{prop}

Indeed, under this substitution, we have
\begin{equation}
\label{eq:canonical_substitution}
\pi_j x_i \longmapsto \frac{(\alpha + \pi_j) x_i}{1 + \alpha x_i}.
\end{equation}
Hence, the geometric distribution $\prob_{Ge}$ transforms to the distribution $\prob_{\mcG}$ in~\eqref{eq:refined_geometric} with $\bal = \alpha$.
Moreover, in our formula for $\prob_{C,n}$ from Theorem~\ref{thm:transition_prob}, the total $\xx$ degree and total $\bpi$ degree in each term of $\bpi^{\lambda/\mu} G_{\lambda \ds \mu}(\xx; \bbe)$ are equal, and so we can perform the substitution~\eqref{eq:canonical_substitution}.
Thus, we obtain Theorem~\ref{thm:canonical_process} in the case $\bal = \alpha$.

\begin{remark}
\label{rem:canonical_KPS_comparison}
Let us discuss the relationship between this model and the doubly geometric inhomogeneous corner growth model defined in~\cite{KPS19}.
In their corresponding TASEP model, there is an additional set of position-dependent parameters $\bnu$ that are only involved after the initial movement of the particle (akin to static friction).
Yet, if we set $\bnu = 0$, then the model in~\cite{KPS19} is the fermionic realization of our model (\textit{cf.}~Remark~\ref{rem:boson_pos_param}) at $\bbe = 0$ with their parameters $(\mathbf{a}, \bbe)$ equaling our parameters $(\bal, \xx)$.
Hence, we end up with another TASEP version that is equivalent to Case~B.
It would be interesting to see if the model in~\cite{KPS19} can be recovered from the free fermionic description such as by using a specialization of the skew Cauchy identity.

We also remark that our model with $\bpi = 0$ was studied in~\cite{Assiotis23}, but using very different techniques based on Toeplitz matrices and Markov semigroups.
Therefore, from the specialization of the canonical Grothendieck polynomials, it is essentially Case~B as before, with a more probabilistic link being made by~\cite[Thm.~2.43]{Assiotis23}.
\end{remark}

We can similarly define a Bernoulli process extending Case~B with the Bernoulli probability depending on the positions as
\begin{equation}
\label{eq:pos_bernoulli}
\prob_{\mcB}(w_{ji} = 1 \mid G(j, i-1) = m) := \frac{(\rho_j + \beta_m) x_i}{1 + \rho_j x_i}.
\end{equation}
Analogously to Theorem~\ref{thm:canonical_process} (including its proof), we have the following.

\begin{thm}
\label{thm:conjugate_canonical_process}
Suppose $\lambda_1 \leq \ell$, $\beta_k x_i \in (0, 1)$, $\rho_j x_i > -1$, and $\rho_j + \beta_k \geq 0$ for all $i, j, k$.
Set $\alpha_j = \rho_{j+1}$.
Let $\prob_{\mcB,n}(\lambda|\mu)$ denote the $n$-step transition probability for the Case~B particle system except using the distribution~\eqref{eq:pos_bernoulli} for the jump probability of the particles.
Then the $n$-step transition probability is given by
\[
\prob_{\mcB,n}(\lambda|\mu) = \frac{(\vec{\beta} + \brho)^{\lambda/\mu}}{\prod_{i=1}^n (1 + \rho_1 x_i)} G_{\lambda'\ds\mu'}(\xx_n; \bal, \bbe).
\]
\end{thm}

If we set $\bal = 0$ in this position-dependent version of Case~B, then we end up with a Bernoulli random variable version of~\cite{KPS19} at $\bnu = 0$.

Next, we consider the analogous particle processes with pushing behavior, but we will only consider the geometric distribution case (the analog of Case~A) as the Bernoulli case is entirely parallel.
In this case, the transition probabilities not given by the dual canonical Grothendieck polynomials as one would expect; this essentially comes from the fact that $(\alpha_k + \pi_j)^{-1} \neq \alpha_k^{-1} + \pi_j^{-1}$ (in general).
Despite this, the combinatorial description of Case~A in Section~\ref{sec:bijection_A} defines a $\overline{\dG}_{\lambda/\mu}(\xx; \bal, \bpi)$ as the sum over reverse plane partitions so the $n$-step transition probability satisfies
\[
\prob_{\mcA,n}(\lambda | \mu) = \prod_{j=1}^{\ell} \prod_{i=1}^n (1 - \pi_j x_i) (\vec{\bal} + \bpi)^{\lambda/\mu} \overline{\dG}_{\lambda/\mu}(\xx_n; \bal, \bpi).
\]
We note that $\dG_{\lambda/\mu}(\xx; \bal, \bpi)$ can be defined as a sum over reverse plane partitions but with the weights now depending also on the $\bal$ parameters similar to~\cite{Yel17} (contrast this with~\cite{HJKSS24,HJKSS25}).
For example, the weight $(\alpha_k + \pi_j)$ in $\dG_{\lambda/\mu}(\xx; \bal, \bpi)$ would be replaced by $(\alpha_k + \pi_j)^{-1}$ in $\overline{\dG}_{\lambda/\mu}(\xx_n; \bal, \bpi)$.
Therefore, we end up with new functions, but studying these functions is outside the scope of this paper.
We also remark that $\dG_{\lambda/\mu}(\xx_n;\bal, \bpi)$ does not appear in these transition probabilities is likely tied to the failure of $\uu^{(\bal,\bbe)}$ to satisfy the Knuth relations.

The continuous limit version of this has also been studied in~\cite{Assiotis20,Petrov20}, but as mentioned in the introduction, it is an open problem to go from our results to theirs.

\section{Concluding remarks}
\label{sec:conclusion}

Let us consider what would happen if we swapped the update rules.
We consider the geometric distribution with smallest-to-largest updating first.
In this case for the blocking behavior, the analysis is more subtle as the number of steps that the $j$-th particle can do depends not only on the position of the $(j-1)$-th particle, but also how many steps the $(j-1)$-th particle takes.
(Contrast this last part with the Bernoulli case, where we never have to consider this because each particle can move at most one step.)
The pushing behavior also has the same difficulty added to the computations.
As a result, we do not expect any nice formulas.

On the other hand, for the Bernoulli distribution with largest-to-smallest updating, the analysis is the same as for the geometric case except the particles can only move one step.
If we consider the blocking behavior case, this agrees with the classical simultaneous update for discrete TASEP (which can be encoded by the LPP for Case~A; see, \textit{e.g.},~\cite{Johansson01,MS20}).
However, this might cause some slight complications for the blocking behavior as we have to now consider when particles can freely move, in contrast to the geometric case where they are always constrained (except for the largest particle).
Despite this, a natural guess for the combinatorics would be to use suitably modified increasing tableaux and consider their generating functions.

Another construction to consider based on~\cite{Iwao21} is replacing $\ket{\lambda}_{[\bbe]}$ and $\ket{\lambda}^{[\bbe]}$ by the vectors
\begin{align*}
\ket{\lambda}_{\langle \bbe \rangle} 
& := \prod^{\rightarrow}_{1 \leq i \leq \ell} \left( \psi_{\lambda_i-i} e^{J(\beta_i)} \right) \ket{-\ell},
&
\ket{\lambda}^{\langle \bbe \rangle} 
& :=
\prod^{\rightarrow}_{1 \leq i \leq \ell} \left( \psi_{\lambda_i-i} e^{-J^*(\beta_i)} \right) \ket{-\ell},
\end{align*}
respectively.
Then we use these vectors (and their $*$ versions) to encode such dynamics of a particle system.
It would be interesting to see what properties the resulting functions
\[
k_{\lambda/\mu}(\xx_n; \bbe) := {}_{\langle \bbe \rangle} \bra{\mu} e^{H(\xx_n)} \ket{\lambda}_{\langle \bbe \rangle},
\qquad\qquad
K_{\lambda \ds \mu}(\xx_n; \bbe) := {}^{\langle \bbe \rangle} \bra{\mu} e^{H(\xx_n)} \ket{\lambda}^{\langle \bbe \rangle},
\]
have compared with (dual) Grothendieck polynomials.
Note that at $\bbe = 0$, these reduce to a (skew) Schur function, so they cannot encode the Bernoulli distribution with largest-to-smallest updating since the first particle can only move one step from the step initial condition $\mu = \emptyset$.

Next, we compute an alternative form of our vector $\ket{\lambda}^{[\bbe]}$.
From~\cite[Thm.~5.3]{MS20}, we can write a Grothendieck polynomial as a multiSchur function of~\cite{Lascoux03}
\begin{equation}
\label{eq:G_multiSchur}
\G_{\lambda}(\xx_n; \bbe) = (-1)^{\binom{n}{2}} \bbe^{\rho_n} s_{\widetilde{\lambda}}(\xx_n, \xx_n/\bbe_1^{-1}, \dotsc, \xx_n/\bbe_{n-1}^{-1}),
\end{equation}
where $\rho_n = (n-1, n-2, \dotsc, 1, 0)$ and $\widetilde{\lambda} = (\lambda_1, \lambda_2 + 1, \dotsc, \lambda_n + n - 1)$.
The precise definition of a multiSchur function is not needed as we will immediately use~\cite{Iwao21} to write Equation~\eqref{eq:G_multiSchur} in terms of free fermions:
\begin{align*}
\G_{\lambda}(\xx_n; \bbe)
& = (-1)^{\binom{n}{2}} \bbe^{\rho_n} \bra{\emptyset} e^{H(\xx_n)} \psi_{\lambda_1-1} e^{-H(\beta_1^{-1})} \psi_{\lambda_2-1} e^{-H(\beta_2^{-1})} \dotsm \psi_{\lambda_n-1} e^{-H(\beta_{n-1}^{-1})} \ket{-n}
\\ & = (-1)^{\binom{n}{2}} \bbe^{\rho_n} \bra{\emptyset} e^{H(\xx_n)} \ket{\widetilde{\lambda}}_{[\emptyset/\bbe^{-1}]}.
\end{align*}
Therefore, the orthonormality~\eqref{eq:orthonormal_basis} implies
\begin{equation}
\label{eq:alt_G_vector}
\ket{\lambda}^{[\bbe]} = (-1)^{\binom{n}{2}} \bbe^{\rho_n} \ket{\widetilde{\lambda}}_{[\emptyset/\bbe^{-1}]} = (-\bbe)^{\rho_n} \ket{\widetilde{\lambda}}_{(-\bbe^{-1})}.
\end{equation}
By applying Wick's theorem, we recover~\cite[Thm.~5.12]{MS20}, which refines~\cite[Thm.~1.10]{Kirillov16}.
Using the expressions in~\eqref{eq:alt_G_vector}, it could be possible to derive some new additional formulas involving $G_{\lambda \ds \mu}(\xx_n; \bbe)$ and $G_{\lambda / \mu}(\xx_n; \bbe)$.
For example, a multipoint distribution formula with the partitions contained with $\lambda$ (as opposed to containing $\lambda$ in Theorem~\ref{thm:mp_case_C}) by using a modification of Proposition~\ref{prop:schur_branching}.


\subsection*{Declarations}

This manuscript has no associated datasets.

The authors have no competing interests to declare that are relevant to the content of this article.


\appendix
\section{\texorpdfstring{\textsc{SageMath}}{SageMath} code}

We include our \textsc{SageMath}~\cite{sage} code used to generate Figure~\ref{fig:discrete_samples} and Figure~\ref{fig:continuous_samples}.

\begin{lstlisting}
def sample_continuous(t, ell, p=1, push=False):
    r"""
    Sample from TASEP with ``ell`` particles moving with rate ``p`` after
    ``t`` time with either blocking or pushing behavior.
    """
    def exp_gen(la):
        return -1/la * ln(1-random())

    # We implement the priority queue in reverse order since lists are better
    #  manipulated at the back than the front.
    priority = [[i, exp_gen(p)] for i in range(ell)]
    priority.sort(key=lambda x: x[1], reverse=True)
    pos = [-i for i in range(ell)]
    time = 0
    while time < t:
        # Get the next particle to move
        i, dt = priority.pop()
        time += dt

        # Reinsert it into the priority queue (recall the order is reversed)
        nt = exp_gen(p)  # next time
        add_pos = 0
        for x in priority:
            x[1] -= dt
            if x[1] > nt:
                add_pos += 1
        priority.insert(add_pos, [i, nt])
        assert priority == sorted(priority, key=lambda x: x[1], reverse=True)

        # Update the particle
        if push:
            pos[i] += 1
            while i > 0 and pos[i] == pos[i-1]:
                i -= 1
                pos[i] += 1
        else:
            if i == 0 or pos[i] != pos[i-1] - 1:
                pos[i] += 1

    return pos

def sample_discrete(t, ell, p=1, a=0, scale=0.01, push=False):
    r"""
    Sample from discrete TASEP with ``ell`` particles moving with rate ``p``
    after ``t`` time with either blocking or pushing behavior with the
    scaling parameter ``scale`` (which should be close to `0` to estimate
    the exponential distribution) with parameter ``a``.
    """
    n = floor(t / scale)
    pos = [-i for i in range(ell)]

    try:
        a = RR(a)
        from numpy.random import geometric
        def sample(i):
            return geometric(1-(a+p)*scale/(1+a*scale))  # numpy uses (1-p)^k p)
    except (ValueError, TypeError):
        def sample(i):
            return modified_geometric(a, p, scale, pos[i] + i)

    for _ in range(n):
        # Update the particles in reverse order
        for i in range(ell-1, -1, -1):
            k = sample(i)
            for __ in range(k-1):  # numpy is supported on positive integers
                # Update the particle
                if push:
                    pos[i] += 1
                    while i > 0 and pos[i] == pos[i-1]:
                        i -= 1
                        pos[i] += 1
                else:
                    if i == 0 or pos[i] != pos[i-1] - 1:
                        pos[i] += 1
    return pos

def modified_geometric(a, p, x, start=0):
    """
    Sample from the modified geometric distribution.
    """
    k = start
    al = RR(a(k))
    while random() < (al + p)*x / (1 + al*x):
        k += 1
        al = a(k)
    return k - start

def get_distribution(a, p, x, maxval, start=0):
    ret = [RR((1-p*x) / (1+a(start)*x))]
    for k in range(start, maxval):
        ret.append( ret[-1] * RR((a(k) + p)*x / (1 + a(k+1)*x)) )
    return ret

def plot_sample(sample):
    ell = len(sample)
    P = line2d([(-ell, ell), (0,0), (sample[0], sample[0])], aspect_ratio=1, color='black')
    P += scatter_plot([(val, 2*i+val) for i,val in enumerate(sample)], marker='.', edgecolor=None, facecolor='blue')
    P.tick_label_color('white')  # to hide them
    return P
\end{lstlisting}

\bibliographystyle{alpha}
\bibliography{grothendiecks}{}
\end{document}